\documentclass[11pt,leqno]{article}
\usepackage{amssymb,amsfonts}
\usepackage{amsmath,amsthm,amsxtra}
\usepackage{epsfig}
\usepackage{color}
\usepackage{verbatim}

\newcommand{\st}[1]{\ensuremath{^{\scriptstyle \textrm{#1}}}}

\newcommand\bigcheck[1]{#1 \raise1ex\hbox{$\hspace{-1ex}{}^\vee$}}
\newcommand\sucheck[1]{#1 \raise0.5ex\hbox{$\hspace{-1ex}{}^\vee$}}

\newcommand{\alphaparenlist}{
  \renewcommand{\theenumi}{\alph{enumi}}%
  \renewcommand{\labelenumi}{(\theenumi)}%
}

\newcommand{\romanparenlist}{
  \renewcommand{\theenumi}{\roman{enumi}}%
  \renewcommand{\labelenumi}{(\theenumi)}%
}

\newcommand{\ad}{\mathop{\rm ad}\,}
\newcommand{\add}{{\rm add}}

\newcommand{\ch}{{\rm ch}}

\newcommand{\corank}{\mathop{\rm corank \, }}

\newcommand{\End}{\mathop{\rm End }}

\newcommand{\gl}{g\ell}

\renewcommand{\Im}{\mathop{\rm Im  \, }}

\newcommand{\Res}{\mathop{\rm Res  \, }}
\newcommand{\rank}{\rm rank \, }

\renewcommand{\sl}{s\ell}
\newcommand{\str}{{\rm str}}
\renewcommand{\span}{{\rm span}}

\newcommand{\sdim}{\mathop{\rm sdim \, }}
\newcommand{\sign}{\rm sign \, }

\newcommand{\tr}{\rm tr \, }
\newcommand{\tw}{\rm tw \, }
\newcommand{\tpm}{\tilde{\Phi}^{[m]}}

\newcommand{\bo}{\bar{1}}
\newcommand{\bz}{\bar{0}}

\renewcommand{\Re}{\mathop{\rm Re  \, }}

\newcommand{\bracm}{{{[m]}}}

\newcommand{\A}{\mathcal{A}}

\newcommand{\CC}{\mathbb{C}}

\newcommand{\NN}{\mathbb{N}}
\newcommand{\QQ}{\mathbb{Q}}
\newcommand{\RR}{\mathbb{R}}
\newcommand{\ZZ}{\mathbb{Z}}

\newcommand{\fb}{\mathfrak{b}}

\newcommand{\fg}{\mathfrak{g}}
\newcommand{\fh}{\mathfrak{h}}

\newcommand{\fn}{\mathfrak{n}}

\newcommand{\rmv}{\text{v}}

\newcommand\bl{(\, . \, | \, . \, )}

\newcommand{\NRH}{\overset{\,\,\text{\tiny{N}}}{R}^{{\raisebox{-1ex}{\scriptsize{($\frac12 -\epsilon$)}}}}_{\frac12}}

\newcommand{\NRE}{\overset{\,\,\text{\tiny{N}}}{R}^{{\raisebox{-1ex}{\scriptsize{($\epsilon$)}}}}_{\frac12}}

 \newcommand{\NOE}{\overset{\,\,\text{\tiny{N}}}{R}^{{\raisebox{-1ex}{\scriptsize{($\epsilon$)}}}}_0}

\newcommand{\ON}{\overset{\text{\hspace{-.55ex}\tiny{N}}}{R}^{{\raisebox{-1ex}{\scriptsize{$\pm$}}}}}

\newcommand{\TON}{\overset{\text{\hspace{-.55ex}\tiny{N}}}{R}^{{\raisebox{-1ex}{\scriptsize{$\tw,\pm$}}}}}

\newcommand{\EON}{\overset{\,\,\text{\tiny{N}}}{R}^{{\raisebox{-1ex}{\scriptsize{($\epsilon$)}}}}_{\epsilon'}}

\newcommand{\fee}{\overset{\,\,\text{\tiny{4}}}{R}^{{\raisebox{-1ex}{\scriptsize{($\epsilon$)}}}}_{\epsilon'}}

\newcommand{\feee}{\overset{\,\,\text{\tiny{4}}}{R}^{{\raisebox{-1ex}{\scriptsize{($|\epsilon-\epsilon'|$)}}}}_{\epsilon'}}

\newcommand{\fpm}{\overset{\,\,\text{\tiny{4}}}{R}^{{\raisebox{-1ex}{\scriptsize{$\pm$}}}}}

\newcommand{\ftwpm}{\overset{\,\,\text{\tiny{4}}}{R}^{{\raisebox{-1ex}{\scriptsize{\tw, $\pm$}}}}}

\newcommand{\ftwp}{\overset{\,\,\text{\tiny{4}}}{R}^{{\raisebox{-1ex}{\scriptsize{\tw \!\!\!\!\!$+$}}}}}

\newcommand{\fRh}{\overset{\,\,\text{\tiny{4}}}{R}^{{\raisebox{-1ex}{\scriptsize{$\frac12$}}}}}

\newcommand{\fReh}{\overset{\,\,\text{\tiny{4}}}{R}^{{\raisebox{-1ex}{\scriptsize{$(\epsilon)$}}}}_{\frac12}}

\newcommand{\fRez}{\overset{\,\,\text{\tiny{4}}}{R}^{{\raisebox{-1ex}{\scriptsize{($\epsilon$)}}}}_0}

\newcommand{\NEO}{\overset{\,\,\text{\tiny{N}}}{R}^{{\raisebox{-1ex}{\scriptsize{($\epsilon'$)}}}}_{\epsilon}}

\newcommand{\OR}{\overset{\,\,\text{\tiny{2}}}{R}^{{\raisebox{-1ex}{\scriptsize{($\epsilon$)}}}}_{\epsilon'}}

\newcommand{\ORO}{\overset{\,\,\text{\tiny{2}}}{R}^{{\raisebox{-1ex}{\scriptsize{($\epsilon$)}}}}_0}

 \newcommand{\OX}{\overset{\,\,\text{\tiny{2}}}{R}^{{\raisebox{-1ex}{\scriptsize{($\epsilon'$)}}}}_\epsilon}

\newcommand{\ORH}{\overset{\,\,\text{\tiny{2}}}{R}^{{\raisebox{-1ex}{\scriptsize{($\epsilon$)}}}}_{\frac12}}

 \newcommand{\ORMH}{\overset{\,\,\text{\tiny{2}}}{R}^{{\raisebox{-1ex}{\scriptsize{($\frac12-\epsilon$)}}}}_{\frac12}}

\renewcommand{\tilde}{\widetilde}
\renewcommand{\hat}{\widehat}

\makeatletter
\renewcommand\section{\@startsection {section}{1}{\z@}%
                                   {-3.5ex \@plus -1ex \@minus -.2ex}%
                                   {2.3ex \@plus.2ex}%
                                   {\normalfont\large\bfseries}}
\renewcommand\subsection{\@startsection{subsection}{2}{\z@}%
                                     {-3.25ex\@plus -1ex \@minus -.2ex}%
                                     {0ex \@plus .0ex}%
                                     {\normalfont\normalsize\bfseries}}

\@addtoreset{equation}{section}
\makeatother

\newtheorem{theorem}{Theorem}[section]
\newtheorem{definition}[theorem]{Definition}
\newtheorem{lemma}[theorem]{Lemma}
\newtheorem{corollary}[theorem]{Corollary}
\newtheorem{proposition}[theorem]{Proposition}
\newtheorem{conjecture}[theorem]{Conjecture}
\newtheorem*{lemma*}{Lemma}

\theoremstyle{remark}
\newtheorem{remark}[theorem]{Remark}
\newtheorem{example}[theorem]{Example}

\makeatletter
\def\@maketitle{\newpage
 \null
 \vskip 2em
 \begin{center}%
 \vskip 3em
  {\Large\bf \@title \par}%
  \vskip 1.5em
  {\normalsize
   \lineskip .5em
   \begin{tabular}[t]{c}\@author
   \end{tabular}\par}%
  \vskip 2em

 \end{center}%
 \par
 \vskip 2.5em}
\makeatother

\renewcommand{\epsilon}{\varepsilon}

\definecolor{light}{gray}{.9}

\begin{document}

\title{Representations of affine superalgebras and mock theta functions}

\author{ Victor G. Kac and Minoru Wakimoto }


\author{Victor G.~Kac \thanks{Department of Mathematics, M.I.T, 
Cambridge, MA 02139, USA  kac@math.mit.edu Supported in part by an NSF
grant and Simons Fellowship} 
and Minoru Wakimoto  \thanks{~~wakimoto@r6.dion.ne.jp~~~~
Supported in part by Department of Mathematics, M.I.T.}}

\maketitle

\begin{center}{To Evgenii  Borisovich Dynkin on his 90th  birthday}\end{center}

\begin{abstract}
  We show that the normalized supercharacters of principal admissible
  modules over the affine Lie superalgebra
  $\hat{\sl}_{2|1}$ (resp. $\hat{p\sl}_{2|2}$) can
  be modified, using Zwegers' real analytic corrections, to form
  a modular (resp. $S$-) invariant family of functions.  Applying
  the quantum Hamiltonian reduction, this leads to a new family
  of positive energy modules over the $N=2$ (resp. $N=4$)
  superconformal algebras with central charge $3\left(
    1-\frac{2m+2}{M}\right)$, where $m \in \ZZ_{\geq 0}$, $M\in \ZZ_{\geq 2}$,
    $\gcd(2m+2,M)=1$ if $m>0$ 
    (resp. $6\left( \frac{m}{M}-1 \right)$, where $m \in \ZZ_{\geq 1}, 
   M\in \ZZ_{\geq 2}$,   $\gcd(2m,M)=1$ if $m>1$), whose modified characters
   and  supercharacters form a  modular invariant family. 
\end{abstract}

\setcounter{section}{-1}
\section{Introduction}
\label{sec:intro}

Modular invariance of characters of affine Lie algebras have been
playing an important role in their representation theory and
applications to physics (see \cite{K2} and references there).

Recall that an {\em affine Lie algebra } $\hat{\fg}$, associated
to a simple finite-dimensional Lie algebra $\fg$ over $\CC$ endowed
with a suitably normalized invariant symmetric bilinear form
$\bl$, is the infinite-dimensional Lie algebra over $\CC$:
\begin{equation}
  \label{eq:0.1}
  \hat{\fg} = \fg [t,t^{-1}]\oplus \CC K \oplus \CC d \,,
\end{equation}
with the following commutation relations $(a,b \in \fg,\,\, m,n \in\ZZ)$:
\begin{equation}
  \label{eq:0.2}
  [at^m,bt^n] = [a,b]t^{m+n}+m\delta_{m,-n} (a|b)K\,,\,
  [d,at^m]=m \,at^m\,,\, [K,\hat{\fg}]=0 \,.
\end{equation}
We identify $\fg$ with the subalgebra $1 \otimes \fg$.  The
bilinear form $\bl$ extends from $\fg$ to a non-degenerate
symmetric invariant bilinear form on $\hat{\fg}$ by:
\begin{equation}
  \label{eq:0.3}
  (at^m | bt^n) = \delta_{m,-n} (a|b)\,,\, (\fg [t,t^{-1}]| \CC
  K + \CC d)=0\,,\, (K|K)= (d|d)=0 \,, \, (K|d)=1\,.
\end{equation}
Choosing a Cartan subalgebra $\fh$ of $\fg$, one defines the
corresponding {\em Cartan subalgebra} of $\hat{\fg}$:
\begin{equation}
  \label{eq:0.4}
  \hat{\fh} = \CC d \oplus \fh \oplus \CC K \,.
\end{equation}
The restriction of the bilinear form $\bl$ from $\hat{\fg}$ to
$\hat{\fh}$ is non-degenerate, hence we shall identify
$\hat{\fh}$ with its dual $\hat{\fh}^*$ via this form.

One uses the following coordinates on $\hat{\fh}$:
\begin{equation}
  \label{eq:0.5}
  \hat{\fh}\ni h = 2\pi i (-\tau d + z + tK)\,,\,\hbox{\,\,where\,\,}\,
   \tau , t \in \CC\,,\, z \in \fh\,.
\end{equation}

Choosing a Borel subalgebra $\fb=\fh \oplus \fn_+$ of $\fg$
containing $\fh$, where $\fn_+$ is a maximal nilpotent subalgebra of $\fg$,
we define the corresponding Borel subalgebra of
$\hat{\fg}$:
\begin{equation*}
  \hat{\fb} = \hat{\fh} \oplus \fn_+ \oplus (\oplus_{n>0}\fg t^n)\,.
\end{equation*}
Given $\Lambda \in \hat{\fh}^*$, one extends it to a linear function
on $\hat{\fb}$ by zero on all other summands,
and defines the {\em highest weight
  module} $L(\Lambda)$ over $\hat{\fg}$ as the irreducible module,
which admits an eigenvector of $\hat{\fb}$ with weight $\Lambda $.
%
%
Since $K$ is a central element of $\hat{\fg}$, it is represented on 
$L(\Lambda)$ 
by a scalar $\Lambda(K)$, 
called the {\em level} of $L(\Lambda)$ (and of $\Lambda$).  

A $\hat{\fg}$-module $L(\Lambda)$ is called {\em integrable} if
any nilpotent element of $\hat{\fg}$ is represented by a locally
nilpotent operator (hence this module can be ``integrated'' to a 
representation of the group,
associated to $\hat{\fg}$).    It is well known \cite{K2} that a
$\hat{\fg}$-module $L(\Lambda)$ is integrable iff for all simple
roots $\alpha_1,\ldots,\alpha_\ell$ and the highest root $\theta$
the numbers
\begin{equation*}
  2 (\Lambda |\alpha_i)/(\alpha_i|\alpha_i)\,,\, 
  i=1,\ldots,\ell, \hbox{\,\, and \,\,} 
  2 (\Lambda|K-\theta)/(\theta |\theta)
\end{equation*}
are non-negative integers.  It is easy to deduce that if the
bilinear form on $\fg$ is normalized by the condition $(\theta
|\theta)=2$, then the level $(\Lambda |K)$ is a non-negative
integer and $(\Lambda|\theta) \leq (\Lambda |K)$.

The {\em character} of $L(\Lambda)$ is defined as the following
series, corresponding to the weight space decomposition with
respect to $\hat{\fh}$, cf. (\ref{eq:0.5}):
\begin{equation*}
  \ch_{L(\Lambda)} (\tau,z,t)=\tr_{L(\Lambda)}
  e^{2\pi i (-\tau d +z+tK)}\, .
\end{equation*}
 It is known \cite{KP}, \cite{K2} that for an integrable
 $L(\Lambda)$ this series converges in the domain
 \begin{equation}
   \label{eq:0.6}
   X=\{ h \in \hat{\fh}|\,\,\Re (h|K)>0\}
       = \{ (\tau,z,t)|\,\,\Im \tau >0 \}
 \end{equation}
to a holomorphic function.

Note that, as a $\hat{\fg}'=\fg [t,t^{-1}]\oplus \CC K$-module, $L
(\Lambda)$ remains irreducible, and it is unchanged if we replace 
$\Lambda$ by $\Lambda
+aK$, $a\in \CC$, and the character of the $\hat{\fg}$-module
gets multiplied by $q^a$.  Here and further $q=e^{2\pi i \tau}$.
Note also that the set of highest weights $\Lambda$ of level~$K$
of integrable $\hat{\fg}$-modules $L(\Lambda)$ is finite $\mod
\CC K$.  We denote this finite set by $P^k_+$.

An important property of integrable $\hat{\fg}$-modules is
modular invariance of its normalized characters, discovered in
\cite{KP}.  Recall that the {\em normalized character}
$\ch_\Lambda$ is defined as
\begin{equation*}
  \ch_\Lambda (\tau,z,t)=q^{m_\Lambda} \ch_{L(\Lambda)}(\tau,z,t)\,,
\end{equation*}
where $m_\Lambda \in \QQ$ is the ``modular anomaly'' (see
formula~(\ref{eq:4.1})).  Note that
$\ch_{\Lambda+aK}=\ch_\Lambda\,,\, a \in \CC$.  Recall the action
of $SL_2 (\RR)$ in the domain $X$ in coordinates (\ref{eq:0.5}):
\begin{equation}
  \label{eq:0.7}
  \binom{a\,\, b}{c \,\, d} \cdot (\tau,z,t) = \left(
    \frac{a\tau +b}{c\tau +d}\,,\, \frac{z}{c\tau +d}\,,\,
    t-\frac{c(z|z)}{2(c\tau +d)}\right) \,.
\end{equation}
The modular invariance of normalized characters of integrable
$\hat{\fg}$-modules means that the $\CC$-span of the finite set
$\{ \ch_\Lambda |\,\, \Lambda \in P^k_+ \}$ is $SL_2 (\ZZ)$-invariant
(for the action~(\ref{eq:0.7})).

The proof of modular invariance of normalized characters of
integrable modules $L(\Lambda)$ relies on the Weyl--Kac character
formula
\begin{equation}
  \label{eq:0.8}
  \hat{R}\ch_{L(\Lambda)}=\sum_{w\in \hat{W}}(\det w)
  \,w(e^{\Lambda +\hat{\rho}})\,,
\end{equation}
where $\hat{R}$ is the affine Weyl denominator, $\hat{W}$ is the affine
Weyl group, $\hat{\rho}$ is the affine Weyl vector (see \cite{K2}
for details).  One has \cite{K2}:
\begin{equation*}
  \hat{W} = W \ltimes \{t_\alpha |\,\, \alpha \in L \}\,,
\end{equation*}
where $W$ is the Weyl group of $\fg$, $L\subset \fh$ is the
coroot lattice, and $t_\alpha \in \End \hat{\fh}$ is defined by
(recall that $k=\Lambda (K)$ is the level):
\begin{equation}
  \label{eq:0.9}
    t_\alpha (\Lambda) = \Lambda + k\alpha - ((\Lambda
    |\alpha)+\frac{k}{2} (\alpha |\alpha))K\, .
\end{equation}
Using this, (\ref{eq:0.8}) can be rewritten, after multiplying
both sides by a suitable power of $q$, as
\begin{equation}
  \label{eq:0.10}
  q^{\frac{\dim\fg}{24}} \hat{R} \ch_\Lambda =\sum_{w\in W}{(\det  w)}
   \, w(\Theta_{ \Lambda +\hat{\rho}})\, .
\end{equation}
Here, for $\lambda \in \hat{\fh}$, such that $n=(\lambda |K)$ is a
positive integer, the theta function (= Jacobi form)
$\Theta_\lambda$ of degree $n$ is defined by
\begin{equation}
  \label{eq:0.11}
  \Theta_\lambda = q^{\frac{(\lambda | \lambda)}{2n}}
  \sum_{\alpha \in L} t_\alpha (e^{\lambda})\, .
\end{equation}
This series converges on $X$ to a holomorphic function, which in
coordinates (\ref{eq:0.5}) takes the usual form, going back to
Jacobi:
\begin{equation}
  \label{eq:0.12}
  \Theta_\lambda (\tau ,z,t)=e^{2\pi i nt}\sum_{\gamma \in
    \frac{\bar{\lambda}}{n}+L} q^{n\frac{(\gamma | \gamma)}{2}} e^{2\pi
    in (\gamma |z)}\, ,
\end{equation}
where $\bar{\lambda}$ denotes the orthogonal projection of
$\lambda$ on $\fh$.  Now modular invariance of Jacobi forms
(which we recall in the Appendix) easily implies the modular
invariance of the numerators of normalized characters of
integrable modules, and modular invariance of the normalized
denominator $q^{\frac{\dim \fg}{24}} \hat{R}$ easily follows from
the Jacobi triple product identity.

As we have discovered out in \cite{KW1}, \cite{KW2}, modular
invariance of normalized characters holds for a much larger class
of irreducible highest weight modules $L(\Lambda)$ over
$\hat{\fg}$, which we 
called {\em admissible} modules (and we
conjectured that these are all $L(\Lambda)$ with modular
invariance property, which we were able to verify only for $\fg =
\sl_2$).  Roughly speaking, a $\hat{\fg}$-module $L(\Lambda)$ is
called admissible, if the $\QQ$-span of coroots of $\hat{\fg}$
coincides with that of $\Lambda$-integral coroots, and with
respect to the corresponding affine Lie algebra
$\hat{\fg}_{\Lambda}$ the weight $\Lambda$ becomes integrable
after a shift by the Weyl vectors.

We showed in \cite{KW1} that a formula similar to (\ref{eq:0.8})
holds for $\ch_{L(\Lambda)}$ if $\Lambda$ is an admissible weight:
one just has to replace $\hat{W}$ by the subgroup, generated by
reflections with respect to non-isotropic $\Lambda$-integral
coroots.  It follows that the numerators of normalized
admissible characters are again expressed as linear combinations
of Jacobi forms, which again implies modular invariance of
normalized characters of admissible $\hat{\fg}$-modules.

Furthermore, using the quantum Hamiltonian reduction, one can
transfer the modular invariance property from the admissible
modules over affine Lie algebras to the ``minimal models'' of
$W$-algebras \cite{FKW}, \cite{KRW}, \cite{A2}, the simplest example
being the Virasoro algebra. (Note that the integrable modules are
``erased'' by the reduction!)

A natural question arises whether the theory of integrable and
admissible modules over affine Lie algebras $\hat{\fg}$ extends to
the case when $\fg$ is a finite-dimensional simple Lie
superalgebra.  Of course, we need to assume that $\fg$ carries  a
non-degenerate supersymmetric invariant bilinear form, and also
that its even part $\fg_{\bz}$ is a reductive Lie
algebra. According to the classification of \cite{K1}, a complete
list of such Lie superalgebras consists of the classical series
$\sl_{m|n} (m>n \geq 1)$, $p\sl_{n|n} (n \geq 2)$, $osp_{m|n} (m
\geq 1, n \geq 2 \hbox{\,\, even})$ and three exceptional
superalgebras.

Of all these Lie superalgebras, 
the above mentioned results extend without difficulty only for 
$\fg = osp_{1|n}$, in particular,
modular invariance property of normalized characters and supercharacters of
integrable and admissible modules still holds \cite{KW1}, and the
quantum Hamiltonian reduction in the case of $\fg = osp_{1|2}$
leads to modular invariance of characters of the Neveu--Schwarz
and Ramond superalgebras ($N=1$ superconformal algebras).

In general, for a Lie superalgebra $\fg$ in question, a
$\hat{\fg}$-module is integrable iff it is integrable with
respect to $\hat{\fg}_{\bz}$. However, such non-trivial $\hat{\fg}$-modules
$L(\Lambda)$ exist iff 
$\fg_{\bz}$ has only one simple component (which is
the case only when $\fg = \sl_{n|1}$,
$osp_{1|2n}$, or $osp_{2|2n}$).  
In all other cases one considers {\em partially integrable} modules, namely
those, for which integrability holds for the affine subalgebra,
associated to one of the simple
components of $\fg_{\bz}$.

Partially integrable $\hat{\fg}$-modules $L (\Lambda)$ in the
``super'' case have been classified in \cite{KW4}, but the
computation of their characters is a very difficult problem in
general.  However, in the special case of ``tame'' modules (see
Definition~\ref{def:3.5}) we found a conjectural formula for the characters
(see formula (\ref{eq:3.10})).  This formula has been proved in
all cases that are considered in the present paper (see
\cite{KW4}, \cite{S}, \cite{GK}).  Note also that in the
``super'' case one has to study supercharacters along with the
characters (when the trace is replaced by the supertrace), but
one can pass from one to the other without difficulty.

The formula for the supercharacter of a tame partially integrable
module $L(\Lambda)$ over an affine Lie superalgebra $\hat{\fg}$
differs little from formula (\ref{eq:0.10}) for the character of
an integrable module $L (\Lambda)$ over an affine Lie
algebra.  One just has to replace the theta function
$\Theta_{\Lambda+\hat{\rho}}$, defined in (\ref{eq:0.11}) (or
(\ref{eq:0.12})) by the mock theta function!  Given a finite subset $T
\subset \fh$ of pairwise orthogonal
vectors, which are also orthogonal to $\lambda$, the following
series converges to a meromorphic function on $X$, called a {\em mock
theta function} of degree~$n$:
  \begin{equation}
    \label{eq:0.13}
    \Theta_{\lambda,T} (\tau,z,t) = e^{2\pi int}
       \sum_{\gamma \in \frac{\bar{\lambda}}{n} +L}
         \frac{q^{n\frac{(\gamma |\gamma)}{2}} 
           e^{2\pi in (\gamma  |z)}}{\prod_{\beta \in T}
           \left(1-q^{-(\gamma|\beta)} e^{-2\pi i (\beta |z)}\right)}\,.
  \end{equation}
These kind of functions (when $\#T=1$ and $\rank L =1$) 
first appeared in the work
of Appell \cite{Ap} in the 1880s in his study of elliptic functions ``of
the third kind'', and also,  a few years later, in the work of Lerch
\cite{L}.  More than 100 years later these
functions made their way to the representation theory of affine Lie
superalgebras \cite{KW4}.

One of the simplest results of \cite{KW4} is the following formula for
the normalized supercharacter of the integrable $\hat{\sl}_{2|1}$-module
$L(d)$ (of level~1), obtained via the super boson-fermion correspondence:
\begin{equation}
  \label{eq:0.14} 
ch^-_d (\tau,z_1,z_2) = \eta (\tau)^{-3}\vartheta_{11} (\tau,z_1)
\vartheta_{11}(\tau, z_2) \mu (\tau, z_1, z_2)\, ,
\end{equation}
where $\eta (\tau)$ is the Dedekind  eta-function,
$\vartheta_{11} (\tau,z)$ is one of the standard Jacobi forms
(see the Appendix), and 
\begin{equation}
\label{eq:0.15}
\mu (\tau,z_1,z_2) = \frac{e^{\pi i
    z_1}}{\vartheta_{11}(\tau,z_2)} \sum_{n \in \ZZ} \frac{(-1)^n
q^{\frac12 (n^2+n)} e^{2\pi i nz_2}}{1-e^{2\pi i z_1}q^n}\,.
\end{equation}
The function $\mu (\tau,z_1,z_2)$, up to the factor
$\vartheta_{11} (\tau,z_1+z_2)$, is a difference of two simplest
mock theta functions (see (\ref{eq:5.16})), which we denote by
$\Phi^{[1]} (\tau,z_1,z_2,0)$ (see (\ref{eq:5.3}), (\ref{eq:5.4})).

It is the function $\mu (\tau,z_1,z_2)$ that plays a central role
in the work of Zwegers on mock theta functions \cite{Z}, which
has been a major advance in the understanding of 
Ramanujan's mock $\vartheta$-functions.  
Ramanujan defined a mock $\vartheta$-function as a function $f$ of the complex
variable~$q$, defined by a $q$-series of a particular type, which
converges for $|q|<1$ and satisfies the following conditions (see
\cite{Z}):
\begin{list}{}{}
\item (i) infinitely many roots of unity are exponential
  singularities,
\item (ii)  for every root  of unity $\xi$ there is a
  $\vartheta$-function $\vartheta_\xi (q)$, such that 
  $f(q)-\vartheta_\xi (q)$ is bounded as $q \to \xi$ radially, 
\item (iii)  there is no $\vartheta$-function that works for all
  $\xi$.
\end{list}
It turns out that the function $\mu (\tau,z_1,z_2)$ is the
prototype for a mock theta function in the sense that,
specializing the complex variables $z_1$ and $z_2$ to torsion
points (i.e. elements of $\QQ+\QQ\tau$), one gets mock $\vartheta$-functions
\cite{Z1}.

An important discovery of Zwegers is the real analytic function $R
(\tau,u)$, $\tau, u \in \CC$, $\Im \tau >0$, such that the
modified function
\begin{equation*}
  \tilde{\mu} (\tau,z_1,z_2) = \mu (\tau,z_1,z_2) + \frac{i}{2} R
  (\tau,z_1-z_2)
\end{equation*}
is a modular invariant function with nice elliptic transformation
properties (\cite{Z}, Theorem~1.11).  Furthermore, Zwegers
introduces real analytic functions $R_{m;\ell} (\tau,u)$, similar
to $R (\tau ,u)$ (they are related by (\ref{eq:5.15})), such that, 
adding to a rank $1$ mock theta function of
arbitrary degree $m>0$ (and $\# T=1$ in our terminology) a
suitable linear combination of rank $1$ Jacobi forms $\Theta_{m,\ell}$
as coefficients, he obtains a modular invariant real analytic
function (\cite{Z}, Proposition~3.5).  The latter functions are
used in the study of Ramanujan's mock theta functions (\cite{Z},
Chapter~4).

In our paper (Section \ref{sec:5}) we use the functions $R_{m+1;\ell}$ of 
Zwegers in
order to modify the normalized supercharacter of the
$\hat{\sl}_{2|1}$-module $L (md)$, where $m$ is a positive
integer.  The normalized supercharacter is given in this case by
the following formula:
  \begin{equation}
    \label{eq:0.16}
     \hat{R}^- \ch^-_{md} (\tau,z_1,z_2, t) = \Phi^{[m]}(\tau,z_1,z_2,t)\,,
  \end{equation}
where $\hat{R}^-$ is the superdenominator (see (\ref{eq:4.10}) for
its expression in terms of the Jacobi theta function
$\vartheta_{11}(\tau,z$)), and $\Phi^{[m]}$ is the following mock
theta function
  \begin{equation*}
\Phi^{[m]} (\tau, z_1,z_2,t) =  e^{2\pi i (m+1)t}  
 \sum_{j \in \ZZ}\left(\frac{e^{2\pi ij(m+1)(z_1+z_2)}q^{j^2 (m+1)}}
          {1-e^{2\pi i z_1}q^j} 
  - \frac{e^{-2\pi i j (m+1) (z_1+z_2)} q^{j^2(m+1)}}
      {1-e^{-2\pi i z_2} q^j} \right) \,.
  \end{equation*}
Following Zwegers' ideas, we introduce the real analytic modified
numerator
\begin{equation*}
  \tilde{\Phi}^{[m]} (\tau,z_1,z_2,t) = \Phi^{[m]}  (\tau,z_1,z_2,t)
    + \Phi^{[m]}_{\add} (\tau,z_1,z_2,t)\,,
\end{equation*}
where $\Phi^{[m]}_{\add}$ is 
a real analytic function, similar to Zwegers' correction in higher degree,
and prove the following modular transformation properties:
\begin{equation}
  \label{eq:0.17}
  \tilde{\Phi}^{[m]} \left( -\frac{1}{\tau} \,,\, \frac{z_1}{\tau}
    \,,\, \frac{z_2}{\tau}, t-\frac{z_1z_2}{\tau}\right) 
  = \tau \bar{\Phi}^{[m]} (\tau,z_1,z_2,t)\,,\,\,
    \Phi^{[m]} (\tau+1,z_1,z_2,t) = \Phi^{[m]} (\tau,z_1,z_2,t)\,,
\end{equation}
along with certain elliptic transformation properties (Theorem~\ref{th:5.9}
and Corollary~\ref{cor:5.10} ).  This establishes modular
invariance of the modified normalized supercharacter
$\tilde{\ch}^-_{md} =\tilde{\Phi}^{[m]} /\hat{R}^-$.

Next, in Sections~\ref{sec:6} and \ref{sec:7} we discuss modular
invariance properties of the modified normalized principal
admissible characters of $\hat{sl}_{2|1}$-modules, associated to
a compatible homomorphism of degree~$M$ (see Section~\ref{sec:3}
for the definition of these modules). The main result of
Section~\ref{sec:6} is the following modular transformation
formula (Theorem~\ref{th:6.5}(a)):
\begin{eqnarray}
  \label{eq:0.18}
\lefteqn{\hspace{-4.75in}  \tilde{\Phi}^{[m]} \left(- \frac{M}{\tau} \,,\, \frac{z_1}{\tau} \,,\,
      \frac{z_2}{\tau}\,,\, t-\frac{z_1z_2}{\tau M} \right)}\\
  =  \frac{\tau}{M} \sum_{j,k \in \ZZ /M \ZZ} q^{\frac{m+1}{M}jk}
    e^{\frac{2\pi i (m+1)}{M} (kz_1+jz_2)} \tilde{\Phi}^{[m]}
      (M\tau \,,\, z_1+j\tau \,,\, z_2+k\tau\, , \, t)\, ,\nonumber
\end{eqnarray}
provided that $\gcd (M,2m+2) =1$ if $m>0$.  We deduce that the
modified normalized principal admissible characters,
supercharacters, and their Ramond twisted analogues, form a modular
invariant family under the above conditions on $M$ and $m$
(Theorem~\ref{th:7.3}).

In Section~\ref{sec:8} we study the behavior under the modular
transformation $\displaystyle{S=\left(
  \begin{smallmatrix}
    0 & -1\\ 1 & \hspace{1ex}0
  \end{smallmatrix}\right)}$ of modified normalized characters of 
$\hat{A}_{1|1}$-modules $L (md)$, where $m$ is a
non-zero integer.  This is related to the $\hat{\sl}_{2|1}$ case,
using a simple connection between the numerators of
$\hat{A}_{1|1}$ and $\hat{\sl}_{2|1}$ normalized supercharacters
(for $m \geq 1$):
\begin{equation}
  \label{eq:0.19}
  \Phi^{A_{1|1}[m]} (\tau,z_1,z_2,t) = \frac{1}{2\pi i} \left(
  \frac{\partial}{\partial z_1} - \frac{\partial}{\partial z_2} \right)
   \Phi^{[m-1]} (\tau,z_1,z_2,t)\,.
\end{equation}
The $S$-transformation of modified normalized characters and
supercharacters, and their Ramond twisted analogues is given by
Theorem~\ref{th:8.4} for $m>0$ and Theorem~\ref{th:8.7} for
$m<0$.  Note that in either case we don't have $T=\left(
  \begin{smallmatrix}
    1 &1 \\0 & 1
  \end{smallmatrix}
\right)$-invariance, and that the process of modification is more complicated
than in the $\hat{\sl}_{2|1}$ case.

In Section~\ref{sec:9} (resp.~10) we study modular invariance of
the modified characters of modules, obtained 
from the principal admissible
$\hat{\sl}_{2|1}$ (resp.~$\hat{A}_{1|1}$)-modules by the quantum
Hamiltonian reduction (developed in \cite{KRW}, \cite{KW5},
\cite{KW6}).  

As a result we obtain in Section~\ref{sec:9} a modular invariant
family of $N=2$ modified characters, supercharacters, and their Ramond 
twisted analogs, of irreducible positive energy modules with central
charge $3 \left(1-\frac{2m+2}{M}\right)$, where 
$m\in \ZZ_{\geq 0}$, $M\in \ZZ_{\geq 2}$ and $\gcd(2m+2,M)=1$ if $m>0$.
(Theorem~\ref{th:9.2}).  If $m=0$ we obtain the famous $N=2$ unitary
discrete series, for which modular invariance holds without modification,
but for $m \geq 1$ we obtain some very interesting new positive
energy $N=2$ modules, which should be of great interest for the conformal
field theory.  For example, as shown in \cite{W}, the
fusion coefficients for the $N=2$ unitary discrete series are equal to 0 or~1. 
Remarkably, using the same method one can show that the same property holds for
arbitrary $m>0$, such that $\gcd (M,2m+2)=1$.

The quantum Hamiltonian reduction of principal admissible
$\hat{A}_{1|1}$-modules, studied in Section~\ref{sec:8}, produces
$N=4$ irreducible positive energy modules.  However we obtain
modular invariant  families of modified characters only for negative level, 
the central
charge being $6\left( \frac{m}{M} -1 \right)$, where 
$m\in \ZZ_{\geq 1}$, $M\in \ZZ_{\geq 2}$ and $\gcd(2m,M)=1$ if $m>1$
(Theorem~\ref{th:10.7}).

For the convenience of the reader, we provide in
Sections~\ref{sec:1}--\ref{sec:4} some necessary material on Lie
superalgebras and their highest weight modules, and in the Appendix
we give a brief review of some basic facts about Jacobi theta
functions, used throughout the paper.

Note that a general definition
of a mock modular form was given by Don Zagier (see \cite{DMZ}
for an introduction to the subject), and that connections of the
theory of mock modular forms to level~1 integrable
$\hat{\sl}_{n|1}$-characters, computed in \cite{KW4}, have been 
established in \cite{BF},
\cite{BO}, \cite{F}.

In our subsequent paper we will consider the remaining case of a rank $2$
simple finite-dimensional Lie superalgebra, $\fg=osp_{3|2}$, and the 
corresponding $N=3$ quantum Hamiltonian reduction.

The main results of the paper were reported by the authors at the
''Lie superalgebras'' conference in Rome in December 2012.

The authors worked on the paper during the second author's visit to M.I.T in 
the spring of 2012, and while both authors visited IHES, France, in the fall of
2012 and in the spring of 2013; the paper was completed while the first 
author was visiting IHES in the summer of 2013. We would like to thank these 
institutions for their hospitality.

We thank K. Bringmann, A. Folsom, K. Ono, and especially S.P. Zwegers,
for the very useful correspondence.

\section{Kac--Moody superalgebras and their highest weight modules}
\label{sec:1}

Let $I$ be a finite index set and let $A=(a_{ij})_{i,j \in I}$ be a
symmetric matrix over $\RR$.  A {\em realization} of the matrix
$A$ is a vector space $\fh_{\RR}$ of dimension $|I| +  \corank\, A$ over
$\RR$ with a linearly independent set of vectors $\{ h_i \}_{i \in
  I}$, and a linearly independent set of linear functions 
$\Pi = \{ \alpha_i\}_{i \in I}$, satisfying 
\begin{equation}
  \label{eq:1.1}
  \alpha_i (h_j) = a_{ij} \, , \qquad j \in I \, .
\end{equation}
The elements $\alpha_i \in \fh^* $, $i
\in I$, are called {\em simple roots}.

Let $\fh=\CC \otimes _{\RR}\fh_{\RR}$.
Given a subset $I_{\bo}\subset I$, one defines the Kac--Moody
superalgebra  $\fg (A,I_{\bo})$ as follows \cite{K1}.
First, denote by $\tilde{\fg} (A, I_{\bo})$ the Lie superalgebra
on generators $e_i, f_i,\, i \in I,$ and $\fh$, the  generators 
$e_i,f_i$ for $i \in I_{\bo}$ being  odd and
all the other generators being even, and the following {\em 
Chevalley relations} 
$(i,j\in I\, , \, h \in \fh):$
\begin{eqnarray*}
  [\fh , \fh] = 0, [e_i, f_j] = \delta_{ij} h_i\, , \, [h,e_i]
     = \alpha_i (h) e_i \, , \, [h,f_i] = -\alpha_i (h) f_i\, .
\end{eqnarray*}
The Lie superalgebra $\tilde{\fg} (A, I_{\bo})$ has a unique
maximal ideal $\tilde{J}$ among those intersecting the subspace
$\fh$ trivially, and we let
\begin{equation*}
  \fg (A,I_{\bo}) = \tilde{\fg} (A, I_{\bo})/ \tilde{J}\, .
\end{equation*}
%
Fix a non-degenerate symmetric bilinear form
$(\, . \, | \, . \, )$ on $\fh_{\RR}$, such that 
$(h_i|h_j)=a_{ij}, \,i,j\in I$, and extend to $\fh$ by bilinearity.

Of course, $\fg (A,I_{\bo})$ is a Lie algebra iff $I_{\bo} =
\emptyset$.  In this case it is isomorphic to  a simple Lie
algebra $\fg $ if $A$ is the Cartan matrix of $\fg$, and to the
corresponding affine Lie algebra $\hat{\fg}$ if $A$ is the
extended Cartan matrix of $\fg$.

\begin{proposition}  \label{prop:1.1}(cf. [K2], Chapter~2). The Lie superalgebra $\fg (A, I_{\bo})$ carries a unique bilinear
form  
$(\, . \, | \, . \, )$
, extending that on $\fh$, which is
supersymmetric (i.e. $(a|b) =(-1)^{p(a)p(b)} (b|a)$) and
invariant (i.e.  $([a,b]|c)= (a|[b,c])$).  This bilinear form is
non-degenerate.  

\end{proposition}

The abelian subalgebra $\fh$ is called the Cartan subalgebra of
the Kac--Moody algebra $\fg (A,I_{\bo})$.  As usual, we have the
{\em root space decomposition}:
\begin{equation}
  \label{eq:1.2}
  \fg (A,I_{\bo}) = \fh \oplus (\oplus_{\alpha \in \Delta} \fg_\alpha)\,,
\end{equation}
where 
$\fg_\alpha = \{ a \in \fg (A,I_{\bo}) |\,\, [h,a] = \alpha
(h) a$ for all $ h \in \fh \}$ and $ \Delta = \{\alpha \in
\fh^* |\,\, \alpha \neq 0 $ and $\fg_\alpha \neq 0\}$ is the {\em set of
  roots}.

Denoting by $\fn_+$ (resp. $   \fn_-$) the subalgebra of $\fg
(A,I_{\bo})$, generated by all the $e_i$ (resp. $f_i$), we have
the {\em triangular decomposition}:
\begin{equation}
  \label{eq:1.3}
  \fg (A, I_{\bo}) = \fn_- \oplus \fh \oplus \, \fn_+ \, .
\end{equation}

Let $Q = \ZZ \Pi \subset \fh^*$ be the {\em root lattice}, and
let $Q_+ = \ZZ_{\geq 0} \Pi$.  Let $\Delta_+ = \Delta \cap Q_+$ be the
set of {\em positive roots}.  Then $\fn_\pm = \oplus_{\alpha \in
  \Delta_+} \fg_{\pm \alpha}$.

Since $\fg (A,I_{\bo})$ has the anti-involution which exchanges
$e_i$ and $f_i$ and fixes $\fh$ pointwise, we conclude that $\dim
\fg_\alpha = \dim \fg_{-\alpha}$.

Since the bilinear form $(\, . \, | \, . \, )$
to $\fh$ is nondegenerate, we may (and will) identify $\fh$ with
$\fh^*$.  For a non-isotropic root $\alpha \in \Delta$, we let
\begin{equation}
  \label{eq:1.4}
  \alpha^\vee = 2 \alpha / (\alpha |\alpha)\, ,
\end{equation}
and define the {\em reflection} $r_\alpha \in GL (\fh^*)$ by
\begin{equation}
  \label{eq:1.5}
   r_\alpha (\lambda) = \lambda - (\lambda |\alpha^\vee)\, \alpha \, ,
   \quad \lambda \in \fh^*\, .
\end{equation}

If $\alpha \in \Delta_+$, then $\alpha = \sum_{i \in I}
k_i\alpha_i$, where $k_i \in \ZZ_{\geq 0}$, and its parity is $p
(\alpha) =\sum_i k_i p (\alpha_i) \mod 2$, which is the same as
the parity of $\fg_\alpha$.  Denote by $\Delta_{\bar0}$ and
$\Delta_{\bo}$ the subsets of $\Delta$, consisting of even and
odd roots respectively.

\begin{proposition}  \label{prop:1.2}  (cf. \cite{K2}, Chapter 3).
   Let $\alpha \in \Delta$ be an even non-isotropic root, and suppose
that $\ad \fg_\alpha$ is locally nilpotent on $\fg (A,I_{\bo})$.
Then $r_\alpha (\Delta ) \subset \Delta$.

\end{proposition}

A root $\alpha$, satisfying conditions of Proposition~\ref{prop:1.2}, is
called {\em integrable}.  The group, generated by all reflections
$r_{\alpha}$, where $\alpha$ is integrable,
is called the {\em Weyl group} of the
Kac--Moody superalgebra $\fg (A,I_{\bo})$, and is denoted by $W
(\subset GL (\fh^*))$.  By Proposition~\ref{prop:1.2}, 
\begin{equation}
  \label{eq:1.6}
  W \Delta = \Delta \, .
\end{equation}

For $w\in W$ let $w=
r_{\gamma_1}...r_{\gamma_s}$ be a decomposition of $w$ in a product
of $s$ reflections with respect to integrable 
roots, and let $s_-$ be the number of those of them, for which the half 
is not a root.
Define 
$\epsilon_+(w)=(-1)^s$ and $\epsilon_-(w)=(-1)^{s_-}$.
(For $\fg = s\ell_{m|n}$, $m>n$, and $g\ell_{m|m}$ one has: 
$\epsilon_- (w) = \epsilon_+ (w)$.)

Let $\rho_{\bar{0}}$
(resp. $\rho_{\bar{1}}$) be the half of the sum of positive even
(resp. odd) roots. The element $\rho=\rho_{\bar{0}}-\rho_{\bar{1}}$ 
is called the Weyl vector. One has \cite{K1}:
\begin{equation}
  \label{eq:1.7}
  (\rho | \alpha_i) = \frac12 (\alpha_i |\alpha_i) (= \frac12
  a_{ii})\, , \quad i \in I\, .
\end{equation}
%

Let $V=V_{\bar{0}} \oplus V_{\bo}$ be a vector superspace (i.e. a
  vector space, decomposed in a direct sum of subspaces
  $V_{\bar{0}}$ and $V_{\bo}$, called even and odd
  respectively).  The associative algebra $\End V$ has the
  corresponding $\ZZ /2 \ZZ$-grading:
  \begin{equation*}
    \End V = (\End V)_{\bar{0}} \oplus (\End V)_{\bo}\, , \, 
   \, \ZZ /2\ZZ = \{ \bar{0}, \bar{1}\}, 
  \end{equation*}
where
\begin{equation*}
  (\End V)_\alpha = \{ a \in \End V |\, a V_\beta \subset V_{\alpha
  + \beta}\}\, , \, \alpha , \beta \in \ZZ/2 \ZZ \, .
\end{equation*}
One denotes by $\gl_V$ the vector superspace $\End V$ with the
bracket
\begin{equation*}
      [a,b] = ab - (-1)^{p (a) p (b)} ba \, .
 \end{equation*}

A {\em module} $V$ over a Lie superalgebra $\fg$ is a
homomorphism of $\fg$ to the Lie superalgebra $\gl_V$.   If $\dim
V < \infty$, one defines the {\em supertrace} on $\End V$ by
\begin{equation}\label{eq:1.8}
   \str\, a = \tr F a\, , \, \hbox{\,\, where \,\,} F |_{V_\alpha}
     =   (-1)^\alpha\, , \quad \alpha \in \ZZ/2\ZZ\, ,
\end{equation}
and the superdimension of $V$ by $\sdim V = \str\,  I_V = \dim
V_{\bar{0}} -\dim V_{\bo}$\, .  In this case $\gl_V$ contains the
subalgebra $\sl_V = \{ a \in \gl_V |$    $\str\, a\,  =0 \}$.

For each $\Lambda \in \fh^*$  one defines an {\em irreducible
  highest weight module} $L (\Lambda)$ 
over the Kac--Moody
superalgebra $\fg (A,I_{\bo})$ as the (unique) irreducible $\fg
(A,I_{\bo})$-module for which there exists an even non-zero
vector $\rmv_\Lambda$, such that 
\begin{eqnarray*}
  h\rmv_\Lambda = \Lambda (h) \rmv_{\Lambda} \hbox{\,\, for all \,\,} h \in
  \fh \, , \quad \fn_+ \rmv_\Lambda =0\, .
\end{eqnarray*}
One has the {\em weight space decomposition} with respect to $\fh$:
\begin{equation}
  \label{eq:1.9}
  L (\Lambda) = \oplus_{\lambda \in \fh^*} L (\Lambda)_\lambda\,,\quad
  \hbox{\,\, where \,\,} L(\Lambda)_\lambda
 = \{ \rmv \in L (\Lambda) |\, h \rmv = \lambda (h) \rmv\, ,
  \quad  h \in \fh \}.
\end{equation}
Since $L (\Lambda) = U (\fn_-)\rmv_\Lambda$, it follows that $\dim L
(\Lambda)_\Lambda =1$ and $\dim L (\Lambda)_\lambda < \infty$.

One then defines the {\em character}
\begin{equation*}
  \ch^+_{L (\Lambda)} =\sum_{\lambda \in \fh^*} ( \dim L (\Lambda)_\lambda)e^\lambda
\end{equation*}
and the {\em supercharacter}
\begin{equation*}
 \ch^-_{L (\Lambda)} =\sum_{\lambda \in \fh^*} (\sdim L (\Lambda)_\lambda)e^\lambda\,.
\end{equation*}
Note that for $h \in \fh$ one has:
\begin{equation*}
  \ch^+_{L (\Lambda)}(h)  =\tr_{L(\Lambda)}\, e^h \, ,\,\, \, 
 \ch^-_{L (\Lambda)}(h) =  \str_{L(\Lambda)} \,e^h \,.
\end{equation*}

An integrable root $\alpha$ is called $\Lambda$-{\em integrable} if
$\fg_\alpha$ and $\fg_{-\alpha}$ are locally nilpotent on $L
(\Lambda)$ (note that $\fg_\alpha$ with $\alpha \in \Delta_+$ is
always locally nilpotent).

\begin{proposition} (cf. \cite{K2}, Chapter 3).
  \label{prop:1.3}
Let $\alpha$ be a $\Lambda$-integrable root.  Then $\dim L
(\Lambda)_\lambda =\dim L (\Lambda)_{r_\alpha (\lambda)}$, and
the same holds for $\sdim$.  Equivalently, $\ch^\pm_{L (\Lambda)}$ 
are $r_\alpha$-invariant.
\end{proposition}

\begin{proposition} (cf. \cite{K2}, Chapter 10).
  \label{prop:1.4}
The series $\ch^\pm_{L (\Lambda)} $ 
converge to a holomorphic function in a convex domain, containing the domain
$Y:=\{ h \in \fh|\,\, \Re \alpha_i (h) >0 \, , \, i \in I \}$.
\end{proposition}

\begin{remark}   
\label{rem:1.5}
Let $\fg' (A,I_{\bo})$ denote the derived Lie superalgebra of
$\fg (A,I_{\bar{1}})$.  Then we have:
\begin{equation*}
  \fg (A,I_{\bar{1}}) = \fh + \fg' (A,I_{\bar{1}})\, ,\quad
  \fg'  (A,I_{\bar{1}}) \cap \fh = \fh' : = \span \{ h_i |\, i \in I \}\, .
\end{equation*}
It follows that $L (\Lambda)$ remains irreducible when restricted
to $\fg' (A,I_{\bar{1}})$, and that $\ch^\pm_{L (\Lambda)}$ 
depend essentially only on $\Lambda |_{\fh'}$, namely:
\begin{equation*}
  \ch^\pm_{L (\Lambda)} = e^{\Lambda -\Lambda'} \ch^\pm_{L (\Lambda')} 
    \hbox{\,\, if \,\,} \Lambda |_{\fh'} = \Lambda' |_{\fh'}\, .
\end{equation*}
%
\end{remark}

\begin{definition}
  \label{def:1.6}

The $\fg (A, I_{\bar{1}})$-module $L (\Lambda)$ is called
  integrable  if any integrable
root $\alpha$ is $\Lambda$-integrable.

\end{definition}

If $A$ is the Cartan matrix of a Kac--Moody algebra $\fg(A)$, then
the integrable $\fg(A)$-modules $L(\Lambda)$ are
precisely the ones that can  be ``integrated'' to
the corresponding Kac--Moody group.

If $\fg (A,I_{\bar{1}})$ is a finite-dimensional Kac--Moody
superalgebra, then it is easy to show that  integrable
modules $L (\Lambda)$ are precisely all finite-dimensional
irreducible $\fg (A, I_{\bar{1}})$-modules.

If $\fg (A,I_{\bar{1}})$ is an affine Kac--Moody superalgebra
with $I_{\bo} \neq \emptyset$, then non $1$-dimensional
integrable modules $L (\Lambda)$ exist in only a few cases,
so that it is natural to consider ``partially'' integrable
modules instead.  They are classified in \cite{KW4}, and will be
discussed in \S3.

\begin{proposition}
  \label{prop:1.7}
Suppose that a Lie superalgebra carries two structures  of a
Kac--Moody superalgebra with the same set of even positive roots.
Then

\alphaparenlist
\begin{enumerate}
\item 
One of these Kac--Moody superalgebra structures can be obtained
from the other one by a sequence of odd reflections.

\item 
If $L (\Lambda)$ is an irreducible highest weight module with
highest weight vector $\rmv_\Lambda$ with respect to the first
structure, then it is an irreducible highest weight module with
respect to the second structure.  Explicitly, if the second
structure is obtained from the first one by an odd reflection
$r_i$, then the new highest weight vector is $f_i \rmv_\Lambda$ if
$(\Lambda |\alpha_i) \neq 0$, and is $\rmv_\Lambda$ if $(\Lambda
|\alpha_i) =0$.

\item 
If $L(\Lambda)$ is integrable  with
respect to the first Kac--Moody superalgebra structure, then it is also true
for the second one.

\end{enumerate}

\end{proposition}

\begin{proof}
  (a) is proved in the same way as Proposition~5.9 from
  \cite{K2}, by making use of odd reflections (described e.g. in
  \cite{KW4}).  (b) is Lemma~1.4 from \cite{KW4} and (c) follows
  from (a) and (b).
\end{proof}



\section{Examples of finite-dimensional and affine Kac--Moody superalgebras}
\label{sec:2}

Let$V=V_{\bar0} \oplus V_{\bo}$ be a finite-dimensional superspace
over $\CC$, let $m=\dim V_{\bar0}$, $n=\dim V_{\bo}$.  Then one
denotes $\gl_{m|n} = \gl_V$, $\sl_{m|n} = \sl_V$.  Since
$s\ell_{m|n} \simeq s\ell_{n|m}$, we shall always assume that $m
\geq n$.

Let $\fg =\sl_{m|n}$ if $m > n$, and $\fg = \gl_{n|n}$ if
$m=n$.  The Lie superalgebra $\fg$ carries several structures of
a Kac--Moody superalgebra, described below (it is easy to show that
there are no others).

Fix a basis $\rmv_1, \ldots , \rmv_{m+n}$ of $V$, such that each
vector $\rmv_i$ lies either in $V_{\bar0}$ or in $V_{\bo}$; we
write $p (\rmv_i)=\bar0$ or $\bar1$ respectively.  Let $I= \{
1,2, \ldots , m+n-1 \}$ and let $I_{\bar{1}} = \{ i\in I |\,\,p (\rmv_i)
\neq p (\rmv_{i+1})\}$.

The choice of basis of $V$ identifies $\End V$ with the space of
$(m+n) \times (m+n)$   matrices.  Let $\{ E_{ij}\}$ denote the
standard basis of this space, and let $\epsilon_i$ denote the
linear function on the space of
diagonal matrices, which picks
out the $i$\st{th} diagonal entry.

Let $\fh$ be the space of all diagonal matrices in $\fg$, and let

\begin{equation}
  \label{eq:2.1}
  h_{ij} = (-1)^{p(\rmv_i)} E_{ii}-(-1)^{p (\rmv_j)}E_{jj}\,.
\end{equation}
The elements $h_{ij}$ lie in $\fh$, and the elements  $h_i: =
h_{i,i+1}$, $i \in I$, are linearly independent.  Let
\begin{equation}
  \label{eq:2.2}
  \alpha_i = \epsilon_i - \epsilon_{i+1} \in \fh^*\, .
\end{equation}
Then the set $\Pi = \{ \alpha_i \}_{i \in I}$ is linearly
independent.

Let $a_{ij} = \alpha_i (h_j)$.  It is easy to see that $A =
(a_{ij})_{i,j \in I}$ is a symmetric matrix.  Moreover, this
matrix is $3$-diagonal, with diagonal entries
\begin{equation*}
  a_{ii}  = (-1)^{p (\rmv_{i})} + (-1)^{p (\rmv_{i+1})}\, ,
\end{equation*}
which are equal to $0$ if $i \in I_{\bo}$, and to $\pm 2$
otherwise, and the non-zero offdiagonal entries are $a_{i,i+1} =
a_{i+1,i} = -(-1)^{p(\rmv_{i+1})}$.

This matrix is depicted by the Dynkin diagram $\bullet- \bullet
- \bullet - \cdots - \bullet -\bullet$, consisting of $|I|$
nodes, where the $i$\st{th} node is $\otimes$, called grey, if 
$i \in I_{\bar{1}}$, and is $\bigcirc$, called white, otherwise.

Since the $h_i$ are supertraceless and $\Pi$, restricted to the
supertraceless diagonal matrices, is linearly independent
(resp. dependent) if $m> n$ (resp. $m =n$), we conclude that
$\corank \,A=0$ if $m> n$, and $\corank \,A=1$ if $m=n$.
Thus, we have constructed a realization of the matrix $A$.

The structure of a Kac--Moody superalgebra $\fg (A,I_{\bar{1}})$ in $\fg$ is
introduced by letting
\begin{equation}
  \label{eq:2.3}
  e_i = E_{i,i+1} \, , \, f_i = (-1)^{p(\rmv_i)} E_{i+1,i}\, , 
      \quad i \in I \, .
\end{equation}
Indeed, it is clear that $e_i$, $f_i$ and $\fh$ generate $\fg$,
and it is easy to check that they satisfy the Chevalley
relations.  Also $\fg$ has no ideals intersecting $\fh$
trivially.

Furthermore, the supertrace form:
\begin{equation*}
(a|b) = \str \,  ab \, .
\end{equation*}
is an invariant bilinear form on $\fg$. The induced bilinear form on $\fh^*$ 
is given by:
\begin{equation}
    \label{eq:2.4}
(\epsilon_i | \epsilon_j) = \delta_{ij} (-1)^{p(\rmv_i)} ,\, \, 
   i,j =1,2,  \ldots , m+n\, . 
\end{equation}

The set of roots is
\begin{equation}
  \label{eq:2.5}
  \Delta= \{ \epsilon_i-\epsilon_j |
\,\, i,j = 1,2,\ldots ,m+n ;\, i\neq j \} \,,
\end{equation}
a root $\epsilon_i-\epsilon_j$ being even if $p (\rmv_i)= p
(\rmv_j)$, and odd otherwise.  The set of positive roots is
\begin{equation}
  \label{eq:2.6}
  \Delta_+ = \{ \epsilon_i - \epsilon_j |\,\, i,j = 1,2,\ldots ,m+n \,; 
\, i <j \}\, .
\end{equation}
The root spaces are $\fg_{\epsilon_i-\epsilon_j} =\CC E_{ij}$,
and the triangular decomposition is the usual one:  $\fn_+$ (resp. $\fn_-$) consists of strictly upper (resp. lower)
triangular matrices.

Note that, in view of (\ref{eq:2.4}), a root
$\epsilon_i-\epsilon_j$ is odd if and only if it is isotropic
(this is not the case for other finite-dimensional Kac--Moody
superalgebras).  The orthogonal reflection with respect to an even root
$\epsilon_i-\epsilon_j$ (which is integrable since $\dim \fg <\infty$) is the
transposition of $i$ and $j$.  Hence the Weyl group $W$ of $\fg$
is $S_m \times S_n$, where $S_m$ (resp. $S_n$) is
the group of permutations of all even (resp. odd) vectors $\rmv_i$'s.

\begin{example}
  \label{ex:2.1}

The constructed above Kac--Moody superalgebra structures on $\fg = \sl_{2|1}$ correspond
to three Dynkin diagrams:
\begin{equation*}
 \otimes - \otimes \, ,\,\, \quad \bigcirc - \otimes \, , \,\, \quad \otimes -\bigcirc\,, 
\end{equation*}
and those on $\fg = \gl_{ 2|2}$ to the Dynkin diagrams:
\begin{equation*}
 \otimes -\bigcirc  -\otimes ,\,   \quad  \bigcirc - \otimes - \bigcirc \, ,\,
    \quad  \otimes - \otimes - \otimes \, .
\end{equation*}
The corresponding Cartan  matrices $A$ are respectively:
\begin{equation*}
\left(
  \begin{array}{rr}
     0 & 1 \\ 1 & 0 
  \end{array}  \right)\qquad ,
\left(
\begin{array}{rr}
     2 & -1 \\ -1 & 0 
  \end{array}  \right)\qquad ,\quad
\left(
  \begin{array}{rr}
     0 & -1 \\ -1 & 2 
  \end{array}  \right)\, ,
\end{equation*}
and, up to an overall sign (which does not change the Kac--Moody
superalgebra structure):
\begin{equation*}
\left(
  \begin{array}{rrr}
     0 & 1 &0\\ 1 & -2 &1\\0 & 1 & 0
  \end{array}  \right)\, , \qquad 
\left(
\begin{array}{rrr}
     2 & -1 & 0\\ -1 & 0 &1 \\0 & 1 & -2
  \end{array}  \right)\, , \qquad
\left(
  \begin{array}{rrr}
     0 & -1  & 0\\ -1 & 0 & 1 \\0 & 1 & 0 
  \end{array}  \right)\, .
\end{equation*}

\end{example}

Recall that a Kac--Moody superalgebra structure exists on an
almost simple finite-dimensional Lie superalgebra iff either
$\fg$ is a Lie algebra (then it is unique up to conjugacy), or
$\fg \simeq \sl_{m|n}$ with $m > n$, $\gl_{n|n}$, $osp_{m|n}$,
$D(2|1 ; a)$, $F(4)$ or $G(3)$ \cite{K1}.

Next, we proceed to construct examples of Kac--Moody superalgebra
structures in the affine superalgebra
\begin{equation*}
  \hat{\fg} = \tilde{\fg} \oplus \CC K \oplus \CC d \, , 
\end{equation*}
associated to a finite-dimensional Kac--Moody superalgebra
$\fg$.  
Here $\tilde{\fg} = \CC F + \sl_ { n|n} [t,t^{-1}]$ if
$\fg = \gl_{n|n}$, and $\tilde{\fg} = \fg [t,t^{-1}]$ in all
other cases, and $K$ is a non-zero central element of
$\hat{\fg}$.  The brackets of all other elements of $\hat{\fg}$
are as follows $(a,b \in \fg \, , m,n \in \ZZ)$:
\begin{equation*}
  [at^m ,bt^n] = [a,b]t^{m+n}+m (a|b)\delta_{m,-n}K,\,\, [d,at^m]=m
  at^m\, .
\end{equation*}
In particular, $\fg$ is a subalgebra of $\hat{\fg}$.

Let $\fg = \fg (A, I_{\bar{1}})$ be  a Kac--Moody subalgebra structure 
on $\fg$ with
$\{h_i \}_{i \in I} \subset \fh$, $A=(a_{ij})_{i,j \in I}$ a
symmetric matrix, and $\Pi = \{ \alpha_i \}_{i \in I}$ defined
by (\ref{eq:1.1}).  Let $\theta =\sum_{i \in I} k_i \alpha_i
\in \Delta_+$ be the highest root (i.e. $\sum_i k_i$ is maximal).

Let $\hat{I}= \{ 0 \} \bigcup I$, and let $\hat{I}_{\bo} = 
 I_{\bar1}$ if $\theta$ is an even root, and
$\hat{I}_{\bo} = \{ 0 \} \bigcup I_{\bar{1}}$ otherwise.  Let
$\hat{A}$ be the $\hat{I} \times \hat{I}$-matrix obtained from
$A$ by adding $0$\st{th} row and column, where
\begin{equation*}
  a_{00} = (\theta | \theta)\, , \, a_{0i}=a_{i0}=-(\theta |
  \alpha_i) \hbox{\,\, for \,\,} i \in I \, .
\end{equation*}
Note that $\det \hat{A} =0$.

Next, we construct a realization of the matrix $\hat{A}$.  Let

\begin{equation}
\label{eq:2.7a}
  \hat{\fh} = \fh + \CC K + \CC d \, , 
\end{equation}
and let $\delta$ be the linear function on $\hat{\fh}$, defined
by
\begin{equation}
  \label{eq:2.7}
  \delta |_{\fh + \CC K} = 0 \, , \, \delta (d) =1\, .
\end{equation}
Let $h_0 = K-\theta$, $\alpha_0 =\delta -\theta$.  Then $\hat{\fh}$ and
$\{h_i\}_{i \in \hat{I}}$ define a realization of $\hat{A}$, with
the set of simple roots $\hat{\Pi} = \{ \alpha_0 \} \bigcup
\Pi$.

We now consider in more detail an important for this paper example
$\fg = s\ell_{m|n}$ with $m>n$ or $g\ell_{n|n}$. 
We introduce a structure of a Kac--Moody superalgebra in
$\hat{\fg}$, associated to the basis $\rmv_1 , \ldots ,
\rmv_{m+n}$ of $V$ as follows.  Let $\hat{I} = \{ 0 \} \cup I$,
and let $\hat{I}_{\bo} = I_{\bo}$ if $\rmv_1$ and $\rmv_{m+n}$ have the
same parity, and $\hat{I}_{\bar{1}} = \{ 0 \} \cup I_{\bo}$ otherwise.
Let $\hat{A}$ be the $\hat{I} \times \hat{I}$-symmetric matrix,
obtained from $A$ by adding $0$\st{th} row and column, where 
$a_{00} = (-1)^{p (\rmv_1)} + (-1)^{p (\rmv_{m+n})}$,
$a_{01} = - (-1)^{p (\rmv_1)}$, $a_{0,m+n-1} =(-1)^{p
  (\rmv_{m+n})}$.  Note that the sum of entries of each row of
$\hat{A}$ is zero.  Such a matrix is depicted by the extended
Dynkin diagram, which is a cycle,
%
where the additional, $0$\st{th} mode, is grey if $0 \in
\hat{I}_{\bar{1}}$ and is white otherwise.

Clearly, $\corank \hat{A} \leq 1$ if $m >n$ and $\corank
\hat{A}\leq 2$ if $m=n$.  In fact, we have equalities since
$\sum_{i \in \hat{I}} \alpha_i = \delta$, and, if $m=n$, we have
a linear dependence of the $\alpha_i$ , restricted to
$\fh'$ (= span of the $h_{ij}$).
Hence $\dim \hat{\fh} = |\hat{I}| + \corank \hat{A}$, and we indeed have 
constructed a realization of the matrix $\hat{A}$.

The structure of a Kac--Moody superalgebra $\fg (\hat{A},
\hat{I}_{\bo})$ in $\hat{\fg}$ is introduced by letting $e_i$,
$f_i$ for $i \in I$ being the same as for $\fg$, and
\begin{equation*}
  e_0 = e_{-\theta} t \, , \, f_0 = e_\theta t^{-1}\, ,
\end{equation*}
where $e_{\pm \theta} \in \fg_{\pm \theta}$ are chosen such that
$(e_\theta | e_{-\theta}) =1$.  (For $\fg = s\ell_{m|n}$ with
$m>n$ and $g\ell_{n|n}$ we take $e_0 = E_{1,m+n} t$, $f_0 =
(-1)^{p (v_{m+n})} E_{m+n,1} t^{-1}$.)

The invariant bilinear form $(\, . \, | \, . \, )$ extends from
$\fg$ to the invariant bilinear form on $\hat{\fg}$ by the
following formulae ($a,b \in \fg, \,  \, i,j \in \ZZ$):

\begin{equation*}
(at^i| bt^j) = \delta_{i,-j} (a|b),\, (at^i |\CC K + \CC d) =
(K|K)=0,\, (d|d)=0,\, (K|d) =1.
\end{equation*}

We identify $\hat{\fh}^*$ with the space
\begin{equation*}
  \fh^* \oplus \CC \delta \oplus \CC \Lambda_0 \, ,
\end{equation*}
where $\delta$ is defined by (\ref{eq:2.7}), and 
\begin{equation}
  \label{eq:2.8}
  \Lambda_0 |_ {\fh + \CC d} =0 \, , \, \Lambda_0 (K) =1\, .
\end{equation}
Then the induced bilinear form on $\hat{\fh}^*$ is given by
(\ref{eq:2.4}) and
\begin{equation}
  \label{eq:2.9}
  (\fh^* | \CC \delta + \CC \Lambda_0) =0 \, , \, (\delta |\delta)
  =0\, , \, (\Lambda_0 | \Lambda_0)=0 \, , \, (\Lambda_0
  |\delta) =1 \, .
\end{equation}

The set of roots $\hat{\Delta}$ of $\hat{\fg}$ is the union of
the sets of real roots $\hat{\Delta}^{re}$ and imaginary roots
$\hat{\Delta}^{im}$, where (cf. (\ref{eq:2.5})):
\begin{equation*}
  \hat{\Delta}^{re} = \{ \alpha + s \delta |\, \alpha \in \Delta \,
  , \, s \in \ZZ \}\, , \quad \hat{\Delta}^{im}  = \{ s \delta |\,
  s \in \ZZ \backslash \{ 0 \}\}\, , 
\end{equation*}
the parity of a root $\alpha + s \delta$ equals that of $\alpha$,
and all $s \delta$ are even roots.  The set of positive roots is
the union of
\begin{equation*}
  \hat{\Delta}^{re}_+ = \{ \alpha + s \delta |\, \alpha \in \Delta
  \, , \, s > 0 \} \cup \Delta_+ \, , \quad \hat{\Delta}^{im}_+ =
  \{ s \delta |\, s > 0 \} \, .
\end{equation*}
The root spaces are $\hat{\fg}_ {\alpha + s \delta} = \fg_\alpha t^s$,
and $\hat{\fg}_{s\delta} = \fh t^s$, except for $\fg = g\ell_{n|n}$, where
$\hat{\fg}_{s\delta} = t^s\{ h \in \fh|\,  \str\, h =0 \} $\, .

The triangular decomposition is:
\begin{equation*}
  \hat{\fg} = \hat{\fn}_- \oplus \hat{\fh} \oplus \hat{\fn}_+ \, ,
\end{equation*}
where $\hat{\fn}_\pm = \fn_\pm + \sum_{s>0} t^{\pm s} \fg$,  For
example, if 
$\fg = s \ell_{m|n}$, $m>n$ or $g\ell_{n|n}$, then
$\hat{\fn}_+$ consists of all supertraceless matrices over $\CC
[t]$, which are strictly upper triangular at $t=0$.

Note that all imaginary roots are isotropic, and the even real
roots are not.  It follows from the description of root spaces
that all even real roots are integrable.  Hence, by definition, the
Weyl group $\hat{W}$ of $\hat{\fg}$ is generated by reflections
$r_\alpha$, $\alpha \in \Delta^{re}_{\bar{0}}$.

An important alternative description of the group $\hat{W}$  is
as follows.
Given $\alpha \in \fh^*$, define the following automorphism
$t_{\alpha}$ of the vector space $\hat{\fh}^*$:
\begin{equation}
  \label{eq:2.10}
  t_\alpha (\lambda) = \lambda + \lambda (K) \alpha - ((\lambda
  |\alpha) + \frac{(\alpha |\alpha)}{2} \lambda (K))\delta\, .
\end{equation}
It is easy to check that $t_\alpha$ preserves the bilinear form 
$\bl$ on $\hat{\fh}^*$ and that $t_\alpha t_\beta = t_{\alpha +
  \beta}$, $\alpha ,\beta \in \fh^*$.  Given an additive subgroup
$L \subset \fh^*$, let $t_L = \{ t_\alpha| \alpha \in L \}$.

Let $L = \ZZ \{ \alpha^\vee | \alpha \in \Delta_{\bar0}\}$.  Then
we have (\cite{K2}, Chapter 6):
\begin{equation}
  \label{eq:2.11}
  \hat{W} = W \ltimes t_L.
\end{equation}

\begin{remark}
\label{rem:2.2}
Note that 
$\epsilon_+(t_\alpha)=1$
for $\alpha \in L$, but $\epsilon_-(t_\alpha)$ is sometimes $-1$.
However for the affine superalgebra, associated to $s \ell_{m|n}$, $m>n$, 
or to $g\ell_{n|n}$, it is always $1$.
\end{remark}

\begin{example}
  \label{ex:2.2}
The extended (symmetric) Cartan matrices $\hat{A}$ 
for Example~\ref{ex:2.1} are obtained from
the matrices $A$ using the property that the sum of entries in
each row and each column is zero.
\end{example}

\section{Partially integrable and admissible highest weight modules over
  affine superalgebras}
\label{sec:3}

Let $\fg$ be either a simple finite-dimensional Lie algebra or
one of the simple finite-dimensional Lie superalgebras
$s\ell_{m|n}, m>n$, $osp_{m|n}$, $D (2|1;\alpha)$, $F(4)$,
$G(3)$, or $g\ell_{n|n}$, endowed with a structure of a
Kac--Moody superalgebra $\fg (A,I_{\bo})$, and let $\hat{\fg}$ be the
corresponding affine superalgebra with the structure of the
Kac--Moody superalgebra $\fg (\hat{A}, \hat{I}_{\bo})$ 
(see Section \ref{sec:2}).

Let $h^\vee$ be the half of the eigenvalue of the Casimir
operator, associated to the bilinear form $\bl$ on $\fg$.  It is
given by the usual formulae \cite{K1}, p.~85:
\begin{equation}  
\label{eq:3.1}   
h^\vee = (\rho |\theta) + \frac12 (\theta | \theta)\,;\,\,h^\vee (a|b) = 
\sum_{\alpha \in \Delta_+}(-1)^{p(\alpha)}\alpha(a)\alpha(b),\,\,\, a,b\in \fh\, .
\end{equation}
Recall that $\theta \in \Delta_+$ is the highest root and $\rho$
is a Weyl vector, defined in Section 1. Recall also that we have
the following ``strange'' formula 
(cf. \cite{KW3}):
\begin{equation}  
\label{eq:3.0}   
 (\rho|\rho)= h^\vee (\sdim \fg) /12\, .
\end{equation}

Now we introduce the important subset $  \Delta^\#_{\bar0}$
of the set of even roots of $\fg$ \cite{KW3}. 
If $h^\vee \neq 0$ (which happens iff 
$\fg \neq g\ell_{n|n}$,
$osp_{2n+2|2n}$ or $D (2|1;\alpha)$ \cite{K1}), let 
\begin{equation}
  \label{eq:3.2}
  \Delta^\#_{\bar0} 
= \{ \alpha \in \Delta_{\bar{0}} |\,\, h^\vee
  (\alpha | \alpha) >0 \} .
\end{equation}
In the case $\fg = g\ell_{n|n}$ there are two choices of
 $ \Delta^\#_{\bar0}$, described in the example below.  
\begin{example}
  \label{ex:3.1}
Let $\fg = s\ell_{m|n},\, m>n$, or 
$g\ell_{n|n}$ 
with one of the
structures of a Kac--Moody algebra, described in Section \ref{sec:2}.
Then $h^\vee = m-n$ for the invariant bilinear form   
$(a|b) = \str $ $ab$ on $\fg$.
We have:  
$\Delta_{\bar{0}}^\# = 
\{ \epsilon_i - \epsilon_j | \,p (v_i)= p(v_j) = \bar{0},\,i\neq j \}$. 
If $m=n$, there is another choice:
$\Delta_{\bar{0}}^\# = \{ \epsilon_i - \epsilon_j |\, \,p (v_i)= p(v_j) = 
\bar{1},\,i\neq j \}$. 
For all $\alpha \in \Delta_{\bar{0}}^\#$ we have 
$(\alpha | \alpha)=2$, 
except for the second choice of
$\Delta_{\bar{0}}^\#$ in the case $m=n$ , when $(\alpha | \alpha)=-2$.
\end{example}
Note that in all cases the set $\Delta^\#_{\bar{0}}$ is $W$-invariant.
Let $L^\#$ be the $\ZZ$-span of  the set $\{ \alpha^\vee |\, \alpha
\in \Delta^\#_{\bar0} \}$, and let's introduce the following subgroup
of the Weyl group $\hat{W}$:
\begin{equation}
  \label{eq:3.4}
  \hat{W}^\# = W \ltimes t_{L^{\#}}\, .
\end{equation}

Let $\Lambda \in \hat{\fh}^*$ and let $L (\Lambda)$ be the
corresponding $\hat{\fg} = \fg (\hat{A}, \hat{I}_{\bo})$-module.
Then the central element $K$ acts as a scalar $\Lambda (K)$,
called the {\em level} of $\Lambda$, which we denote by the same letter $K$,
unless where confusion may arise.

\begin{definition}
  \label{def:3.2}
The $\hat{\fg}$-module $L(\Lambda)$ is called {\em partially
  integrable} if 

(i) any root $\alpha + n\delta$, where
$\alpha \in \Delta^\#_{\bar0}$, $n \in \ZZ$,  is
$\Lambda$-integrable; 

(ii)~~any root $\alpha \in
\Delta_{\bar{0}}$ is $\Lambda$-integrable.  
\end{definition}
Condition (i) means
that $L (\Lambda)$ is integrable with respect to the subalgebra
$\hat{\fg}_{\bar0}^\#$ of $\hat{\fg}$ 
, where $\hat{\fg}_{\bar0}^\#$ is the affine subalgebra with the set
of real roots $\{\alpha+n\delta|\, \alpha\in \Delta_{\bar0}^\#, n\in \ZZ\}$
, and condition (ii) means that $L(\Lambda)$ is
integrable with respect to $\fg_{\bar0}$ ($=$ locally finite with
respect to $\fg$).

We let $\alpha_j^\vee =\alpha_j$ if $\alpha_j$ is a simple isotropic root,
and define it as in (\ref{eq:1.4}) otherwise.
Define the fundamental weights $\Lambda_i \in \hat{\fh}^*$, $i \in \hat{I}$, by
\begin{equation}
  \label{eq:3.5}
  (\Lambda _i | \alpha^\vee_j) = \delta_{ij}\, , \quad j \in
  \hat{I}\, , \quad \Lambda_j (d) =0\, ,
\end{equation}
and also $\Lambda_j(F)=0$ for $\fg = g\ell_{n|n}$ (which
is consistent with (\ref{eq:2.8})).

As has been explained in Section~1, we may assume that $\Lambda
(d) = 0$, and $\Lambda (F) =0$ if $\fg = g\ell_{n|n}$
for a highest weight $\Lambda$.  Then we can write:
\begin{equation}
  \label{eq:3.6}
  \Lambda = \sum_{i \in {\hat{I}}} m_i \Lambda_i
 \hbox{\,\, for some \,\,} m_i \in \CC\, .
\end{equation}
The numbers $m_i$ are called the {\em labels} of $\Lambda$ .

Define $\hat{\rho} \in \hat{\fh}^*$ by
\begin{equation}
  \label{eq:3.7}
  (\hat{\rho} | \alpha_j) =\frac{1}{2} a_{jj}\, , \quad j \in \hat{I} \,
  , \quad (\hat{\rho}|d)=0\, ,
\end{equation}
and, in addition, $\hat{\rho} (F)=0$ 
if $\fg=g\ell_{n|n}$.  It is easy to see that (cf. \cite {K2}, Chapter 12)
\begin{equation}
  \label{eq:3.8}
  \hat{\rho} = \rho  + h^\vee \Lambda_0 \,.
\end{equation}

In \cite{KW4} we gave a classification of partially integrable
modules $L (\Lambda)$ in terms of the labels $m_i$ of $\Lambda$.
Here we give the answer only for the relevant to this paper two
examples.  Unlike in the present paper, we used in \cite{KW4}
Dynkin diagrams with minimal number of grey nodes.  Here we use
diagrams with a white $0$\st{th} node and recalculate the results
of \cite{KW4} using Proposition~\ref{prop:1.7}(b).

\begin{example}
  \label{ex:3.3} $\fg = s\ell_{2|1}$; 
we use its first Dynkin diagram in Example \ref{ex:2.1} and the corresponding 
extended Dynkin diagram.
The level of a $\hat{\fg}$-module $L(\Lambda)$, where $\Lambda$
has labels $m_0$, $m_1$, $m_2$, is:
\begin{equation*}
  K= m_0+ m_1 + m_2 \, .
\end{equation*}
We have 
$\Delta_{\bar{0}}^\# =\{\pm\theta\}$, where $\theta=\alpha_1+\alpha_2$. 
Let $K' = m_1 + m_2$.  Then $L (\Lambda)$ is partially integrable iff 
$m_0, K' \in \ZZ_{\geq 0}$ 
(hence $K \in \ZZ_{\geq 0}$), and $m_1=m_2=0$ if
$K'=0$ \cite{KW4} .  (Note that in these cases $L
(\Lambda)$ is integrable.  In fact, any partially integrable $\hat{\fg}$-module
$L(\Lambda)$ is integrable iff $\fg$ is either a Lie algebra, or
$\fg \simeq s\ell_{m|1}$ with $m>1$, or $\fg \simeq osp_{m|n}$ with
$m=1$ or $2$ and $n \geq 2$
\cite{KW4}).
\end{example}

\begin{example}
  \label{ex:3.4}$\fg = g\ell_{2|2}$; we use its first Dynkin diagram in 
Example \ref{ex:2.1} and the corresponding extended Dynkin diagram. 
The level of a $\hat{\fg}$-module $L(\Lambda)$ is
\begin{equation*}
  K=m_0 + m_1 - m_2 + m_3  \, .
\end{equation*}
The first choice of 
$\Delta_{\bar{0}}^\#$ is $\{\pm\theta\}$, where 
$\theta=\alpha_1+\alpha_2+\alpha_3$. 
Let $K'=m_1-m_2+m_3$.  Then $L(\Lambda)$ is partially integrable
iff:  $m_0 , K' \in \ZZ_{\geq 0}$ (hence $K \in \ZZ_{\geq 0}$), and
$m_1= m_3 =0$ if   $K'=0$ and $m_1 m_3 = 0$ if $K'=1$. 
For this choice, $\theta $,
$\delta - \theta$ and $\alpha_2$ are $\Lambda$-integrable. 
The second choice is $\Delta_{\bar{0}}^\#=\{\pm\alpha_2\}$, i.e. 
$\alpha_2, \delta-\alpha_2$ and $\theta$ are $\Lambda$-integrable. 
Let then $K'' = m_0 + m_1 + m_3$. The partial integrability conditions
in this case are: $m_2\, , \, -K'' \in \ZZ_{\geq 0}$ 
(hence $-K \in \ZZ_{\geq 0}$), and $ m_1 =m_3 =0 $ if $K'' = 0$,
$m_1m_3 =0$ if $K'' =-1$. 
Thus, a $\hat{g\ell}_{2|2}$-module $L (\Lambda)$ is integrable iff $\dim L
(\Lambda) =1$.
\end{example}

\begin{definition}
  \label{def:3.5}
\begin{list}{}{}

\item (a)
Let $\lambda \in \fh^*$, where $\fh$ is a Cartan subalgebra of
$\fg$.  A $\lambda$-{\em maximal isotropic} subset of $\Delta$ is
a subset $T_\lambda$, consisting of the maximal number of
positive roots $\beta_i$, such that $(\lambda
|\beta_i) =0$ and $(\beta_i|\beta_j) =0$ for all $\beta_i,\beta_j
\in T_\lambda$.

\item (b)
A $\fg$-module $L(\Lambda)$ 
is called {\em tame} if there exists
a $\Lambda + \rho$-maximal isotropic subset of $\Delta$,
contained in the set of simple roots $\Pi$.

\item (c)
A $\hat{\fg}$-module $L (\Lambda)$ of level $K$ is called {\em tame} if the
$\fg$-module $L (\bar{\Lambda})$ is tame, where $\bar{\Lambda}
=\Lambda |_{\fh}$, and $K+h^\vee \neq 0$. 

\end{list}

\end{definition}

As usual, we introduce the Weyl denominator $R^+$ and superdenominator $R^-$ 
by:
\begin{equation}
  \label{eq:3.9}
  R^\pm = \frac{e^\rho \prod_{\alpha \in \Delta_{\bar{0}+}}(1-e^{-\alpha})}
  {\prod_{\alpha \in \Delta_{\bo +}}(1\pm e^{-\alpha})} \, . 
\end{equation}

\begin{conjecture}
  \label{conj:3.6}

\cite{KW3}.  Let $L (\lambda)$ be a tame finite-dimensional
$\fg$-module.  Then there exists a non-zero integer $j_\lambda$,
such that
\begin{equation}
  \label{eq:3.10}
  j_\lambda R^+ \ch^+_{L (\lambda)} = \sum_{w \in W} \epsilon_+(w)
w \frac{e^{\lambda +\rho}}{\prod_{\beta \in T_{\lambda + \rho}}(1+e^{-\beta})},
\end{equation}
where $T_{\lambda +\rho} \subset \Pi$ is a $\lambda +
\rho$-maximal subset of $\Delta$, and $\epsilon_+ (w) = \det_ \fh w$.

\end{conjecture}

It is not difficult to show, using (\ref{eq:1.8}), that
(\ref{eq:3.10}) implies the following formula for the
supercharacter:
\begin{equation}
  \label{eq:3.11}
  j_\lambda R^- \ch^-_{L(\lambda)} = \sum_{w \in W}
  \epsilon_- (w) w \frac{e^{\lambda + \rho}}
{\prod_{\beta \in
      T_{\lambda +\rho}} (1-e^{-\beta})}.
\end{equation}
%

\begin{remark}   
\label{rem:3.7}
(a) Due to Remark \ref{rem:1.5}, the 
$g\ell_{n|n}$-module 
$L(\lambda)$ 
remains irreducible when restricted to 
$s\ell_{n|n}$
. Furthermore, if
$\lambda (I_{2n})=0$, then $L(\lambda)$ is actually a module over 
$ps\ell_{n|n}:=s\ell_{n|n}/\CC I_{2n}$.

(b) Due to Remark \ref{rem:1.5}, the 
$\hat{g\ell}_{n|n}$-module
$L(\Lambda)$  remains irreducible when restricted to
$\hat{s\ell}_{n|n}
=s\ell_{n|n}[t, t^{-1}] +\CC K + \CC d$
. Furthermore, if $\Lambda (I_{2n})=0$, then $L(\Lambda)$ is actually a 
module over $\hat{ps\ell}_{n|n}=
ps\ell_{n|n}[t, t^{-1}] 
+\CC K + \CC d$.
\end{remark}   

\begin{conjecture}
  \label{conj:3.7} (cf. \cite{KW4}).  Let $L (\Lambda)$ be a
  partially integrable tame $\hat{\fg}$-module. Then
  \begin{equation} 
\label{eq:3.12} 
\hat{R}^+ \ch^+_ {L (\Lambda )} = \sum_{\alpha \in L^\#}t_\alpha  
(e^{\hat{\rho}+\Lambda-\rho-\bar{\Lambda}} R^+ \ch^+_ {L (\bar{\Lambda})})\,  
\end{equation}
(and, consequently, the same formula holds for $\ch^-_{ L (\Lambda)} $ if 
we replace
$\hat{R}^+$ by $\hat{R}^-$ and $R^+$ by $R^-$, and insert 
$\epsilon_- (t_\alpha)$ in front 
of $t_\alpha$).  Here in the LHS we have
the affine Weyl denominator $\hat{R}^+$ and superdenominator $\hat{R}^-$,
defined by:
\begin{equation}
  \label{eq:3.13}
  \hat{R}^\pm = e^{\hat{\rho}-\rho} R^\pm \prod^\infty_{n=1} 
  \left((1-e^{-n\delta })^{\ell} 
   \frac{\prod_{\alpha \in \Delta_{\bar0}} (1-e^{\alpha -n \delta})}
    {\prod_{\alpha \in \Delta_{\bo}} (1\pm e^{\alpha -n\delta})} \right)\, ,
\end{equation}
where $\ell$ is the multiplicity of the root $\delta$. 
\end{conjecture}
Note that
$\ell=\dim\,\fh$
if 
$\fg\neq g\ell_{m|m}$,
$\ell=\dim\,\fh -1=2m-1$ (resp. $=2m-2$) if $\fg= g\ell_{m|m}$
and $\Lambda(I_{2m})\neq 0$ (resp. $=0$).

Note that, using (\ref{eq:3.10}), formula (\ref{eq:3.12}) can be rewritten 
as follows:
 \begin{equation} 
\label{eq:3.15a} 
j_{\bar{\Lambda}}\hat{R}^\pm \ch^\pm_ {L (\Lambda )} = 
\sum_{w \in \hat{W}^\#}\epsilon_\pm(w)   
w \frac{e^{\Lambda +\hat{\rho}}}
{\prod_{\beta \in T_{\bar{\Lambda} + \rho}}(1\pm e^{-\beta})}.
\end{equation}
%
Conjectures \ref{conj:3.6} and  \ref{conj:3.7} have been verified
in many cases (\cite{KW4}, \cite{S}, \cite{GK}).


Given $\Lambda \in \hat{\fh}^*$ of level $K \neq -h^\vee$, by
analogy with the affine Lie algebra case, we introduce the following
number (\cite{K2}, Chapter 12):
\begin{equation}
  \label{eq:4.1}
  m_\Lambda = \frac{(\Lambda + \hat{\rho} | \Lambda +  \hat{\rho})}
  {2 (K+h^\vee)} - \frac{\sdim \fg}{24} \, ,
\end{equation}
called the {\em modular anomaly} of $\Lambda$.  As in the affine
algebra case, using the ``strange'' formula (\ref{eq:3.0}), we
obtain another important expression for $m_\Lambda$:
\begin{eqnarray}
  \label{eq:4.2}
  m_\Lambda &=& h(\Lambda) - \frac{c(K)}{24},\\
\noalign{\hbox{where}}\nonumber\\
 \label{eq:4.3}
   h(\Lambda) &=& 
     \frac{(\Lambda + 2 \hat{\rho}|\Lambda)}{2(K+h^\vee)} \, , \\[1ex]
 \label{eq:4.4}  
    c(K) &=& \frac{K \sdim \fg}{K+h^\vee}\, .
\end{eqnarray}

As in the affine Lie algebra case, $c(K)$ is the central charge
of the Sugawara's Virasoro field $L (z) = \sum_{n \in \ZZ} L_n
z^{-n-2}$, and $h_\Lambda$ is the minimal eigenvalue of $L_0$ in
$L (\Lambda)$ (cf. \cite{K2}, Chapter 12).

As in the affine Lie algebra case, in order to ``improve''
modular invariance properties of characters, one introduces the
{\em normalized character and supercharacter} of the $\hat{\fg}$-module 
$L(\Lambda)$ by the formula
\begin{equation} \label{eq:4.5}
  \ch^\pm_\Lambda = e^{-m_\Lambda \delta} \ch^\pm_{L(\Lambda)}\, .
\end{equation}
Note that $\ch^\pm_\Lambda$ depends only on $\Lambda \mod \CC \delta$.


\begin{definition}
  \label{def:3.8}

An injective homomorphism $\varphi : \hat{\fg} \to \hat{\fg}$ is called
{\em compatible} if it respects the
triangular decomposition and preserves the invariant bilinear
form.
\end{definition}

It is clear that for a compatible homomorphism $\varphi$, the map
$\varphi : \hat{\fh} \to \hat{\fh}$ is an isomorphism, that
$\hat{\Delta}_+\subset \varphi^*(\hat{\Delta}_+)$,  and
that $\varphi (K) = MK$, where $M$ is a positive integer, called
the {\em degree} of $\varphi$.
Furthermore, the subset $S=(\varphi^*)^{-1} (\hat{\Pi}) \subset \hat{\Delta}_+$
clearly satisfies the following two properties:
\begin{equation}
  \label{eq:3.14}
  \alpha - \beta \notin \hat{\Delta} \hbox{\,\, for \,\,}
  \alpha,\beta \in S\, ;
\end{equation}
\begin{equation}
  \label{eq:3.15}
\QQ S = \QQ \hat{\Pi}\, .
\end{equation}
Such an $S$ is called a {\it simple subset} of $\hat{\Delta}_+$.
As in \cite{KW2}, it is not difficult to classify all simple subsets $S$
in 
$\hat{\Delta}_+$.  (In fact,
for $\fg = s\ell_{m|n}$, $m>n$, and $g\ell_{n|n}$ the answer
obviously is the same as for the Lie algebra $s\ell_{m+n}$.)  We
give here the answer for $\fg = s\ell_{2|1}$ and $g\ell_{2|2}$.

\begin{example}
  \label{ex:3.9}
$\otimes_1 - \otimes_2$
from Example \ref{ex:2.1}.  There are two types of simple subsets $S \subset
\hat{\Delta}_+$:
%
\begin{eqnarray*}
S_1 &=& \{ \alpha_0+k_0\delta \,,\, \alpha_1 +k_1 \delta \,,\,
\alpha_2 + k_2 \delta \} \, ;\\
S_2 &=& \{ -\alpha_0 + k_0 \delta \,,\, -\alpha_1 + k_1\delta
\,,\, -\alpha_2+k_2\delta\}\, .
\end{eqnarray*}
Here $k_i \in \ZZ$ are such that $S_1$ and $S_2$ lie in
$\hat{\Delta}_+$ (i.e, $k_i \in \ZZ_{\geq 0}$ for $S_1$, and $k_i \in
\NN$ for $S_2$).

\end{example}

\begin{example}
  \label{ex:3.10}
$\otimes_1 - \bigcirc_2 -\otimes_3$
from Example \ref{ex:2.1}.  
There are four types of simple subsets $S
\subset \hat{\Delta}_+$:
%
\begin{eqnarray*}
  S_1 &=& \{ \alpha_0 +k_0 \delta \,,\, \alpha_1 + k_1 \delta
  \,,\,\alpha_2 +k_2 \delta \,,\, \alpha_3 + k_3 \delta \};\\
  S_2 &=& \{ -\alpha_0 +k_0 \delta \,,\, -\alpha_1 + k_1 \delta
  \,,\,-\alpha_2 +k_2 \delta \,,\,- \alpha_3 + k_3 \delta \};\\
  S_3 &=&  \{ \alpha_0 +k_0 \delta \,,\, \alpha_1+\alpha_2 + k_1 \delta
  \,,\,-\alpha_2 +k_2 \delta \,,\, \alpha_2 +\alpha_3 + k_3 \delta \};\\
  S_4 &=&  \{- \alpha_0 +k_0 \delta \,,\,- \alpha_1-\alpha_2 + k_1 \delta
  \,,\,\alpha_2+k_2\delta,\, -\alpha_2 - \alpha_3  + k_3 \delta \}\,.
\end{eqnarray*}
Here $k_i $ are all non-negative integers, such that 
$S \subset \hat{\Delta}_+$.

\end{example}

\begin{definition}
  \label{def:3.11}
Let $\varphi : \hat{\fg} \to \hat{\fg}$ be a compatible
homomorphism, and let $S=\varphi^{*-1} (\hat{\Pi})$.  Then $\Lambda
\in \hat{\fh}^*$ is called a (corresponding to $\varphi$) {\em principal
admissible} weight (with respect to $S$) if the following two
properties hold:

\begin{list}{}{}
\item  (i)~~$(\Lambda + \hat{\rho}\, |\, \alpha^\vee) \in \ZZ$ for $\alpha
  \in \hat{\Delta}_{\bar0}$ implies that $\alpha \in \ZZ S \cap \hat {\Delta}$;
\item (ii)~~ the $\hat{\fg}$-module $L (\Lambda^0)$ is partially
  integrable, where
  \begin{equation}
    \label{eq:3.16}
    \Lambda^0 = \varphi^* (\Lambda + \hat{\rho}) - \hat{\rho}\,.
  \end{equation}
\end{list}
\end{definition}
Note that the level of $\Lambda^0$ is expressed via the level $\Lambda(K)$
of $\Lambda$ by
\begin{equation}
  \label{eq:3.17}
  \Lambda^0 (K) = M ( \Lambda(K)+h^\vee) - h^\vee \, ,
\end{equation}
where $M$ is the degree of $\varphi$.

\begin{conjecture}
  \label{conj:3.12}
Let $\varphi : \hat{\fg} \to \hat{\fg}$ be a compatible
homomorphism and let $\Lambda \in \fh^*$ be a corresponding principal
admissible weight.  Then we have the following formula for normalized 
characters:
\begin{equation}
\label{eq:3.a}
  (\hat{R}^\pm \ch^\pm_{\Lambda})(h) = (\hat{R}^\pm \ch^\pm_{\Lambda^0})
  (\varphi^{-1} (h)), \, h \in \hat{\fh}\, .
\end{equation}

\end{conjecture}
This formula is proved, in the case when $\hat{\fg}$ is an affine
Lie algebra, for more general, admissible modules, in
\cite{KW1}. We conjecture that a similar result holds also in the
Lie superalgebra case.  (However, all admissible modules are
principal admissible for Lie superalgebras, considered in this
paper.) Note that Conjecture \ref{conj:3.7}, in the form, given by
equation (\ref{eq:3.15a}) is proved in \cite{GK} for all cases,
considered in the present paper.


As shown in \cite{KW2}, formula (\ref{eq:3.a}) can be written in
a more explicit form as follows. Let $M$ be  a positive integer
and let $S_{(M)} =\{ (M-1)\delta + \alpha_0$, $\alpha_1, \ldots ,
\alpha_\ell \}$ be a simple set of degree $M$.
Let $\beta \in \fh^*$ and $y \in W$ be such that $S:=t_{\beta}y
(S_{(M)}) \subset \hat{\Delta}_+$.  Then $S$ is a simple set.  All
principal admissible weights with respect to $S$ and of level $K$
are of the form (up to adding a multiple of $\delta$):
\begin{equation}
  \label{eq:3.b}
  \Lambda = (t_\beta y) (\Lambda^0-(M-1)(K+h^\vee)\Lambda_0+\hat{\rho})-\hat{\rho}\,,
\end{equation}
where $\Lambda^0$ is a partially integrable weight, and, by
(\ref{eq:3.17}), we have:
\begin{equation}
  \label{eq:3.c}
  K=\frac{m+h^\vee}{M} - h^\vee\,,\, \hbox{\,\, where\,\,} m
  \hbox{\,\, is the level of \,\,}\Lambda^0\,.
\end{equation}

It is convenient to write formula (\ref{eq:3.a}) in the following
  coordinates on $\hat{\fh}$, which we will be using throughout the paper:
  \begin{equation}
    \label{eq:4.8}
    h =2 \pi i (-\tau d + z + tK)\, , \, \hbox{\,\, where \,\,}
       z \in \fh\,,\, t,\tau \in \CC\, .
  \end{equation}
We will always assume that $\Im \tau >0$ in order to have all our
series convergent.  Note that
$  e^{-\delta} (h) =q:=e^{2\pi i \tau}$, so that
  $|q|<1$.

In these coordinates formula (\ref{eq:3.a}) becomes (cf. \cite{KW2}):
\begin{equation*}
  (\hat{R}^\pm ch^\pm_\Lambda) (\tau,z,t) 
     = (\hat{R}^\pm ch^\pm_{\Lambda^0})
     \left(M \tau \,,\, y^{-1}(z+\tau \beta)\,,\, \frac1M \left(t+(z|\beta)
           +\frac{\tau (\beta |\beta)}{2}\right)\right)\,,
\end{equation*}
or, a little more explicitly:
\begin{equation}
\label{eq:3.d}
  (\hat{R}^\pm \ch^\pm_\Lambda) (\tau ,z,t) 
     = q^{\frac{m+h^\vee}{M} (\beta |\beta)}
        e^{\frac{2\pi i (m+h^\vee)}{M} (z|\beta)}
        (\hat{R}^\pm \ch^\pm_{\Lambda^0})
      \left(M\tau \,,\, y^{-1} (z+\tau\beta)\,,\, \frac{t}{M}\right)\,.
\end{equation}


Next, we describe the principal admissible weights in the cases considered 
in this paper.

\begin{proposition}
  \label{prop:3.13}
Let $\fg = s\ell_{2|1}$, let $\varphi : 
\hat{\fg} \to \hat{\fg}$ be a compatible homomorphism of degree $M$, and
let $\Lambda \in \hat{\fh}^*$ be a principal admissible weight with
respect to a simple subset $S \subset \hat{\Delta}_+$ such that $\Lambda^0 = m
\Lambda_0$, $m \in \ZZ_{\geq 0}$ (see (\ref{eq:3.16})), and $gcd
(M,m+1)=1$.  Then we have for the level $K$ of $\Lambda$:
\begin{equation}
  \label{eq:3.18}
  K=\frac{m+1}{M} -1 \, .
\end{equation}
Furthermore, the value of $M$ and the description of all such
principal admissible weights $\Lambda$ with respect to $S=S_i$, $i=1,2$,
from Example \ref{ex:3.9}, is as follows ($k_0, k_1, k_2 \in \ZZ_{\geq 0}$): 

\begin{eqnarray*}
  S&=& S_1 :\,\, M = k_0 +k_1 +k_2+1,\,\, k_1, k_2\geq 0,\, k_1+k_2\leq M-1,\\
  &&\Lambda^{(1)}_{k_1,k_2}= (m-k_0 (K+1)) \Lambda_0 - k_1 (K+1)
  \Lambda_1-k_2 (K+1) \Lambda_2 \, ;\\
  S &=& S_2 : \,\, M = k_0 +k_1 +k_2-1,\,\, 1\leq k_1, k_2 \leq M-1,\, k_1+k_2
\leq M, \\
 &&\Lambda^{(2)}_{k_1,k_2}= (k_0 (K+1)-m-2) \Lambda_0 +k_1 (K+1)
 \Lambda_1 + k_2 (K+1)\Lambda_2\, .
\end{eqnarray*}

\end{proposition}

\begin{proof}
  Formula (\ref{eq:3.18}) follows from (\ref{eq:3.17}) since
  $h^\vee =1$. In the case of $\Lambda$, admissible with respect
  to $S_1$, we have:
  \begin{eqnarray*}
    (\Lambda^0 +\rho |\alpha_0)= m+1 &=& (\Lambda +\rho | k_0 \delta +
    \alpha_0) = k_0 (K+1) + (\Lambda |\alpha_0) +1 \,;\\
   (\Lambda^0 +\rho |\alpha_1) =0 &=&  (\Lambda +\rho | k_1 \delta +
    \alpha_1) = k_1 (K+1) + (\Lambda |\alpha_1)\,;\\
  (\Lambda^0 +\rho |\alpha_2) =0 &=&  (\Lambda +\rho | k_2 \delta +
    \alpha_2) = k_2 (K+1) + (\Lambda |\alpha_2),
  \end{eqnarray*}
hence $(\Lambda |\alpha_0) =m-k_0 (K+1)\, , \, (\Lambda |
\alpha_1) = -k_1 (K+1), (\Lambda |\alpha_2) =-k_2 (K+1)$, and
similarly for $S_2$.

The computation of $M$ immediately follows from the fact that $M\delta$
equals to the sum of the elements of $S_i$.  Condition (i) on
principal admissible weight is equivalent to (\ref{eq:3.18}).

\end{proof}

\begin{proposition}
  \label{prop:3.14}
Let $\fg = g\ell_{2|2}$, let $\varphi :\hat{\fg}\to\hat{\fg}$ be a
compatible homomorphism at degree $M$, and let $\Lambda \in
\hat{\fh}^*$ be a principal admissible weight with respect to a simple subset 
$S \subset \hat{\Delta}_+$, such that $\Lambda^0 =m\Lambda_0$, where $m$
is a non-zero integer and $gcd (m,M)=1$.  Then we have for the level $K$ of
$\Lambda$:
\begin{equation}
  \label{eq:3.19}
  K=\frac{m}{M}\, .
\end{equation}
%
Furthermore, the value of $M$ and the description 
of all such principal admissible weights $\Lambda$
with respect to $S=S_i$, $i=1,2,3,4$, from Example \ref{ex:3.10},
is as follows 
($k_0,...,k_3 \in \ZZ_{\geq 0}$ are such that $S \subset \hat{\Delta}_+$):
\begin{eqnarray*}
  S=S_1 &: & M = \sum^3_{i=0} k_i +1\, , \\
  \Lambda^{(1)} &=& (m-k_0K) \Lambda_0 -k_1K \Lambda_1
  +k_2K\Lambda_2-k_3 K \Lambda_3;\\
S=S_2 &:&  M = \sum^3_{i=0} k_i -1\, , \\
  \Lambda^{(2)} &=& (k_0K-m-2) \Lambda_0 + k_1 K\Lambda_1 -(k_2 K+2) 
     \Lambda_2 + k_3K\Lambda_3;\\
S=S_3 &:&  M = \sum^3_{i=0} k_i +1\, , \\
 \Lambda^{(3)} &=& (m-k_0K) \Lambda_0  - (1+ (k_1 + k_2)K)
 \Lambda_1-(2+ k_2 K) \Lambda_2 - (1+ (k_2+k_3)K)\Lambda_3;\\
S=S_4 &:&  M = \sum^3_{i=0} k_i -1\, , \\
\Lambda^{(4)} &=& (k_0 K-m-2) \Lambda_0 + ((k_1+k_2) K+1) \Lambda_1 +
k_2K \Lambda_2 + ((k_2 + k_3) K+1)\Lambda_3\, .
\end{eqnarray*}

\end{proposition}

\begin{proof}
It is the same as that of Proposition \ref{prop:3.13}
\end{proof}

\section{Characters of partially integrable modules and  mock theta functions}
\label{sec:4}

In this section we rewrite the characters of partially integrable
modules $L (\Lambda)$ over an affine superalgebra $\hat{\fg}$ in
terms of ``mock'' theta  functions.  Throughout the section, $\fg$
is a finite-dimensional Lie superalgebra as in Section \ref{sec:3} and 
$\ell$ is its rank (= rank $\fg_{\bar0}$).

Let $m$ be a positive integer and let $j \in \ZZ /2m\ZZ$.  In the
course of the paper we shall 
often use the well-known {\em theta
  functions} of {\em degree}  $m$ (rather Jacobi forms)
$\Theta_{j,m}(\tau,z)=
\Theta_{j,m}(\tau,z,0)$, where $\Theta_{j,m}(\tau,z,t)$ are defined by
formula (\ref{eq:A1}) in the Appendix.
Especially important are the four Jacobi theta
functions of degree two \cite{M}:
\begin{eqnarray*}
  \vartheta_{00} = \Theta_{2,2} + \Theta_{0,2} \, , \, 
     \vartheta_{01} =-\Theta_{2,2} + \Theta_{0,2}\, , \, 
  \vartheta_{1 0} = \Theta_{1,2} + \Theta_{-1,2} \, , \, 
      \vartheta_{11} = i \Theta_{1,2} -i\Theta_{-1,2}\, ,
\end{eqnarray*}
discussed in detail in the Appendix.

Since $\hat{\rho} (K) = h^\vee$
by (\ref{eq:3.8}), we obtain the following formula for the affine
Weyl denominator and and superdenominator (\ref{eq:3.13}) in the coordinates
(\ref{eq:4.8}):
\begin{equation*}
  \hat{R}^\pm (h) = e^{2\pi i h^\vee t}
\frac{\prod^\infty_{n=1} ((1-q^n)^\ell \prod_{\alpha \in  \Delta_{\bar0,+}}   
 (1-e^{-2\pi i \alpha(z)}q^{n-1}) (1- e^{2\pi i \alpha (z)}q^n))}
{\prod^\infty_{n=1} \prod_{\alpha \in \Delta_{\bo,+}} 
(1\pm e^{-2 \pi i \alpha (z)} q^{n-1})(1\pm e^{2 \pi i \alpha (z)} q^n)} \, .
\end{equation*}
Using formulae (\ref{eq:4.7}) from the Appendix, we can rewrite these 
expressions (or
rather their normalization by a power of $q$) in terms of the four
Jacobi functions of degree two:
\begin{equation}
  \label{eq:4.9}
  q^{\frac{1}{24} \sdim \fg} \hat{R}^+ (h)= e^{2\pi i h^\vee t}
   i^{d_{\bar0}}\eta (\tau)^
{\ell -d_{\bar0}+d_{\bar1}}
    \frac{\prod_{\alpha \in \Delta_{\bar0,+}} \vartheta_{11} (\tau ,\alpha(z))}
      {\prod_{\alpha \in \Delta_{\bo,+}} \vartheta_{10} (\tau, \alpha (z))}\,;
\end{equation}
\begin{equation}  
\label{eq:4.10}
  q^{\frac{1}{24}\sdim \fg} \hat{R}^- (h) = e^{2\pi i h^\vee t}
     i^{d_{\bar0} - d_{\bar1}} \eta (\tau)^
{\ell -d_{\bar0}+ d_{\bar1} }
 \frac{\prod_{\alpha \in \Delta_{\bar0,+}} \vartheta_{11} (\tau ,\alpha (z))}
   {\prod_{\alpha \in \Delta_{\bar1,+}} \vartheta_{11} (\tau, \alpha (z))}\, .
\end{equation}
Here and further, $d_\alpha = |\Delta_{\alpha, +}|$, $\alpha \in \ZZ / 2\ZZ$.

In order to obtain a modular invariant family, we need to consider
also twisted affine Weyl denominators and superdenominators.  For
this purpose, fix an element $\xi \in \fh^*$, satisfying
\begin{equation}
\label{eq:4.3a}
  (\xi |\alpha) \in \frac12 p (\alpha)+\ZZ\, , \quad \alpha \in
  \Delta\, ,
\end{equation}
and let
\begin{equation*}  
  \hat{R}^{\tw,\pm} = t_\xi (\hat{R}^\pm)\, .
\end{equation*}
Recalling (\ref{eq:2.10}), we have:
\begin{equation*}
  t_\xi (\delta ) = \delta ,\,t_\xi (\alpha) = \alpha-(\xi|\alpha)\delta  
\hbox{\,\, if   \,\,} \alpha \in \Delta\, , \, 
\, \hat{\rho}^{\tw}: = t_\xi 
(\hat{\rho})=\hat{\rho}
+h^\vee \xi - (\frac12 h^\vee (\xi |\xi)+ (\rho |\xi))\delta\, .
\end{equation*}
Hence we obtain from (\ref{eq:3.13}):
\begin{equation*}
  \hat{R}^{\tw,\pm} = e^{\hat{\rho}^{\tw }} \prod^\infty_{n=1}
  \frac{(1-q^n)^\ell \prod_{\alpha \in \Delta_{\bar0,+}} 
    (1-e^{-\alpha} q^{n-1-(\xi |\alpha)}) (1-e^{\alpha} q^{n+(\xi |\alpha)})}
  {\prod_{\alpha \in
      \Delta_{\bo,+}} (1\pm e^{-\alpha}q^{n-1-(\xi |\alpha)})
(1\pm e^{\alpha}q^{n+(\xi |\alpha)})}\, .
\end{equation*}
This can be rewritten in terms of the four Jacobi functions, as in the
non-twisted case (cf. (\ref{eq:4.9}), (\ref{eq:4.10})).  
For this we need to use
Proposition \ref{prop:A6} from the Appendix.
%
By a straightforward  (but a bit lengthy calculation) we
derive formulae, similar to (\ref{eq:4.9}) and
(\ref{eq:4.10}):
\begin{eqnarray}
  \label{eq:4.11}
 q^{\frac{1}{24}\sdim \fg} \hat{R}^{\tw,+} (h) &=& e^{2 \pi i h^\vee t}
     (-1)^{2(\rho_{\bar0}|\xi)} i^{d_{\bar0}} \eta (\tau)^
 {\ell -d_{\bar0}+d_{\bar1}} 
\frac{\prod_{\alpha \in \Delta_{\bar{0},+}}\vartheta_{11} (\tau ,\alpha (z))}
  {\prod_{\alpha \in \Delta_{\bo,+}} \vartheta_{00} (\tau ,\alpha
    (z))}\, , \\[2ex]
  \label{eq:4.12}
 q^{\frac{1}{24}\sdim \fg} \hat{R}^{\tw,-}(h) &=& e^{2\pi i h^\vee t}
  (-1)^{2(\rho|\xi)-\frac12 d_{\bo}} i^ {d_{\bar0}}
   \eta (\tau)^ {\ell -d_{\bar0}+d_{\bar1}} 
\frac{\prod_{\alpha \in \Delta_{\bar0,+}} \vartheta_{11} (\tau, \alpha (z))}
{\prod_{\alpha \in \Delta_{\bo,+}} \vartheta_{01}(\tau ,\alpha (z)) } \,.
\end{eqnarray}

It turns out that modular transformation formulae can be written in a 
beautiful unified form if we use the following notations for
(\ref{eq:4.9}), (\ref{eq:4.10}), (\ref{eq:4.11}),
(\ref{eq:4.12}):

\begin{center}
 (\ref{eq:4.9})  $ = \hat{R}^{(\frac12)}_0       (\tau , z, t)$,\quad
 (\ref{eq:4.10}) $ = \hat{R}^{(0)}_0         (\tau, z, t)$,\quad
 (\ref{eq:4.11}) $ = 
\hat{R}^{(\frac{1}{2})}_{\frac{1}{2}}   (\tau, z ,t)$,\quad
 (\ref{eq:4.12}) $ = \hat{R}^{(0)}_{\frac{1}{2}}     (\tau , z, t)$,
\end{center}
i.e., the superscripts $(\frac{1}{2})$ and $(0)$ refer to the
denominator and superdenominator respectively, and the
subscripts $0$ and $\frac{1}{2}$ refer to the non-twisted and twisted
cases respectively.

\begin{theorem}
  \label{th:4.2}
Let $\epsilon,\epsilon' =0$ or 
$\frac{1}{2}$.  Then 
\begin{list}{}{}
\item (a)~~$\hat{R}^{(\epsilon)}_\epsilon \left(- \frac1\tau ,
    \frac{z}{\tau}, t-\frac{(z|z)}{2\tau} \right) =c_\epsilon
  (-i\tau)^{\frac12 \ell}  R^{(\epsilon)}_\epsilon (\tau,z,t)$,\hfill\break
  where $c_0 = (-i)^{d_{\bar0} -d_{\bar1}}$ and $c_{\frac{1}{2}} =$
 $ (-i)^{d_{\bar0}}$;\hfill\break

$\hat{R}^{(\epsilon')}_\epsilon
  \left( -\frac{1}{\tau} , \frac{z}{\tau} , t-
    \frac{(z|z)}{2\tau} \right) = c_{\epsilon, \epsilon'}
  (-i\tau)^{\frac12 \ell} R^{(\epsilon)}_{\epsilon'} (\tau,z ,t)$
  if $\epsilon \neq \epsilon'$,\hfill\break
where $c_{0,\frac12} = 
c_{\frac12,0} =
  (-1)^{2 (\rho |\xi)-\frac12 d_{\bar1} + 
d_{\bar0}}i^{d_{\bar0}}$.

\item (b)~~$\hat{R}^{(\epsilon)}_0 (\tau +1, z,t) = e^{\frac{\pi
      i}{12} \sdim \fg} \hat{R}^{(\epsilon)}_0 (\tau ,z,t)$;\hfill\break
  $\hat{R}^{(\epsilon)}_{\frac12} (\tau + 1,z,t) = (-1)^{2
    (\rho_{\bo}|\xi) - \frac12 d_{\bo }}e^{\frac{\pi i}{12} 
(2d_{\bar0} +d_{\bar1}+\ell)}\hat{R}^
{(\frac12-\epsilon)}_{\frac12} (\tau ,z,t)$.
\end{list}
\end{theorem}

\begin{proof}
  It is immediate by the modular transformation formulae for the four Jacobi
theta  functions of degree two, given by Proposition \ref{prop:A7} in the
Appendix
and the modular transformation formulae for the $\eta$-function:
\begin{equation*}
  \eta \left(-\frac{1}{\tau}\right) = (-i\tau)^\frac12 \eta (\tau) \, ,
  \, \,\,
  \eta (\tau +1) = e^{\frac{\pi i}{12}} \eta (\tau)\, .
\end{equation*}
\end{proof}

The representation theoretical meaning of twisted denominators and 
superdenominators is as follows. Let
 \begin{equation*}
    \hat{\fg}^{\tw} = \fg_{\bar0} [t,t^{-1}] \oplus \fg_{\bo}  
  [t,t^{-1}] t^{\frac12} \oplus \CC K \oplus \CC d  
\end{equation*}
with the same commutation relations and the same $\hat{\fh}$ as in Section~\ref{sec:2}.  The action of $t_\xi$ on $\hat{\fh}^*$ induces the action of $t_{-\xi}$ on $\hat{\fh}$,  
which extends to the following isomorphism 
$t_{-\xi} : \hat{\fg}^{\tw} \overset{\sim}{\to} \hat{\fg}$: 
\begin{eqnarray*}
t_{-\xi} (h \, t^n)  &=& ht^n + \xi (h) K \delta_{n,0}\, ,\,  h \in   \fh \, , \,\,\,\, t_{-\xi} (K) =K,\\
 t_{-\xi} (d) &=& d - \xi - \frac12 (\xi | \xi) K \, , \,\,    t_{-\xi} (e_\alpha t^n) = e_\alpha t^{n+ (\xi |\alpha)}\, , \, e_\alpha \in \fg_\alpha\,.
\end{eqnarray*}
Via this isomorphism, the $\hat{\fg}$-module $L (\Lambda)$ becomes a 
$\hat{\fg}^{\tw}$
-module, which we denote by 
$L^{\tw} (\Lambda)$
. The $\hat{\fg}^{\tw}$-module $L^{\tw} (\Lambda)$ is a highest weight module
with respect to the triangular decomposition of $\hat{\fg}^{\tw}$ induced from
that of $\hat{\fg}$ via the isomorphism $t_{-\xi}$. Its highest weight is  
 \begin{equation}     
 \label{eq:4.a}
\Lambda^{\tw}=t_\xi (\Lambda),
\end{equation}      
and its character and supercharacter are:
 \begin{equation}
\label{eq:4.b}
\ch^\pm_{L^{\tw}(\Lambda)}=t_\xi (\ch^\pm_{L(\Lambda)}),
\end{equation}      
their denominators being $\hat{R}^{\tw,\pm}$. The corresponding normalized
twisted character and supercharacter are given by
 \begin{equation}
\label{eq:4.14a}
 \ch^{\tw,\pm}_\Lambda = e^{-m^{\tw}_\Lambda \delta} \ch^\pm_{L^{\tw}(\Lambda)}\,, 
\end{equation}      
where
 \begin{equation}
\label{eq:4.c}
m^{\tw}_\Lambda = \frac{(\Lambda^{\tw} +\hat{\rho}^{\tw}|\Lambda^{\tw} + \hat{\rho}^{\tw} ) }{2 (K+h^\vee)}- \frac{\sdim \fg}{24}. 
\end{equation}

\begin{remark}
  \label{rem:4.3}
  \begin{list}{}{}
  \item (a)~~As in the affine Lie algebra case (see \cite{K2}, Chapter 12) 
we have, in view of (\ref{eq:4.2}) - (\ref{eq:4.5}), the
    following formula for the normalized character $\ch_\Lambda$
    in coordinates (\ref{eq:4.8}):

    \begin{equation}     
 \label{eq:4.13}
      \ch^+_\Lambda (h) = \ch_\Lambda (\tau, z , t) = e^{2\pi i Kt}
      \tr_{L (\Lambda)} q^{L_0 - \frac{1}{24} c(K)} e^{2\pi i z},
    \end{equation}
where $K$ is the level of $L (\Lambda)$, $L_0$ is the $0$\st{th}
mode of the Sugawara's Virasoro field $L(z) = \sum_{n \in \ZZ} L_n z^{-n-2}$, 
and $c(K)$, given by 
(\ref{eq:4.4} ), is its central
charge. A similar formula holds for $\ch^-_\Lambda$ by replacing
$\tr$ by $\str$.  
\item (b)
We have the twisted Sugawara's Virasoro field
$L^{\tw}(z) = \sum_{n \in \ZZ} L^{\tw}_n z^{-n-2}$, 
for which
we take $s_\alpha =-(\xi|\alpha)$, $\alpha \in \Delta_+$ (see
\cite{KW6}, Section~1).  Then 
$m^{\tw}_\Lambda =
h^{\tw}(\Lambda)- \frac{c(K)}{24}$, where $h^{\tw}(\Lambda)
= \frac{(\Lambda^{\tw} +2 \hat{\rho}^{\tw}|\Lambda^{\tw})}{2(K+h^\vee)}$ is 
the minimal eigenvalue of
$L^{\tw}_0$ in $L^{\tw }(\Lambda)$, and we have: 
%
%
%
\begin{equation}
  \label{eq:4.14}
  \ch^{\tw,+}_\Lambda (\tau ,z, t)=e^{2\pi i Kt}
  \tr_{L^{\tw}(\Lambda)} q^{L^{\tw}_0 -\frac{1}{24}c(K)} e^{2\pi i z},
\end{equation}
and similarly for $\ch^{\tw,-}_\Lambda$, and $\hat{R}^{\tw,\pm}$ are the
denominators of $\ch^{\tw,\pm}_\Lambda$.
  \end{list}
\end{remark}

Next, we introduce mock theta functions.  Having in mind
applications to affine Lie superalgebras, we use notation similar
to that above.

Let $\fh_\RR$ be an $\ell$-dimensional vector space over $\RR$
with a non-degenerate symmetric bilinear form $\bl$ (not
necessarily positive definite).  We shall identify $\fh_\RR$ with
$\fh^*_\RR$ using this bilinear form.  Let $L^\#$ be a positive
definite integral sublattice of $\fh^*_\RR$.  Let $\hat{\fh}_\RR
= \fh_\RR \oplus (\RR K \oplus \RR d)$ be the
$\ell+2$-dimensional vector space over $\RR$ with a symmetric
bilinear form $\bl$, which coincides with that on $\fh_\RR$, and
such that $\fh_\RR \perp (\RR K \oplus \RR d)$ and $\RR K \oplus
\RR d$ is the $2$-dimensional hyperbolic space, i.e. $(K |K) =
(d|d) =0$, $(K|d)=1$.  We identify $\hat{\fh}_{\RR}$ with
$\hat{\fh}^*_\RR = \fh^*_\RR \oplus\RR \delta \oplus \RR
  \Lambda_0$, using this bilinear form, so that $\fh_{\RR}$ gets
  identified with $\fh^*_\RR$, and $K$ (resp. $d$) with $\delta$
  (resp. $\Lambda_0$).  Given $\alpha \in \fh^*_\RR$, define the
  automorphism $t_\alpha$ of the vector space $\hat{\fh}^*_\RR$
  by formula (\ref{eq:2.10}).

Let $\Lambda\in\hat{\fh}^*_\RR $ be such that $\Lambda(K)$ is a positive 
integer, which, as before, we denote by $K$, 
and $(\Lambda|L^\#)\subset \ZZ $. Let $T \subset \fh^*_\RR$ 
be a finite set of
vectors that spans an isotropic subspace of $\fh^*_\RR$, and
such that 
$(T|L^\#)\subset \ZZ$ 
and $\bar{\Lambda} \perp T$, where, as before,
$\bar{\Lambda} \in \fh^*_\RR$ denotes the restriction of
$\Lambda$ to $\fh_\RR$.

\begin{definition}
  \label{def:4.4}
A {\em mock theta function} of degree $K$ is defined by the
following series:
\begin{equation}
\label{eq:4.15}
  \Theta^\pm_{\Lambda ,T} = e^{-\frac{(\Lambda | \Lambda)}{2K}\delta} 
\sum_{\alpha \in L^\#}\epsilon_\pm (t_\alpha) t_\alpha
  \frac{e^\Lambda}{\prod_{\beta \in T} (1\pm e^{-\beta})} \, .
\end{equation}
\end{definition}
It is not difficult to deduce from (\ref{eq:2.10}) that the series 
$\Theta^+_{\Lambda, T}$ (resp. 
$\Theta^-_{\Lambda, T}$ ) converges in the domain 
$$
X:=\{h \in \hat{\fh}  |\,\,\Re \delta (h) >0\} 
$$
to a meromorphic function with poles at the hyperplanes 
\begin{equation*}
\{ h \in \hat{\fh} |\,\, \beta (h) -(\alpha |\beta) \delta  (h) = 
(2n+1)\pi i \,(\hbox {resp.}  = 2n\pi i)\},
\end{equation*}
where
$\alpha \in L^\# , \beta \in T , n \in \ZZ$.  It is also clear that 
$\Theta^\pm_{\Lambda, T}$ depends only on $\Lambda\mod \CC \delta$. 
We shall call also a mock theta function an arbitrary linear combination
of functions of the form (\ref{eq:4.15}).

By (\ref{eq:2.10}), in coordinates (\ref{eq:4.8}) 
the mock theta function (\ref{eq:4.15})  looks as follows:
%
\begin{equation}
  \label{eq:4.16} 
  \Theta^\pm_{\Lambda ,T} (\tau ,z ,t) = e^{2\pi i K t}
  \sum_{\gamma \in L^\# + K^{-1}\bar{\Lambda}}\epsilon_\pm (t_\gamma)
 \frac{q^{\frac{K (\gamma |\gamma)}{2}}
  e^{2\pi i K \gamma (z)}}
{\prod_{\beta \in T} (1 \pm q^{-(\gamma |\beta) } e^{-2\pi i \beta  (z)})}.
\end{equation}
Recall that 
$\epsilon_+(t_\gamma)=1$
in all cases, and $\epsilon_-(t_\gamma)=1$ in all cases, considered in this paper.
\begin{remark}  
\label{rem:4.5}
Let $D$ be the Laplace operator, associated with the bilinear
form $\bl$.  Using that $D (e^\lambda) = (\lambda | \lambda)e^\lambda$, we
immediately see that $D (\Theta^\pm_{\Lambda ,T}) =0$.
Also, obviously, these functions are $t_\alpha$-invariant (rather 
anti-invariant) for all
$\alpha \in L^\#$, 
invariant under the translations $z\mapsto z+2\pi i\alpha$ 
for $\alpha \in L^\#$, 
and satisfy the degree property: $\Theta (h+aK)=
e^{Ka}\Theta(h),\, a\in \CC$.
It is known (\cite{K2}, Chapter 13) that these properties characterize
classical theta functions of degree $K$. It is an interesting problem
to find out what 
are the properties which, along with the above, characterize mock theta 
functions.
\end{remark}

In the same way as the normalized characters of integrable
modules over affine Lie algebras are rewritten in terms of the
theta functions (see \cite{K2}, Chapter~13), by Conjecture \ref{conj:3.7}, 
the normalized
characters and supercharacters of partially integrable tame modules
over affine Lie superalgebras can be rewritten in terms of mock
theta functions:
\begin{equation}  
\label{eq:4.17} 
 j_{\bar{\Lambda}} \hat{R}^\pm \ch^\pm_\Lambda = \sum_{w \in W} \epsilon_\pm (w) \Theta^\pm_{w (\Lambda +\hat{\rho}), w(T_{\bar{\Lambda}+\rho})}\, ,
\end{equation}
where $T_{\bar{\Lambda} +\rho}$ is a maximal subset in
$\Delta_{\bo, +}$, consisting of linearly independent pairwise orthogonal 
isotropic roots, which are orthogonal to $\bar{\Lambda}+\rho$.

\begin{remark}  
\label{rem:4.6}
Since 
$\ch_\Lambda^{\tw,\pm}=t_\xi (\ch^\pm_\Lambda)$, 
we obtain in coordinates 
(\ref{eq:4.8}):
$$
\ch_\Lambda^{\tw,\pm}(\tau, z, t)=\ch^\pm_\Lambda (\tau, z+\tau\xi, t+\frac{(z+\tau\xi|z+\tau\xi)-(z|z)}{2\tau}). 
$$
\end{remark}  

The main objective of our discussion in the next Section will be modular 
transformation properties of the {\it numerator} (cf. the RHS of 
(\ref{eq:4.17})):

\begin{equation}  
\label{eq:4.18} 
\Phi^\pm_{\Lambda, T} =\sum_{w \in W} \epsilon_\pm (w) \Theta^\pm_{w (\Lambda), w(T)}\, ,
\end{equation}
where $\Lambda\in\hat{\fh}^*_\RR $ is such that $\Lambda(K)$ is a positive 
integer (denoted as before by $K$), and $T \subset \hat{\fh}^*_\RR$ consists 
of pairwise orthogonal linearly independent isotropic roots, which are 
orthogonal to $\Lambda$, and $(T|L^\#)\subset \ZZ$ .

Recall the action of the group $SL_2 (\RR)$ in the domain 
\begin{equation*}
     X = \{(\tau,z,t)|\Im \tau >0 \} \subset \hat{\fh}^*
\end{equation*}
in coordinates (\ref{eq:4.8}):
\begin{equation}
\label{eq:4.6}
  \binom{a\, b}{c \,d} \cdot (\tau ,z,t) = \left(\frac{a\tau +b}{c\tau +d} \, , \, \frac{z}{c\tau +d}\, , \,t-\frac{c(z|z)}{2(c\tau +d)}  \right)\, ,
\end{equation}
and the action of the corresponding metaplectic group $Mp_2
(\RR)$ on the space of meromorphic functions on $X$:
\begin{equation}
\label{eq:4.6a}
  F|_A (\tau ,z, t) = (c\tau +d)^{-\frac{\ell}{2}} F (A \cdot
  (\tau ,z,t))
\end{equation}
(see \cite{K2}, Chapter 13 for details).  
Everywhere in the paper the square root of a complex number $a=r e^{i\theta}$,
where $r\geq 0$ and $-\pi <\theta <\pi$ is, as usual, chosen to be
$a^{1/2}=r^{1/2}e^{i\theta/2}$.

Let's consider the mock theta functions of the form $\Theta^-_{\Lambda
,\beta}$, for which the set $T$ consists of one isotropic vector
$\beta$.  Choose a basis $\rmv_1,\rmv_2, \ldots$ of $\fh$, such
that $(\beta|\rmv_j)=\delta_{1j}$.  Then coordinates
(\ref{eq:4.8}) on $\hat{\fh}$ become $h=2\pi i (-\tau d + \sum_j
 z_j \rmv_j +tK)$, and  we have:
\begin{equation*}
  \Theta^-_{\Lambda ,\beta} (h)=e^{2\pi i Kt} \sum_{\alpha \in  L^\#}  
  \frac{e^{2\pi i \left(
\sum_{j\geq 1} z_j (\Lambda +K\alpha |\rmv_j)+\tau
      (\frac{K}{2} (\alpha |\alpha)+(\Lambda
      |\alpha))\right)}}{1-e^{-2\pi i (z_1 + (\alpha |\beta) \tau)}}\, .
\end{equation*}
Hence for each fixed $\tau$, $\Im \tau >0$, this function has
only simple poles and they occur at the hyperplanes $z_1 = n -
(\gamma |\beta) \tau$ $(n \in \ZZ\, ,\,  \gamma \in L^\#)$.  For
$\gamma \in L^\#$ let $L_\gamma = \{ \alpha \in L^\# |\, (\alpha
|\beta) = (\gamma |\beta) \}$.  It is immediate to compute the
residue:

\begin{eqnarray}
\label{eq:4.18a}
  \Res_{z_1=n-(\gamma |\beta )\tau} \Theta^-_{\Lambda ,\beta} (h)
  = \frac{e^{2\pi i Kt}}{2\pi i} \sum_{\alpha \in L_\gamma}
  e^{2\pi i  \left(n(\Lambda +K\alpha | \rmv_1) +\sum_{j \geq 2}
      z_j (\Lambda +   K\alpha |\rmv_j)\right)}\\
   \times e^{2\pi i \tau \left(\frac{K}{2}(\alpha |\alpha) + (\Lambda|\alpha) -(\gamma |\beta) (\Lambda + K\alpha|\rmv_1)\right)}\, .\nonumber
\end{eqnarray}
On the other hand, for each fixed $\tau,\, \Im \tau >0$, the function,
transformed by $S=\binom{0\, -1}{1\,\,\,\,\, 0}$, looks as follows:
\begin{eqnarray*}
  (\Theta^-_{\Lambda ,\beta}|_S)(h) = \tau^{-\frac{\ell}{2}}
  \Theta^-_{\Lambda ,\beta} (S \cdot h) &=& \tau^{-\frac{\ell}{2}}
  e^{2\pi i K \left(t-\frac{\| \sum_{j\geq 1} z_j\rmv_j \|^2
    }{2\tau}\right)}\\
 && \times \sum_{\alpha \in L^\#} 
\frac{e^{{\frac{2\pi i}\tau}\left(\sum_{j\geq 1}z_j(\Lambda +K\alpha |\rmv_j) 
- \frac{K}{2}(\alpha|\alpha)- (\Lambda |\alpha)\right)}}
{1-e^{-\frac{2\pi i}{\tau}(z_1-(\alpha |\beta))}} \, .
\end{eqnarray*}
For each fixed $\tau$, this function has only simple poles as
well, and they occur at the hyperplanes $z_1 = (\gamma |\beta)
+n\tau\,\, (n \in \ZZ, \gamma \in L^\#)$.  The corresponding residue
is
\begin{eqnarray}
\label{eq:4.19}
  \Res_{z_1= (\gamma |\beta) +n\tau} (\Theta^-_{\Lambda  ,j}|_S)(h)= \frac{\tau^{1-\frac{\ell}{2}} e^{2\pi iKt}}{2\pi i}   e^{-\frac{\pi K i }{\tau} \|
       ((\gamma |\beta) +n\tau) \rmv_1+ \sum_{j \geq 2} z_j\rmv_j\|^2} \\  
\times 
 \sum_{\alpha \in L_\gamma} e^{2\pi i n (\Lambda +K\alpha |\rmv_1)} 
    e^{\frac{2\pi i}{\tau}\left((\gamma |\beta) (\Lambda +K\alpha
      |\rmv_1)+\sum_{j \geq 2} z_j (\Lambda + K\alpha |\rmv_j)-
      (\frac{K}{2} (\alpha |\alpha) + (\Lambda|\alpha))\right) }\,.\nonumber
\end{eqnarray}

\begin{proposition}
  \label{prop:4.7}
Let $\dim \fh =2$ and let $\Lambda \in \hat{\fh}^*$,   
$\alpha , \beta \in \fh^*$ be such that 
$\Lambda (K)$ is a positive integer, 
$(\Lambda |\alpha) \in \ZZ$, $(\Lambda |\beta)=0$,
$\Lambda (K)(\alpha | \alpha)$ is a positive integer, 
$(\alpha |\beta) =1$, $(\beta |\beta)
=0$.  Let $L^\#=\ZZ \alpha$. 
Then the function
$G=\Theta^-_{\Lambda ,\beta} - \Theta^-_{\Lambda, \beta} |_S$ is
holomorphic in the domain $X$.
\end{proposition}
\begin{proof}
The function  $G$ is holomorphic in the domain $X$ if and only if 
$\Theta^-_{\Lambda ,\beta}$ and $\Theta^-_{\Lambda, \beta} |_S$ have the 
same poles and equal residues at each pole. For
$\gamma\in L^\#$ we have: $L_\gamma = \{ \gamma\}$, hence formulae 
(\ref{eq:4.18a}) and 
(\ref{eq:4.19}) become:
$$
 \Res_{z_1=n-(\gamma |\beta )\tau} \Theta^-_{\Lambda ,\beta} (h) = 
\frac{e^{2\pi i Kt}}{2\pi i}  e^{2\pi i  \left(n(\Lambda +K\gamma | \rmv_1) +
z_2 (\Lambda +   K\gamma |\rmv_2)\right)} q^{\frac{K}{2}
(\gamma |\gamma)+ (\Lambda | \gamma) -(\gamma |\beta) 
(\Lambda + K\gamma|\rmv_1)}\, ,
$$
and
\begin{eqnarray*}
 \Res_{z_1= (\gamma |\beta) +n\tau} (\Theta^-_{\Lambda  ,j}|_S)(h)= 
\frac{e^{2\pi iKt}}{2\pi i}   
e^{-\frac{\pi K i }{\tau}|((\gamma|\beta)+n\tau)\rmv_1 +z_2\rmv_2|^2} 
e^{2\pi in(\Lambda+K\gamma|\rmv_1)}\\
\times e^{-\frac{\pi i}{\tau}((\gamma|\beta)(\Lambda+K\gamma|\rmv_1)+
z_2(\Lambda+K\gamma|\rmv_2)-\frac{K}{2}(\gamma|\gamma)
-(\Lambda|\gamma))}.
\end{eqnarray*} 
Let $\rmv_1=\alpha$, $\rmv_2=\beta$, 
$\gamma=a\alpha$, $a\in \ZZ$. Then the above formulae become:
$$
 \Res_{z_1=n-a\tau} \Theta^-_{\Lambda ,\beta} (h) = 
\frac{e^{2\pi i Kt}}{2\pi i} e^{2\pi iKa z_2}q^{-\frac{Ka^2}{2}(\alpha|\alpha)}
;\,\,
 \Res_{z_1=a+n\tau} \Theta^-_{\Lambda ,\beta}|_S (h) =
\frac{e^{2\pi i Kt}}{2\pi i} e^{-2\pi iKnz_2}q^{-\frac{Kn^2}{2}(\alpha|\alpha)}.
$$
Hence the residues at all poles of the function $G$ are zero. 
\end{proof}

\section {Transformation properties of the mock theta functions
  $\Phi^{[m]}$ and their modifications $\tilde{\Phi}^{[m]}$}
\label{sec:5}

Let $\fg = s\ell_{2|1}$ with a structure of a Kac--Moody algebra
as in Example \ref{ex:3.3}, and let $\hat{\fg}$ be the
corresponding affine superalgebra.  We have
\begin{eqnarray*} 
& &  \sdim \fg =0 \, , \, \ell =2\, , \, \Pi = \Delta_{\bo ,+}=
\{\alpha_1,\alpha_2\}, \,\Delta_{\bar{0},+}=\{\theta=\alpha_1+\alpha_2\}, \\
& & (\alpha_i|\alpha_i)=0, \,  (\alpha_1|\alpha_2)=1\, 
, \, (\theta |\theta)=2 \, , \,
  \rho_{\bar0}=\rho_{\bo} =\frac12 \theta \, , \, \rho =0 \,, h^{\vee}=1 . 
\end{eqnarray*}
We introduce the following coordinates in the Cartan subalgebra
$\hat{\fh}$ of $\hat{\fg}$ (cf. (\ref{eq:4.8})):
\begin{equation}
  \label{eq:5.1}
  h=2\pi i (-\tau \Lambda_0 -z_1 \alpha_2 -z_2 \alpha_1 + t
  \delta)= 2\pi i (-\tau \Lambda_0 + u (\alpha_1 +\alpha_2) +
  v (\alpha_1 - \alpha_2) + t \delta)\, .
\end{equation}

In this section we shall study transformation properties of the
numerator of the normalized supercharacter of the integrable
$\hat{\fg}$-module $L (m\Lambda_0)$, where $m$ is a non-negative
integer (see Example \ref{ex:3.3}), 
using formula (\ref{eq:3.15a}) for $\ch^-_{L
  (m\Lambda_0)}$.  We have: $\hat{\rho}=\Lambda_0$,
$m_{m\Lambda_0} =0$, hence $\ch^-_{L(m\Lambda_0)} =
\ch^-_{m\Lambda_0}$.  We choose in this formula $T_0 = \{
\alpha_1\}$. Then $j_0 = 1$ and (\ref{eq:3.15a}) gives:
\begin{equation}
  \label{eq:5.2}
  \hat{R}^- \ch^-_{m\Lambda_0} = \sum_{w \in \hat{W}^\#} \epsilon
  (w) w \frac{e^{(m+1)\Lambda_0}}{1-e^{-\alpha_1}}\, .
\end{equation}
Since $L^\# = \ZZ\theta$ and $W =\{ 1, r_{\alpha_1 + \alpha_2}\}$,
due to (\ref{eq:3.4}) and (\ref{eq:2.10}),
formula (\ref{eq:5.2}) in coordinates (\ref{eq:5.1}) looks as
follows:
\begin{equation}  
\label{eq:5.3}
  (\hat{R}^- \ch^-_{m \Lambda_0}) (\tau, z_1,z_2,t) = e^{2\pi i
    (m+1)t} \sum_{j \in \ZZ}
   \left( \frac{e^{2\pi i j (m+1) (z_1+z_2)} q^{j^2 (m+1)}}
    {1-e^{2\pi i z_1} q^j} - 
\frac{e^{-2\pi i j(m+1)(z_1+z_2)}q^{j^2 (m+1)}}{1-e^{-2\pi i
    z_2} q^j}\right).
\end{equation}
We denote the right-hand side of this formula by 
\begin{equation}
\label{eq:5.4}
\Phi^{[m]} (\tau ,z_1,z_2,t) = \varphi^{[m]}(\tau, u, v ,t)\,,
\end{equation}
where
$u=-\frac12 (z_1+z_2)$, $v =\frac12 (z_1-z_2)$.  This is the
{\em numerator} of $\ch^-_{m\Lambda_0}$.  
%

The main properties of the functions $\Phi^{[m]}$ are described
by the following lemma.

\begin{lemma}
  \label{lem:5.1}
  \begin{list}{}{}
  \item (a)~~$\Phi^{[m]} (\tau, z_1 +a, z_2 + b,t) = \Phi^{[m]}
    (\tau, z_1,z_2,t)$ for all $a,b \in \ZZ$.

\item (b)~~$\Phi^{[m]} (\tau, -z_1,-z_2,t)= -\Phi^{[m]}
  (\tau,z_1,z_2,t)$.

\item (c)~~$\Phi^{[m]} (\tau, z_2,z_1,t) = \Phi^{[m]} (\tau, z_1,
  z_2, t)$.

\item  (d)~~ $\Phi^{[m]} (\tau, z_1 +\tau, z_2 + \tau, t) 
     = q^{-(m+1)} e^{-2\pi i (m+1)(z_1+z_2)} \Phi^{[m]} (\tau,
       z_1,z_2,t)$.

\item (e)~~$\Phi^{[m]} (\tau, z_1,z_2,t) -e^{2\pi i (m+1)z_1}
  \Phi^{[m]} (\tau,z_1,z_2+\tau ,t) $\\ 
$=\sum^{m-1}_{j=0} e^{\pi i (j+1) (z_1-z_2)} q^{-\frac{(j+1)^2}{4 (m+1)}}$
 $ \left(\Theta_{j+1,m+1} (\tau, z_1+z_2,t) -\Theta_{-(j+1),
    m+1} (\tau, z_1+z_2,t)\right)$, \\ 
where $\Theta_{j,m} (\tau,z,t) =e^{2\pi i mt} \Theta_{j,m} (\tau,z)$, and
$\Theta_{j,m} (\tau,z)$ is given by (\ref{eq:A2}).
  \end{list}
\end{lemma}

\begin{proof}
  Without loss of generality we may assume that $t=0$.  Property~(a) is
  obvious, while property (b) follows easily from (c). Property (c) follows 
from the expansion of $\Phi^{[m]}$ from formula (8.3) in \cite{KW3}:
 $$
\Phi^{[m]} (\tau ,z_1,z_2,0) = \left(\sum_{j,k\geq 0,\, \min(j,k)|m+1} - 
\sum_{j,k <0,\, \max(j,k)|
m+1}\right)
 e^{2\pi i (jz_1+kz_2)}q^{\frac{jk}{m+1}}.
$$

Next, we prove (e).  We have:
\begin{eqnarray*}
  \Phi^{[m]} (\tau,z_1,z_2,0) &=& 
\sum_{j \in \ZZ} \frac{e^{2\pi i j (m+1) (z_1+z_2)} q^{j^2(m+1)}}
{1-e^{2\pi i z_1} q^j} - \sum_{j \in \ZZ} \frac{e^{-2\pi i j (m+1) (z_1+z_2)} q^{j^2(m+1)}}{1-e^{-2\pi i z_2} q^j}\, 
, \\
\noalign{\hbox{and}}\\
\Phi^{[m]} (\tau, z_1,z_2+\tau ,0) &=& e^{-2\pi i (m+1)z_1}
\Bigg(\sum_{j \in \ZZ} \frac{e^{2\pi i j (m+1) (z_1+z_2)} q^{j^2
      (m+1)}(e^{2\pi i z_1} q^j)^{m+1}}{1-e^{2\pi i z_1}q^j}  \\
 & &-\sum_{j \in \ZZ} \frac{e^{-2\pi i j (m+1) (z_1+z_2)} q^{j^2
     (m+1)} (e^{-2\pi i z_2} q^j)^{m+1}}
{1-e^{-2\pi i z_2} q^j}\, \Bigg) .
\end{eqnarray*}
Hence we have:
\begin{eqnarray*}
\lefteqn{\hspace{-2ex}\Phi^{[m]} (\tau ,z_1,z_2,0) -e^{2\pi i (m+1)  z_1}
  \Phi^{[m]}(\tau, z_1,z_2+\tau,0)}\\
&=& \sum_{j \in \ZZ} e^{2\pi i j (m+1) (z_1+z_2)} q^{j^2(m+1)} 
    \frac{1-(e^{2\pi i z_1} q^j)^{m+1}}{1-e^{2\pi i z_1} q^j}\\
&&- \sum_{j \in \ZZ} e^{-2\pi i j (m+1) (z_1+z_2)} q^{j^2(m+1)}
  \frac{1-(e^{-2\pi i z_2} q^j)^{m+1}}{1-e^{-2\pi i z_2} q^j}\\
&=&\sum^m_{k=0}\sum_{j \in \ZZ} e^{2\pi i j (m+1)}
   q^{j^2 (m+1)} (e^{2\pi i z_1} q^j)^k
- \sum^m _{k=0}\sum_{j \in \ZZ}
    e^{-2\pi i j (m+1)(z_1+z_2)} q^{j^2 (m+1)} (e^{-2\pi i z_2} q^j)^k\\
&=&  \sum^m_{k=0} e^{\pi i k (z_1-z_2)} 
    q^{-\frac{k^2}{4(m+1)}} \\
& &
\times \left( \sum_{j \in \ZZ} e^{2\pi i (m+1)(z_1+z_2)}
    q^{(m+1)\left(j+\frac{k}{2(m+1)}\right)^2} 
 - \sum_{j \in \ZZ} e^{2 \pi i (m+1) (z_1+z_2) (j-\frac{k}{2(m+1)})}
     q^{(m+1) \left(j-\frac{k}{2(m+1)}\right)^2} \right) \\
&=& \sum^m_{k=1} e^{\pi i k (z_1-z_2)} 
q^{-\frac{k^2}{4(m+1)}}\left(\Theta_{k,m+1}
(\tau,z_1+z_2)-\Theta_{-k,m+1} (\tau, z_1+z_2)\right)\, .
\end{eqnarray*}

Finally, we prove (d).  Exchanging $z_1$ and $z_2$ in (e) and
using (c), we obtain:
\begin{eqnarray*}
  \Phi^{[m]} (\tau,z_1,z_2,0) -e^{2\pi i (m+2)} \Phi^{[m]}
  (\tau,z_1+\tau, z_2,0)\\
= \sum^m_{j=1} e^{\pi i j (z_2-z_1)} q^{-\frac{j}{4(m+1)}}
\left( \Theta_{j,m+1}-\Theta_{-j,m+1}\right)
(\tau, z_1+z_2).
\end{eqnarray*}
Replacing $z_2$ by $z_2+\tau$, we deduce:
\begin{eqnarray*}
\lefteqn{\hspace{-2ex}  \Phi^{[m]} (\tau,z_1,z_2+\tau ,0)-e^{2\pi i (m+1)z_2}
  q^{m+1}\Phi^{[m]} (\tau, z_1+\tau, z_2 + \tau,0)}\\
&=& \sum^m_{j=1} e^{\pi i j (z_2-z_1)}q^{\frac{j}{2}-\frac{j^2}{4(m+1)}}
\left( \Theta_{j,m+1}-
\Theta_{-j,m+1}\right) (\tau, z_1+z_2+\tau)\\
&=& \sum^m_{j=1} e^{\pi ij (z_2-z_1)} e^{-\pi i (m+1) (z_1+z_2)} 
q^{-\frac{(m+1-j)^2}{4 (m+1)}}\left( \Theta_{-(m+1-j), m+1} -
  \Theta_{m+1-j,m+1}\right) (\tau, z_1+z_2)\, .
\end{eqnarray*}
Replacing in the summation $j$ by $m+1-j$, and multiplying both
sides by $e^{2\pi i (m+1) z_1}$, we obtain:
\begin{eqnarray*}
  e^{2\pi i (m+1)z_1} \Phi^{[m]} (\tau, z_1,z_2+\tau,0) -e^{2\pi
    i (m+1) (z_1+z_2)} q^{m+1} \Phi^{[m]} (\tau, z_1 + \tau,
  z_2+\tau,0)\\
=-\sum^m_{j=1} e^{\pi ij (z_1-z_2)} q^{-\frac{j^2}{4(m+1)}}\left(
  \Theta_{j,m+1} -\Theta_{-j,m+1} \right) (\tau, z_1 + z_2)\, .
\end{eqnarray*}
Adding this equality to (e), we obtain (d).

\end{proof}

\begin{remark}
  \label{rem:5.2}

Properties (c) and (e) of $\Phi^{[m]}$ have a simple
representation theoretical meaning (and proof).  Property (c)
means that $\ch^-_{L(m\Lambda_0)}$ is unchanged under the flip of
the Dynkin diagram of $\hat{\fg}$.  Property (e) follows from the
fact that the odd reflection with respect $\alpha_1$ maps
$m\Lambda_0$ to $m\Lambda_2 =m(\Lambda_0+\alpha_1)$ (and doesn't
change the supercharacter), hence we have:
\begin{eqnarray*}
  \sum_{w \in \hat{W}} \epsilon (w) w\frac{e^{m\Lambda_0}}{1-e^{-\alpha_1}}-
  \sum_{w \in \hat{W}} \epsilon (w)
 w \frac{e^{m(\Lambda_0+\alpha_1)}}{1-e^{-\alpha_1}}= -\sum_{w \in \hat{W}}
  \epsilon (w) w (e^{m (\Lambda_0+\alpha_1)} \sum^{m-1}_{j=0}
  e^{-j \alpha_1})\, .
\end{eqnarray*}

\end{remark}

Lemma \ref{lem:5.1} on properties of the functions $\Phi^{[m]}$
immediately implies the following lemma on properties of the
functions $\varphi^{[m]}$.

\begin{lemma}
  \label{lem:5.3}

  \begin{list}{}{}
  \item (a)  $\varphi^{[m]} (\tau, u+a, v +b, t) =
    \varphi^{[m]} (\tau ,u, v,t)$ if $a,b \in \frac12 \ZZ$,  $a+b \in \ZZ$.

\item (b)  $\varphi^{[m]} (\tau,-u,v,t) =- \varphi^{[m]} (\tau  ,u,v,t)$.

\item (c)  $\varphi^{[m]} (\tau,u,-v,t) = \varphi^{[m]}(\tau ,u,v,t)$.

\item (d) $\varphi^{[m]} (\tau ,u+\tau,v,t) =q^{-(m+1)} e^{-4\pi i
    (m+1)u} \varphi^{[m]} (\tau,u,v,t)$,\\ $\varphi^{[m]} (\tau
  ,u-\tau,v,t) = q^{-(m+1)} e^{4\pi i (m+1)u}  \varphi^{[m]} (\tau
  ,u,v,t)$.

\item (e) $\varphi^{[m]}(\tau , u,v,t) - e^{2\pi i (m+1) (v-u)} 
  \varphi^{[m]} (\tau ,u-\frac{\tau}{2}, v-\frac{\tau}{2}, t) $\\
$  =-e^{2\pi i (m+1)t} \sum^m_{j=1} e^{2\pi ijv}
q^{-\frac{j^2}{4m+4}} (\Theta_{j,m+1} -\Theta_{-j,m+1}) (\tau ,2u)$.

  \end{list}
\end{lemma}

\begin{proof}
  
Properties (a), (b), (c), (d), and (e) of $\Phi^{[m]}$
immediately translate into properties (a), (b), (c), the second
formula in (d), and (e) of $\varphi^{[m]}$.  The first formula
in (d) is obtained from the second one by replacing $u$ by
$u+\tau$.

\end{proof}

Note that formulae (d) of Lemma~\ref{lem:5.3} imply the following
version of property~(e):
\begin{eqnarray}
  \label{eq:5.5}
  \varphi^{[m]} (\tau ,u,v,t)-e^{2\pi i
    (m+1)(2v-\tau)}\varphi^{[m]} (\tau,u,v-\tau,t)\\
\nonumber   = -e^{2\pi i (m+1)t}\sum^{2m+1}_{j=1}
e^{2\pi ijv}  q^{-\frac{j^2}{4m+4}}(\Theta_{j,m+1}-\Theta_{-j,m+1})(\tau, 2u)\, .
\end{eqnarray}

Recall that in coordinates $\tau,u,v,t$ we have:
\begin{equation*}
  \varphi^{[m]} = \hat{R}^- \ch^-_{m\Lambda_0} = \Theta^-_{m\Lambda_0,\alpha_1}
  -r_{\theta}\Theta^-_{m\Lambda_0,\alpha_1}\, , 
\end{equation*}
and  note that the function $\Theta^-_{m\Lambda_0,\alpha_1} $
satisfies all conditions of Proposition~\ref{prop:4.7} (with
$\alpha = \theta$).  Hence the
function $\Theta^-_{m\Lambda_0,\alpha_1} -\Theta^-_{m\Lambda_0,\alpha_1}|_S$ 
is holomorphic in the domain
$X$.  Since the action of the Weyl group on $X$ commutes with the
action of $SL_2(\RR)$, the same holds for $r_\theta
\Theta^-_{m\Lambda_0,\alpha_1}-r_\theta \Theta^-_{m\Lambda_0,\alpha_1} |_S $.
Consequently, the function
\begin{equation}
  \label{eq:5.6}
  G (\tau ,u,v,t) = \varphi^{[m]} (\tau
  ,u,v,t)-\tau^{-1}\varphi^{[m]} \left(-\frac{1}{\tau}, \frac{u}{\tau}, 
\frac{v}{\tau} ,t-\frac{u^2-v^2}{\tau}\right)
\end{equation}
is holomorphic in the domain $X$.

\begin{proposition}
  \label{prop:5.4}

  \begin{list}{}{}
  \item (a)  $G (\tau,u,v+1,t) -G (\tau,u,v,t)$\\
$ = e^{2\pi i    (m+1)t} \sqrt{\frac{2}{m+1} }(-i\tau)^{-\frac12}
    \sum^{2m+1}_{j=1}  \sum_{k \in \ZZ / (2m+2)\ZZ} e^{\frac{2\pi
      i (m+1)}{\tau} (v+\frac{j}{2m+2})^2} \sin \frac{\pi  jk}{m+1} \Theta_{k,m+1} (\tau,2u)$.

\item (b) $G (\tau,u,v,t) -e^{2\pi i (m+1) (2v-\tau)}
  G(\tau,u,v-\tau,t)$\\
   $= -e^{2\pi i (m+1)t} \sum^{2m+1}_{j=1} e^{2\pi ijv}
   e^{-\frac{\pi i j^2\tau}{2m+2}} [\Theta_{j,m+1} -
   \Theta_{-j,m+1}](\tau, 2u)$

\item (c)  The holomorphic function $G$ is determined uniquely by
  the above properties (a) and~(b).

  \end{list}

\end{proposition}

\begin{proof}
  Without loss of generality we may put $t=0$.  In order to prove
  (a), first recall the well-known transformation formula for
  theta functions (cf. formula (\ref{eq:A4}) in the Appendix): 
  \begin{equation}
    \label{eq:5.7}
    \Theta_{j,m} \left( -\frac{1}{\tau} \, , \, \frac{z}{\tau}\right)
      =e^{\frac{\pi i mz^2}{2\tau}} \left( -\frac{i\tau}{2m}
      \right)^\frac12  \sum_{n \in \ZZ\!\!\!\! \mod 2m\ZZ} e^{-\frac{\pi
        i nj}{m}} \Theta_{n,m} (\tau,z)\, .
  \end{equation}
Using (\ref{eq:5.5}) and (\ref{eq:5.7}) we obtain:
\begin{eqnarray*}
  \varphi^{[m]} \left( -\frac{1}{\tau} \, ,\, \frac{u}{\tau}
    \,,\, \frac{v}{\tau} \,  ,\, 0 \right)
 -e^{\frac{2\pi i}{\tau} (2v+1)} \varphi^{[m]} \left(
   -\frac{1}{\tau} \, ,\, \frac{u}{\tau}
    \,,\, \frac{v+1}{\tau} \,  ,\, 0 \right)\\
=i \left( \frac{-2i\tau}{m+1} \right)^\frac12    
   e^{\frac{2\pi  i (m+1)u^2}{\tau}} \sum^{2m+1}_{j=1}\sum_{k \in \ZZ
    /(2m+2)\ZZ} 
e^{\frac{2\pi ijv}{\tau}}
 e^{\frac{2\pi i j^2}{(4m+4)\tau}} \sin\frac{\pi jk}{m+1}
  \Theta_{k,m+1} (\tau,2u)\, .
\end{eqnarray*}
We deduce (a) from this by a straightforward calculation. 

 In order to prove (b), by a straightforward calculation we get:
\begin{equation*}
  G (\tau ,u,v,0) -e^{2\pi i (m+1)} G (\tau ,u, v-\tau,0) = 
   \varphi^{[m]} (\tau ,u, v,0)-e^{2\pi i (m+1) (2v-\tau)}
     \varphi^{[m]} (\tau,u,v-\tau ,0)\, .
\end{equation*}
By (\ref{eq:5.5}), the RHS of this equation is equal to the RHS
of (b).

In order to prove (c), note that the difference, say, $F
(\tau,u, v,t)$ of two holomorphic functions, satisfying (a) and
(b), satisfies the following two equations:
\begin{eqnarray*}
  F (\tau ,u,v+1,t) = F(\tau,u,v,t)\, ,\,\,\, 
  F (\tau,u,v-\tau,t) = e^{-4\pi i (m+1)} F (\tau,u, v,t)\,.
\end{eqnarray*}
Consider the function
\begin{equation*}
  P(\tau,u,v,t) = F (\tau,u, v,t) \vartheta_{11} ((m+1)\tau
  , (m+1)v)^2\, .
\end{equation*}
Since, by Proposition \ref{prop:A6} from the Appendix,
\begin{equation*}
  \vartheta_{11} ((m+1)\tau,(m+1)(v-\tau) )
   =-e^{2\pi i (m+1)v} q^{-\frac{m+1}{2}} 
   \vartheta_{11} ((m+1)\tau, (m+1)v)\, ,
\end{equation*}
we deduce that
\begin{equation*}
  P (\tau,u,v+1,t) = P (\tau,u, v, t)\, , \,\,\, 
  P (\tau, u,v-\tau,t) = P(\tau ,u,v,t)\, .
\end{equation*}
Since $P$ is a holomorphic function in $v$, which is doubly
periodic (for each fixed value of $\tau$, $u$ and $t$), we
conclude that $P$ is constant in $v$.  Since $\vartheta_{1,1}
(\tau,0)=0$ (see formula (\ref{eq:4.7})), we conclude that $P$ is
identically zero.

\end{proof}

Now we relate the function $G$ to the functions $h_\ell$,
introduced by Zwegers in \cite{Z}, page 51; we will denote
$h_\ell$ by $h_{m;\ell}$ to emphasize its dependence on $m$.
Replacing $x$ by $x+i$ in Zwegers' formula, it is straightforward
to obtain a slightly different formula:
\begin{equation}
  \label{eq:5.8}
  h_{m;j}(\tau,v) =ie^{-\frac{\pi i \tau}{2m}(2m-j)^2 + 2\pi i (2m-j)v}
   \int_{\RR +is}\frac{e^{2\pi i m\tau x^2 + 2\pi (2m-j)\tau x-4\pi
         m\sqrt{x}}}{1-e^{2\pi x}} \, dx \, , 
\end{equation}

where $s \in \RR$, $0<s<1$.

\begin{theorem}
  \label{th:5.5}
$G (\tau,u,v,t) =-e^{2\pi i (m+1)t} \sum^{2m+1}_{j=1}
h_{m+1;2m+2-j} (\tau,v) (\Theta_{j,m+1}- \Theta_{-j,m+1}) (\tau,2u)$.
\end{theorem}

\begin{proof}
Due to uniqueness of a holomorphic function on $X$,
  satisfying properties (a) and (b) of
  Proposition~\ref{prop:5.4}, it suffices to show that the RHS
  satisfies these two properties.

Let $a_j(\tau,v)=h_{m+1;2m+2-j} (\tau, v)$ to simplify
notation, $1 \leq j \leq 2m+1$.
Replacing $x$ by $x+i\frac{j}{2m+2}$ in (\ref{eq:5.8}) and taking
$s=\frac{j}{2m+2}$, we obtain a simpler expression:
\begin{equation}
  \label{eq:5.9}
  a_j (t,v) =i \int_{\RR} \frac{e^{2\pi i (m+1)x^2\tau - 4\pi
      (m+1)xv}}{1-e^{2\pi (x+\frac{ij}{2m+2})}} \, dx \, .
\end{equation}
Property (a) of Proposition~\ref{prop:5.4} of the function
\begin{equation*}
  G (\tau,u,v ,t)=e^{2\pi it} \sum^{2m-1}_{j=1}
  a_j(\tau, v) (\Theta_{j,m+1}- \Theta_{-j,m+1})(\tau ,2u)
\end{equation*}
is equivalent to the following property of the functions $a_j$:
\begin{equation}
  \label{eq:5.10}
  a_j (\tau,v+1)-a_j (\tau,v)=\frac{i}{\sqrt{2m+2}}  (-i\tau)^{-\frac12}
    \sum^{2m-1}_{k=1} e^{-\frac{\pi ijk}{m+1}} e^{\frac{2\pi
        i (m+1)}{\tau} (v+\frac{k}{2m+2})^2}\, .
\end{equation}
Equation (\ref{eq:5.10}) is established as follows.  Using the
  expression (\ref{eq:5.9}) for $a_j (\tau ,v)$, we obtain:
  \begin{eqnarray*}
    a_j (\tau,v+1) - a_j (\tau ,v) = i \int_{\RR} e^{2\pi i
      (m+1)x^2 \tau - 4\pi (m+1)xv} R_j(x)dx \, , 
  \end{eqnarray*}
where
\begin{eqnarray*}
  R_j(x) &=& \frac{e^{-4\pi (m+1)x}-1}{1-e^{2\pi i
      (x+\frac{ij}{2m+2})}} =e^{-4\pi (m+1)x} \frac{1-e^{4\pi
      (m+1)(x+\frac{ij}{2m+2})}}{1-e^{2\pi
      (x+\frac{ij}{2m+2})}}\\
&=& e^{-4\pi (m+1)x} \sum^{2m+1}_{k=0} e^{2\pi k (x+
  \frac{ij}{2m+2})}\, ,
\end{eqnarray*}
which, replacing $k$ by $2m+2-k$, is equal to $  \sum^{2m+2}_{k=1} e^{-2\pi k x} e^{-\frac{\pi ijk}{m+1}}$.
Thus,
\begin{eqnarray*}
 a_j (\tau, v+1)- a_j (\tau ,v)
= i \sum^{2m+2}_{k=1} e^{-\frac{\pi ijk}{m+1}} \int_{\RR}
     e^{2\pi i (m+1) x^2 \tau - 4\pi (m+1) xv-2\pi kx}\, dx\\
  = i \sum^{2m+2}_{k=1}  e^{-\frac{\pi ijk}{m+1}}e^{\frac{2\pi i
      (m+1)}{\tau}(v+\frac{k}{2m+2})} \int_\RR e^{2\pi i (m+1)\tau \left( x+i
         \frac{(2m+2)v+k}{(2m+2)\tau}\right)^2}\, dx   \, . 
\end{eqnarray*}
We compute the integral in the above expression using the
formula
\begin{equation*}
  \int_{\RR +ib} e^{-ax^2}\, dx = \left(\frac{\pi}{a} \right)^{1/2}
    \hbox{\,\, if \,\,} \Re a>0\, , \quad b \in \RR\, , 
\end{equation*}
to obtain (\ref{eq:5.10}).

In order to establish property (b) of Proposition \ref{prop:5.4},
we need to prove
\begin{equation}
  \label{eq:5.11}
  a_j (\tau ,v)-e^{2\pi ijv (m+1)(2v-\tau)} a_j (\tau ,v) =
  -e^{2\pi i jv} e^{-\frac{\pi i j^2}{2m+2}}\, .
\end{equation}
By (\ref{eq:5.8}) we have:
\begin{equation}
  \label{eq:5.12}
  a_j (\tau,v) =i \int_{\RR +is} \frac{P (\tau ,v,x)}{1-e^{2\pi
      x}}\, dx \quad (0<s<1)\, , 
\end{equation}
where
\begin{equation*}
  P(\tau,v,x)=e^{\frac{-\pi ij^2}{2m+2} \tau + 2\pi i jv} e^{2\pi i
    (m+1) \tau x^2 - 4\pi (m+1) vx +2\pi j \tau x}\, .
\end{equation*}
The function $P$ satisfies the identity
\begin{equation*}
  e^{2\pi i (m+1) (2\tau -v)} P (\tau,v-\tau ,x)= P (\tau ,v,x-i)\,.
\end{equation*}
Hence
\begin{eqnarray*}
  e^{2\pi i (m+1) (2v-\tau)} a_j (\tau ,v-\tau)
 =i \int_{\RR+is} \frac{P(\tau,v,x-i)}{1-e^{2\pi x}}=i\int_{\RR+i(s-1)}
     \frac{P (\tau ,v,x)}{1-e^{2\pi x}}\, dx\, .
\end{eqnarray*}
Using this and (\ref{eq:5.12}), we obtain:
\begin{eqnarray*}
  a_j (\tau ,v) -e^{2\pi i (m+1)(2v-\tau)} a_j (\tau ,v-\tau)
  = i \left( \int_{\RR +is} - \int_{\RR +i (s-1)} \right)
      \frac{P (\tau,v,x)}{1-e^{2\pi x}} \, dx\\
  =  2\pi \Res_{x=0} \frac{P (\tau ,v,x)\, dx}{1-e^{2\pi x}}
       = P (\tau ,v,0)\, , 
\end{eqnarray*}
proving (b), and the theorem.
\end{proof}

\begin{remark}  
\label{rem:5.6}

Let $\phi^{[m]}  =\varphi^{[m]} -\tfrac12 G = \tfrac12 (\varphi^{[m]}
+ \varphi^{[m]} |_S) $.  This function has a good modular
transformation property:  $\phi^{[m]} |_S = \phi^{[m]}$, but not
so good elliptic transformation property:
\begin{eqnarray*}
  \phi^{[m]} (\tau,u,v-\tau ,t) = e^{2\pi i (m+1) (2v-\tau)}
  \phi^{[m]} (\tau ,u,v,t)\\
   -\tfrac12 \sum^{2m+1}_{j=1} e^{2\pi ij v} q^{- \frac{j^2}{4m+4}} 
      (\Theta_{j,m+1} -\Theta_{-j,m+1}) (\tau ,2u)\, .
\end{eqnarray*}
\end{remark}

Following the idea of Zwegers \cite{Z}, we introduce below a
non-holomorphic modification $\tilde{\varphi}^{[m]}$ of
$\varphi^{[m]}$, which has both good modular and elliptic
transformation properties. For that
we shall be making use of the functions $R_{m;j} (\tau ,v)$,
where $m$ is a positive integer and $j\in \ZZ$, 
introduced in \cite{Z}, p.~51 $(\tau ,v \in \CC\, , \, \Im \tau
>0)$:
\begin{equation*}
  R_{m;j} (\tau ,v) 
    = \displaystyle{\sum_{\substack{n \in \ZZ \\ n\equiv j\!\!\!\!\!\mod 2m}}}
\left( {\rm sign} \left(n+\tfrac12 \right)
  -E \left(\left( n+2m  \frac{\Im  v}{\Im \tau}\right) 
\left(\frac{\Im \tau}{m}\right)^{\tfrac12}\right)\right)
e^{-\frac{\pi i n^2}{2m} \tau - 2\pi i n v}\, , 
\end{equation*}
where $E (x)$  is the odd entire function
$2 \int^x_0 e^{-\pi u^2}\, du$. The explicit expression of the
holomorphic function $E(x)$ is used only in the proof of
formula (\ref{eq:5.13}) below, given in \cite{Z}.
The key property of these functions relates them to the functions $h_{m;j}$
as follows (see \cite{Z}, Remark~3.6):
\begin{equation}
  \label{eq:5.13}
  R_{m;j} (\tau ,v) + \frac{i}{(-2mi\tau)^\frac12}
    e^{\frac{2\pi i mv^2}{\tau}} \sum_{k \in \ZZ /2m\ZZ}
    e^{-\frac{\pi ijk}{m}} R_{m;k} 
      \left(- \frac{1}{\tau} \, , \, \frac{v}{\tau}   \right)
       = 2h_{m;j} (\tau ,v)\, ,
\end{equation}
provided that $0\leq j\leq 2m-1$.

\begin{remark}  
\label{rem:5.6a}
Note that the functions $R_{m;j}$ depend on $j$ only $\mod 2m$, while this is
not the case for the functions $h_{m;j}$, namely 
$h_{m;j}(\tau,u)- h_{m;j+2m}(\tau,u)=q^{-\frac{j^2}{4m}} e^{-2\pi i ju}$. 
\end{remark}  

The functions $R_{m;j}$ have also the following elliptic
transformation properties, which is straightforward to check.

\begin{lemma}
  \label{lem:5.7}
  \begin{list}{}{}
  \item (a) $R_{m;j}(\tau,v+\tfrac12) = (-1)^j R_{m;j} (\tau, v)$.

\item (b)  For $0 \leq j \leq 2m-1$ one has:
  \begin{equation*}
    R_{m;j} (\tau ,v) -e^{2\pi i m (2v-\tau)} R_{m;j} (\tau,v-\tau)
      = 2e^{-\frac{\pi i}{2m} (2m-j)^2\tau + 2\pi i (2m-j)v}\, .
  \end{equation*}
  \end{list}
\end{lemma}


Now let
\begin{eqnarray*}
  \varphi^{[m]}_{\add} (\tau ,u,v,t) &=& \tfrac12 e^{2\pi i (m+1) t}
   \sum_{j \in \ZZ/(2m+2)\ZZ} R_{m+1;j} (\tau ,v)
   (\Theta_{j,m+1} -\Theta_{-j,m+1} )(\tau, 2u)\, ,\\
    && \tilde{\varphi}^{[m]} (\tau ,u,v,t)= \varphi^{[m]}
       (\tau,u,v,t) + \varphi^{[m]}_{\add} (\tau ,u, v ,t)\,.
\end{eqnarray*}

The correction function $\varphi^{[m]}_{\add}$ has the following
properties.

\begin{lemma}
  \label{lem:5.8}
  \begin{list}{}{}
  \item (a) $\varphi^{[m]}_{\add} (\tau,u,v,t) -
    \varphi^{[m]}_{\add}  |_S (\tau,u,v,t)  = - G (\tau,u,v,t)  $.

\item (b) $\varphi^{[m]}_{\add} (\tau, u+a, v+b, t)
    = \varphi^{[m]}_{\add} (\tau ,u,v,t)    $ if $a,b \in \tfrac12
    \ZZ$ are such that $a+b \in \ZZ$.

\item (c) $\varphi^{[m]}_{\add} (\tau,u+\tau,v,t) =q^{-(m+1)} 
    e^{-4\pi i (m+1)u }  \varphi^{[m]}_{\add} (\tau,u,v,t) $.

\item (d) $\varphi^{[m]}_{\add} (\tau,u,v,t)   -e^{2\pi i
    (2v-\tau)} \varphi^{[m]}_{\add} (\tau, u,v-\tau ,t)$\\
   \hspace*{4ex} 
$=e^{2\pi i (m+1)t} \sum^{2m+1}_{j=1} e^{2\pi ijv}
    q^{-\frac{j^2}{4m+4}} (\Theta_{j,m+1} -\Theta_{-j,m+1})(\tau,2u)  $.
  \end{list}

\end{lemma}

 \begin{proof}
   We have by definition of $\varphi^{[m]}_{\add}$:
  \begin{eqnarray*}
 \lefteqn{\hspace{-6.35in}\varphi^{[m]}_{\add} |_S (\tau ,u,v,t) = \tau^{-1}   
   \varphi^{[m]}_{\add}
  \left( - \frac{1}{\tau}\, , \, \frac{u}{\tau} \, , \,
    \frac{v}{\tau}\, ,\,  t-\frac{u^2-v^2}{\tau}\right) }\\
 = \frac{1}{2\tau} e^{2\pi i (m+1)t} e^{-\frac{2\pi i  (m+1)}{\tau} 
(u^2-v^2)} \sum_{j \in \ZZ /(2m+2)\ZZ} R_{m+1;j}
      \left( -\frac{1}{\tau}\, , \, \frac{v}{\tau} \right) 
        ( \Theta_{j,m+1}-\Theta_{-j,m+1}) \left( - \frac{1}{\tau} \,
          , \, \frac{2u}{\tau}\right)\, .
   \end{eqnarray*}
Now we apply to the RHS the following modular transformation
formula, which is immediate by (\ref{eq:5.7}):
\begin{eqnarray*}
 \lefteqn{\hspace{-5in}(\Theta_{j,m+1} - \Theta_{-j,m+1}) (-\frac{1}{\tau} \, ,
\,\frac{2u}{\tau})  }\\
   = e^{\frac{2\pi i (m+1)}{\tau}} u^2 \left(
     -\frac{i\tau}{2m+2}\right)^\frac12 \sum_{k \in \ZZ /  (2m+2)\ZZ}
     e^{-\frac{\pi i jk}{m+1}} (\Theta_{k,m+1}-\Theta_{-k,m+1})
       (\tau ,2u) \, .
\end{eqnarray*}
Substituting this in the previous formula, we obtain (after
exchanging $j$ and $k$):
\begin{eqnarray*}
 \lefteqn{\hspace{-5in}  \varphi^{[m]}_{\add} |_S (\tau , u, v, t) 
   = \frac{-i}{2\sqrt{2m+2}}  (-i\tau)^{-\frac12} e^{2\pi i (m+1)t}  }\\
 \times \sum_{j,k \in \ZZ /(2m+2)\ZZ} e^{\frac{2\pi i(m+1)}{\tau}v^2}
   e^{-\frac{\pi ijk}{m+1}} R_{m+1;k} \left( -\frac{1}{\tau}\, ,
     \, \frac{v}{\tau}  \right) (\Theta_{j,m+1}-\Theta_{-j,m+1})
      (\tau ,2u)\, .
\end{eqnarray*}
Using this and the key formula (\ref{eq:5.13}), we have:
\begin{eqnarray*}
 \lefteqn{\hspace{-4.15in}\varphi^{[m]}_{\add} (\tau ,u,v,t) 
      - \varphi^{[m]}_{\add} |_S (\tau ,u,v,t)  }\\
=  e^{2\pi i (m+1)t} \sum_{j=1}^{2m+1} h_{m+1;j}
   (\tau ,v) (\Theta_{j,m+1}-\Theta_{-j,m+1}) (\tau, 2u)\, .
\end{eqnarray*}
The RHS of this equation equals to $-G (\tau,u,v,t)$ by Theorem
\ref{th:5.5}.  This proves claim~(a).  The remaining claims (b),
(c) and (d) follow from Lemma~\ref{lem:5.7}.

\end{proof}

Now we can prove the following important theorem.

\begin{theorem}
  \label{th:5.9}
  \begin{list}{}{}
  \item (a)  $\tilde{\varphi}^{[m]} (-\frac{1}{\tau},\frac{u}{\tau} ,
\frac{v}{\tau}, t -\frac{u^2-v^2}{\tau}) = \tau \tilde{\varphi}^{[m]}(\tau,u,v,t)$.

\item (b) $\tilde{\varphi}^{[m]} (\tau,u+a, v+b,t) =
  \tilde{\varphi}^{[m]} (\tau,u,v,t)$ 
if $a,b \in \tfrac12 \ZZ$  are such that $a+b \in \ZZ$.

\item (c)$\tilde{\varphi}^{[m]} (\tau,u,-v,t) 
      =\tilde{\varphi}^{[m]} (\tau,u,v,t)$,\,\,
      $\tilde{\varphi}^{[m]} (\tau,-u,v,t) = -
      \tilde{\varphi}^{[m]} (\tau,u,v,t)$.

\item (d) $\tilde{\varphi}^{[m]} (\tau +1,u,v,t) =
  \tilde{\varphi}^{[m]} (\tau ,u ,v,t,)$.

\item (e)  $\tilde{\varphi}^\bracm (\tau, u+a\tau, v+b\tau ,t)=q^{(m+1)(b^2-a^2)} e^{4\pi i (m+1)(-au+bv)}\tilde{\varphi}^\bracm (\tau ,u,v,t)$
if $a,b \in \tfrac12 \ZZ$  are such that $a+b \in \ZZ$.
\end{list}
\end{theorem}

\begin{proof}
  Adding to the equation of Lemma~\ref{lem:5.8}(a)
  formula (\ref{eq:5.5}), we get $\tilde{\varphi}^\bracm
  -\tilde{\varphi}^\bracm |_S =0$, proving claim (a).
Adding equations of Lemmas~\ref{lem:5.3}(a) and \ref{lem:5.8}(b),
we get claim (b). 

Claim (c) is derived as follows. It is straightforward to check that
\begin{equation}
  \label{eq:5.14}
R_{m;j} (\tau, -v) + R_{m;-j} (\tau, v) = 2\delta_{0,j},\,\,\,j\in \ZZ/2m\ZZ.
\end{equation}
It follows easily that 
$\varphi^{[m]}_{\add}(\tau,u,v,t)$ 
is unchanged if we change the sign of $v$. It is also immediate to see
that $\varphi^{[m]}_{\add}(\tau,u,v,t)$ changes sign if we change the sign of
$u$. This proves (c). Claim (d) is obvious.

Finally, we derive claim (e) from claims (a) and (b) as follows.
Replacing $u$ by $u+a\tau$ and $v$ by $v+b\tau$ in (a), we obtain
\begin{eqnarray*}
  &&\tilde{\varphi}^\bracm \left(-\frac{1}{\tau}\,,\, \frac{u}{\tau} +a \,,\,
    \frac{v}{\tau} +b \,,\, t\right)\\[1ex]
   && =e^{\frac{2\pi i (m+1)}{\tau}(u^2-v^2)}
       e^{4\pi i (m+1)(au-bv)} q^{(m+1)(a^2-b^2)}
         \tilde{\varphi}^\bracm (\tau, u+a\tau, v+b\tau, t)\, .
\end{eqnarray*}
But an equivalent form of (a) is
\begin{equation*}
  \tilde{\varphi}^\bracm \left( - \frac{1}{\tau} \,,\, \frac{u}{\tau}\,,\,
    \frac{v}{\tau}\,,\, t \right) 
     = \tau e^{\frac{2\pi i (m+1)}{\tau}(u^2-v^2)}
        \tilde{\varphi}^\bracm (\tau ,u,v,t)\, .
\end{equation*}
Since, by (b), the LHS of the last two formulae are equal, we
conclude that the RHS are equal as well, which gives (e).

\end{proof}

Translating the discussion on $\varphi^{[m]}_{\add}$ and
$\tilde{\varphi}^{[m]}$ into the language of $\Phi$'s, we obtain
the following modification of the function
$\Phi^{[m]}(\tau,z_1,z_2,t)$:
\begin{equation*}
    \tilde{\Phi}^{[m]} (\tau, z_1,z_2,t )=
    \Phi^{[m]}(\tau,z_1,z_2,t) + \Phi^{[m]}_{\add}
    (\tau,z_1,z_2,t),
\end{equation*}
where
\begin{equation*}
 \tilde{\Phi}^{[m]}_{\add}  (\tau,z_1,z_2,t)  
     = -\tfrac12 e^{2\pi i (m+1)t} \sum_{j \in \ZZ/(2m+2)\ZZ} 
        R_{m+1;j} (\tau, \frac{z_1-z_2}{2}) 
          (\Theta_{j,m+1}-\Theta_{-j,m+1}) (\tau,z_1+z_2)\,.
\end{equation*}
Then we obtain the following corollary of Theorem \ref{th:5.9}.

\begin{corollary}
  \label{cor:5.10}

  \begin{list}{}{}
  \item (a)  
$\tpm (-\frac{1}{\tau}\,,\, \frac{z_1}{\tau}\,    ,\,\frac{z_2}{\tau}\,,\, t-\frac{z_1z_2}{\tau})      = \tau \tpm (\tau,z_1,z_2,t)$.

\item (b)  $\tpm (\tau, z_1+a , z_2+b,t) = \tpm (\tau,z_1,z_2,t)$
  if $a,b \in \ZZ$.

\item (c) $\tpm (\tau,-z_1,-z_2,t) = - \tpm (\tau,z_1,z_2,t)$;\,\, 
     $\tpm (\tau,z_2,z_1,t) = \tpm (\tau,z_1, z_2,t)$.

\item (d)  $\tpm (\tau +1,z_1,z_2,t)  = \tpm (\tau,z_1,z_2,t)$.

\item (e)  $ \tilde{\Phi}^{\bracm} (\tau ,z_1 + j\tau , z_2 +k\tau, t) = q^{-(m+1)jk} e^{-2\pi i (m+1) (kz_1 + jz_2)} \tilde{\Phi}^\bracm (\tau ,z_1,z_2,t)$
if $j,k \in \ZZ$.
  \end{list}
\end{corollary}

\begin{remark}
 \label{rem:5.12}
The following Appell--Lerch sum plays a key role in Zwegers'
paper \cite{Z}:
\begin{equation*}
  \mu (\tau,z_1,z_2) = \frac{e^{\pi      iz_1}}{\vartheta_{11}(\tau,z_2)}
    \sum_{n \in \ZZ} \frac{(-1)^n q^{\tfrac12 (n^2+n)}e^{2\pi i nz_2}}
{1- e^{2\pi iz_1}q^n}\, ,
\end{equation*}
along with its non-meromorphic modification 
$$
{\tilde{\mu}
  (\tau,z_1,z_2) = \mu (\tau,z_1,z_2) +\tfrac{i}{2}R(\tau,z_1-z_2)},
$$ 
where $R$ is defined in \cite{Z},
p.~11. Let us compare these functions with the functions
$\Phi^\bracm (\tau,z_1,z_2):= \Phi^{\bracm} (\tau,z_1,z_2,0)$
and $R_{m;j} (\tau, z)$, which play a key role in the present
paper.  First it is easy to see that
\begin{equation}
  \label{eq:5.15}
  R(\tau,z) = R_{2;1} (\tau,\tfrac{z}{2})-R_{2;-1} (\tau,\tfrac{z}{2})\,.
\end{equation}
Furthermore, in \cite{KW4} (formula (3.22) for $s=0,m=2$) we
proved the following character formula:
\begin{equation*}
  \ch^-_{\Lambda_0} (\tau,z_1,z_2,0) =\vartheta_{11}(\tau,z_1)
     \vartheta_{11}(\tau,z_2)\mu (\tau,z_1,z_2)/\eta (\tau)^3\, .
\end{equation*}
 Comparing this formula with (\ref{eq:5.3}) for $m=1$, we
 obtain the following identity:
 \begin{equation}
   \label{eq:5.16}
   \Phi^{[1]} (\tau,z_1,z_2,0) = \vartheta_{11}(\tau,z_1+z_2)
      \mu (\tau,z_1,z_2)\, .
 \end{equation}
On the other  hand, using (\ref{eq:5.15}), it is easy to see
that
\begin{equation}
  \label{eq:5.17}
  \Phi^{[1]}_{\add} (\tau,z_1,z_2,0) =  \tfrac{i}{2}R(\tau,z_1-z_2)
     \vartheta_{11} (\tau,z_1+z_2)\,.
\end{equation}
Comparing (\ref{eq:5.16}) and (\ref{eq:5.17}), we get the ``modified''
identity
\begin{equation}
  \label{eq:5.18}
 \tilde{\Phi}^{[1]} (\tau,z_1,z_2,0) = \vartheta_{11}(\tau,z_1+z_2)
    \tilde{\mu} (\tau,z_1,z_2) \, .
\end{equation}

\end{remark}

\section{Modular transformation formula for the function
  $\tilde{\Phi}^{[m]} \left( \frac{\tau}{M} \, , \, \frac{u}{M} \, , \,
    \frac{v}{M} \, , t  \right)$}
\label{sec:6}

Given an additive subgroup $A$ of $\QQ$, consider the following
abelian group:
\begin{equation*}
  \tilde{\Omega}_A = \left\{(a,b ) \in \frac12 A \times \frac12 A |\,\, a+b
      \in A \right\}\, .
\end{equation*}
For a positive integer $M$ consider the abelian group
 \begin{equation*}
   \Omega_M = \tilde{\Omega}_\ZZ  / \tilde{\Omega}_{M\ZZ}\, .
 \end{equation*}

The first result of this section is the following theorem.

\begin{theorem}
  \label{th:6.1}
Let $M$ be a positive integer and let $m$ be a non-negative
integer, such that $gcd (M,2m+2)=1$ if $m>0$.  Then

\begin{list}{}{}
\item (a)  $\tilde{\varphi}^{\bracm} \left( \frac{\tau}{M} \, , \,
    \frac{u}{M}\, , \, \frac{v}{M} \, , \, t \right) = \frac{M}{\tau}
  \tilde{\varphi}^{\bracm} \left( -\frac{M}{\tau} \, , \,
    \frac{u}{\tau} \, , \, \frac{v}{\tau} \, , \,
    t-\frac{u^2-v^2}{\tau M}  \right)$.

\item  (b)  $\tilde{\varphi}^{\bracm} \left( \frac{\tau}{M} \, , \,
    \frac{u}{M}\, , \, \frac{v}{M} \, , \, t \right) = \displaystyle{\sum_{a,b
    \in \Omega_M} q^{\frac{m+1}{M} (a^2-b^2)} e^{\frac{4\pi i
      (m+1)}{M} (au-bv)} \tilde{\varphi}^\bracm (M\tau \, , \,
  u+a\tau \, , \, v+b\tau,t)}$.

\end{list}

\end{theorem}

  Replacing $(\tau,u,v,t)$ by $\left( \frac{\tau}{M}\, , \,
    \frac{u}{M}\, , \, \frac{v}{M} \, , \, t  \right)$ in
  Theorem~\ref{th:5.9}(a), we obtain (a).  The proof of (b) is
  based on several lemmas, proven below.


Given coprime positive integers $p$ and $q$, for each integer $n
\in [0,q-1]$ there exist unique integers $n' \in [0,p-1]$ and
$b_n$, such that
\begin{equation}
  \label{eq:6.1}
  n=n'q+b_np\, .
\end{equation}
Furthermore, the set 
$$
I_{q,p}:= \{ b_n |\,\,n=0,1,\ldots ,q-1 \}
$$
consists of $q$ distinct integers.  Any $n \in \ZZ$ can be
uniquely represented in the form (\ref{eq:6.1}), where $n'\in \ZZ$
and $b_n \in I_{q,p}$ and this decomposition has the following
properties:

\begin{list}{}{}
\item (i)  $n \geq 0$ iff $n'\geq 0 $\, ;

\item (ii) if $j,j_0 \in \ZZ/p\ZZ$ are such that $j \equiv qj_0
  \mod p$, then $n \equiv j \mod p$ iff $n'\equiv j_0 \mod p$.
\end{list}

We shall apply this setup to $p=2m+2$, $q=M$, and let
$I=I_{M,2m-2}$ for short.

\begin{lemma}
  \label{lem:6.2}

  \begin{list}{}{}
  \item (a) $\Phi^\bracm \left( \frac{\tau}{M}\, , \,
      \frac{z_1}{M}\, , \, \frac{z_2}{M }\, , \, t \right) $\\[1ex] 
     $\displaystyle{=  \sum_{\substack{0 \leq a <M\\b \in I}} e^{\frac{2\pi i
            (m+1)}{M} ((a+2b) z_1 + az_2)} q^{\frac{m+1}{M}(a^2+ 2ab)} 
        \Phi^\bracm (M\tau \, , \, z_1 + a\tau \, , \, z_2+(a+2b) \tau ,\,t)}$.

\item (b)  $\varphi^\bracm \left(  \frac{\tau}{M}\, , \,\frac{u}{M}\, , \, 
\frac{v}{M }\, , \, t \right)$\\
  $   \displaystyle{ =  \sum_{\substack{0 \leq a <M\\b \in I}} 
     e^{-\frac{4\pi i (m+1)}{M}
          ((a+b)u+bv)} q^{\frac{m+1}{M} (a^2+2ab)} }  
\varphi^\bracm \left(  M\tau,\, u-(a+b)\tau,\, v+b\tau,\, t \right)$.
  \end{list}

\end{lemma}

\begin{proof}
  We prove (a); (b) follows immediately from (a).
Recall that
\begin{eqnarray*}
  \Phi^\bracm (\tau,z_1,z_2,t) &=& e^{2\pi i (m+1)t}
    (\Phi^\bracm_1 (\tau,z_1,z_2) - \Phi^\bracm_1 (\tau,-z_2,-z_1)),\\
\noalign{\hbox{where}}\\
\Phi^\bracm_1 (\tau,z_1,z_2)&=& \sum_{j \in \ZZ}
  \frac{e^{2\pi ij (m+1)(z_1+z_2)}q^{j^2(m+1)}}{1-e^{2\pi i z_1}q^j}\, .
\end{eqnarray*}
We have, by expanding each term in this series in a geometric
series, and replacing $(\tau,z_1,z_2)$ by $\left(
  \frac{\tau}{M}\, , \, \frac{z_1}{M}\, , \, \frac{z_2}{M}\right)$:
\begin{eqnarray}
  \label{eq:6.2}
  \Phi^\bracm_1 \left(\frac{\tau}{M} \, , \, \frac{z_1}{M}\, , \,  \frac{z_2}{M}\right)
   = \left( \sum_{j,k \geq 0} - \sum_{j,k<0}  \right)
     e^{\frac{2\pi i (m+1)}{M} j (z_1+z_2)} e^{\frac{2\pi ik z_1}{M}}
        q^{\frac1M (j^2 (m+1)+jk) }\, .
\end{eqnarray}
Now we divide $j$ by $M$ with the remainder $a$, $0 \leq a <M :
j=j'M+a$, and decompose $k$ according to (\ref{eq:6.1}): $k=k'M +
(2m+2)b_k$.  Since
\begin{equation*}
  \frac1M (j^2 (m+1)+jk) = j' (j'(m+1)+k')M +j' (a+2b_k)(m+1)
      + a (j' (m+1)+k') + \frac{a (a+2b_k)(m+1)}{M}\, , 
\end{equation*}
we obtain from (\ref{eq:6.2}):
\begin{eqnarray*}
  \lefteqn{\hspace{-5.95in} \Phi_1 \left( \frac{\tau}{M} \, , \, \frac{z_1}{M}\, , \,
    \frac{z_2}{M}   \right) }\\
    = \sum_{\substack{0 \leq a <M\\b \in I}} \left(
      \sum_{j',k'\geq 0}-\sum_{j',k' < 0}   \right)
    A_{j',k'} e^{\frac{2\pi i (m+1)a}{M}(z_1+z_2)} e^{\frac{4\pi i(m+1)}{M} bz}
       q^{Mj' (j'(m+1)+k')} B_{j',k'} q^{\frac{(m+1)a(a+2b)}{M}}\, , 
\end{eqnarray*}
where
\begin{equation*}
  A_{j',k'} = e^{2\pi i (m+1) j' (z_1+z_2)} e^{2\pi i k'z_1}\, ,\,\,\, 
    B_{j',k'} = e^{2\pi i \tau (2j' (m+1)(a+b)+ak')}\, .
\end{equation*}
Therefore, using again (\ref{eq:6.2}), we obtain:
\begin{eqnarray*}  \lefteqn{\hspace{-5.35in}  
\Phi^\bracm_1 \left( \frac{\tau}{M}\,,\,\frac{z_1}{M}\,,\, \frac{z_2}{M}\right)}\\
    =\sum_{\substack{0 \leq a <M\\b \in I}} e^{\frac{2\pi i(m+1)}{M} 
      ((a+2b) z_1+az_2)} q^{\frac{m+1}{M} (a^2+2ab)}\Phi^\bracm_1
       (M\tau, z_1 + a\tau, z_2+(a+2b)\tau)\,, 
\end{eqnarray*}
which proves (a).

\end{proof}

\begin{lemma}
  \label{lem:6.3}

Given $j \in \ZZ/(2m+2)\ZZ$, let $j_0 \in \ZZ/(2m+2)\ZZ$ be the element, such that $j-Mj_0 \in (2m+2)\ZZ$.
  Then

\begin{list}{}{}

\item (a) $R_{m+1;j} \left( \frac{\tau}{M} \, , \, \frac{v}{M}  \right)
    =\displaystyle{\sum_{b\in I} q^{-\frac{m+1}{M} b^2}  e^{-\frac{4\pi i (m+1)}{M}bv}
       R_{m+1;j_0} (M\tau,\,v+b\tau)}$.

\item (b)  $\Theta_{j,m+1} \left( \frac{\tau}{M} \,,\,    \frac{2u}{M}\right)
   = \displaystyle{\sum_{a\in I} q^{\frac{m+1}{M} a^2} e^{\frac{4\pi i (m+1)}{M} au}
     \Theta_{j_0,m+1} (M \tau,\, 2u+2a\tau)}$.

\end{list}
\end{lemma}

\begin{proof}
  By definition of $R_{m;j}$ we have:
\begin{eqnarray*}  
   && R_{m+1;j} \left( \frac{\tau}{M} \, , \, \frac{v}{M}  \right)\\
   &=& \sum_{n\equiv j\!\!\!\!\mod 2m+2} \left( \sign (n+\tfrac12) -
E \left(\left(n+(2m+2)\frac{{\Im} {\it v} }{\Im \tau}\right) \, \frac{1}{M} 
\left( \frac{\Im(M\tau)}{m+1}\right)^{\tfrac12}\right)\right)
e^{-\frac{\pi in^2}{(2m+2)M}\tau -\frac{2\pi in}{M}v}\, .
  \end{eqnarray*}
We have the decomposition (\ref{eq:6.1}):
\begin{equation*}
  n=n'M +b_n (2m+2)\, , 
\end{equation*}
and, by its property (ii),
\begin{equation*}
  n \equiv j\!\!\!\! \mod 2m+2 \hbox{\,\, iff\,\,} n'\equiv j_0\!\!\!\! \mod 2m+2\, .
\end{equation*}
Using property (i) of the decomposition (\ref{eq:6.1}), we
obtain:
\begin{eqnarray*}
&& R_{m+1;j} \left( \frac{\tau}{M} \,,\, \frac{v}{M} \right)\\
  &=&\sum_{b\in I}\,\, \sum_{n'\equiv j_0\!\!\!\!\! \mod 2m+2}
     \left( \sign (n'+\tfrac12) - E \left(\left(n'+ (2m+2)
         \frac{\Im  ({\it v+b}\tau)}{\Im (M\tau)}\right) \left(
           \frac{\Im (M \tau)}{m+1}\right)^{\tfrac12}\right)\right)\\[1ex]
  &&\times e^{-\frac{\pi i M\tau}{2m+2} n^{\prime^2} -2\pi i n'(v+b\tau) }
q^{-\frac{m+1}{M}  b^2}
     e^{-\frac{4\pi i (m+1)}{M}bv}\\[1ex]
 &=& \sum_{b \in I} q^{-\frac{m+1}{M}b^2}e^{-\frac{4\pi i(m+1)}{M}bv}
     R_{m+1;j_0} (M\tau , v+b\tau)\, , 
\end{eqnarray*}
proving (a).

In order to prove (b), note that we have from (\ref{eq:A2}):
%
\begin{equation}
  \label{eq:6.3}
 \Theta_{j,m+1} \left( \frac{\tau}{M} \,,\, \frac{z}{M} \right)     
= \sum_{n \equiv j\!\!\!\!\! \mod\! 2m+2}e^{\frac{2\pi i \tau}{(4m+4)M}n^2}
        e^{\frac{\pi i n z}{M}}\, .
\end{equation}
Using the decomposition (\ref{eq:6.1}): $n=n'M + (2m+2)a$, where
$a \in I$, and its property (ii), we deduce:
\begin{eqnarray*}
  \Theta_{j,m+1} \left( \frac{\tau}{M}\, , \, \frac{z}{M} \right)
    &=& \sum_{a\in I}\,\, \sum_{n'\equiv j_0 \!\!\!\!\!\mod 2m+2}
        e^{\frac{2\pi i \tau}{ (4m+4)M} (n'M+(2m+2)a)^2}
          e^{\frac{\pi i}{M} (n'M+(2m+2)a)z}\\
    &=& \sum_{a\in I}q^{\frac{m+1}{M}a^2} e^{\frac{2\pi i (m+1)}{M}az}
    \Theta_{j_0,m+1} (M \tau , z+2a\tau)\, , 
\end{eqnarray*}
and (b) follows by replacing $z$ by $2u$.

\end{proof}

Lemma \ref{lem:6.3} implies that for $\varphi^\bracm_\add
(\tau,u,v,t)$, defined in Section~\ref{sec:5}, we have:
\begin{equation}
  \label{eq:6.4}
  \varphi^\bracm_{\add}
 \left( \frac{\tau}{M} \, ,\, \frac{u}{M}\,,\, \frac{v}{M}\, , \,t\right)
   = \sum_{a,b \in I} q^{\frac{m+1}{M} (a^2-b^2)}
      e^{\frac{4\pi i (m+1)}{M}(au-bv)}  \varphi^\bracm_{\add}(M\tau, u+a\tau, v+b\tau, t)\, .
\end{equation}

\vspace{2ex}
\noindent{\em End of the proof of Theorem \ref{th:6.1}(b).}
First, by induction on $|j|$ we obtain from Lemma
\ref{lem:5.8}(c) for $j \in \ZZ$:
\begin{equation}
  \label{eq:6.5}
  \varphi^\bracm (\tau,u+j \tau , v,t) = q^{-j^2 (m+1)}
     e^{-4\pi i j (m+1)u} \varphi^\bracm (\tau ,u,v,t)\, .
\end{equation}
This equation implies that for $j \in \CC$, the expression
\begin{equation}
  \label{eq:6.6}
  q^{\frac{m+1}{M}j^2} e^{\frac{4\pi i (m+1)}{M}ju}\varphi^\bracm
    (M\tau , u+j\tau ,v,t)
\end{equation}
remains unchanged if we replace $j$ by $j'=j+Mn$, $n \in \ZZ$ .
Indeed, we have: $\varphi^\bracm (M\tau, u+j'\tau,v,t) =
\varphi^\bracm (M\tau , u+j\tau + \frac{j'-j}{M}M\tau,v,t)$, and
we apply (\ref{eq:6.5}) with $u$ replaced by $u+j\tau$, $\tau$
replaced by $M\tau$ and $j$ replaced by $\frac{j'-j}{M}$.

Replacing $v$ in (\ref{eq:6.6}) by $v+b\tau\, (b \in \CC)$ and
multiplying it by $q^{-\frac{m+1}{M}b^2} e^{-\frac{4\pi i
    (m+1)}{M}bu}$, we deduce that the expression
\begin{equation}
  \label{eq:6.7}
  q^{\frac{m+1}{M}(a^2-b^2)} e^{\frac{4\pi i (m+1)}{M}(au-bv)}
  \varphi^\bracm (M\tau, u+a\tau,v+b\tau,t)
\end{equation}
remains unchanged if we replace $a \in \CC$ by $a+Mn$, $n \in
\ZZ$.

It follows that Lemma \ref{lem:6.2}(b) can be rewritten as
follows:
\begin{equation}
  \label{eq:6.8}
  \varphi^\bracm \left( \frac{\tau}{M}\,,\, \frac{u}{M}\,,\,
    \frac{v}{M} \, , \, t \right)
= \sum_{a,b \in I} q^{\frac{m+1}{M}  (a^2-b^2)}
   e^{\frac{4\pi i (m+1)}{M} (au-bv)}
\varphi^\bracm (M\tau \, , \, u+a\tau \, , \, v+b\tau \, , \, t)\,,
\end{equation}
by making use of the following lemma.

\begin{lemma}
  \label{lem:6.4}
Given $b \in \ZZ$, for each $a \in \ZZ_{\geq 0}$, such that $a<M$ there
exists a unique $a'\in I$, such that $-(a+b) \equiv a' \mod M$.
Moreover $\{ a' |\,\, 0 \leq a < M \} =I$.
\end{lemma}

\begin{proof}
We have, by decomposition (\ref{eq:6.1}) with $p=2m+2$, $q=M$:
\begin{equation*}
  -2 (m+1) (a+b) = n'M + (2m+2)a'\,\, (n'\in\ZZ \, , \, a'\in I)\, .
\end{equation*}
Hence $-(2m+2) (a+b) \equiv (2m+2) a'\mod M$, and $-(a+b) \equiv
a'\mod M$.
\end{proof}

We obtain from  (\ref{eq:6.8}) and (\ref{eq:6.4}):
\begin{equation}
  \label{eq:6.9}
  \tilde{\varphi}^\bracm \left( \frac{\tau}{M}\,,\, \frac{u}{M}\,,\,
    \frac{v}{M} \,,\, t\right)
     = \sum_{a,b \in I} q^{\frac{m+1}{M} (a^2-b^2)} 
        e^{\frac{4\pi i (m+1)}{M} (au-bv)}
        \tilde{\varphi}^\bracm (M\tau,u+a\tau,v+b\tau, t)\, .
\end{equation}
It follows from Theorem \ref{th:5.9}(e) that the expression
\begin{equation}
  \label{eq:6.11}
  q^{\frac{m+1}{M} (a^2-b^2)} e^{\frac{4\pi i (m+1)}{M}(au-bv)}
    \tilde{\varphi}^{[m]} (M \tau, u+a\tau, v+b\tau ,t)
\end{equation}
is independent of the choice of $(a,b) \in \tilde{\Omega}_\ZZ \mod
\tilde{\Omega}_{M\ZZ}$.  Along with (\ref{eq:6.9}), this
completes the proof of Theorem~\ref{th:6.1}(b).

Next, we translate the obtained results from
$\tilde{\varphi}^\bracm$ to $\tilde{\Phi}^\bracm$.  For that we
define the map
\begin{equation}
  \label{eq:6.12}
  \ZZ^2 \to \left(\tfrac12 \ZZ \right)^2 \, , \, 
  (j,k) \mapsto \left( a=-\frac{j+k}{2}\, , \, b=\frac{j-k}{2}
  \right)\, .
\end{equation}
This map induces a bijective map
\begin{equation}
  \label{eq:6.13}
  (\ZZ/M\ZZ)^2 \to \Omega_M \, .
\end{equation}

It follows from Corollary \ref{cor:5.10}(e)  that the function
\begin{equation}
  \label{eq:6.15}
  q^{\frac{m+1}{M}jk} e^{\frac{2\pi i (m+1)}{M} (kz_1 + jz_2)}
     \tilde{\Phi}^{[m]} (M\tau, z_1 + j\tau, z_2+k\tau ,t )
\end{equation}
is independent of the choice of $j,k \!\!\! \mod M \ZZ$.

\begin{theorem}
  \label{th:6.5}
Let $M$ be a positive integer and let $m$ be a non-negative integer, such 
that $gcd (M,2m+2)=1$.  Then

\begin{list}{}{}
\item (a)  $\tilde{\Phi}^{\bracm} \left(- \frac{M}{\tau} , \,
    \frac{z_1}{\tau} , \, \frac{z_2}{\tau}  , \, t
    -\frac{z_1z_2}{\tau M}\right)$\\ 
\hspace*{4ex}$= \frac{\tau}{M} \displaystyle{\sum_{j,k \in
    \ZZ /M\ZZ}} q^{\frac{m+1}{M} jk} e^{\frac{2\pi i (m+1)}{M} (kz_1+jz_2)}
       \tilde{\Phi}^\bracm (M\tau , \, z_1+j\tau ,\, z_2+k\tau ,\, t)$.             
\item  (b)  $\tilde{\Phi}^{\bracm} \left( \frac{\tau}{M} \, , \,
         \frac{z_1}{M}\, , \, \frac{z_2}{M} \, , \, t \right) 
           = \displaystyle{\sum_{j,k \in \ZZ/M\ZZ}}
              q^{\frac{m+1}{M}jk} 
                 e^{\frac{2\pi i(m+1)}{M} (kz_1+jz_2)} 
    \tilde{\Phi}^\bracm (M\tau , \, z_1+j\tau ,\, z_2+k\tau ,\, t)$.
\end{list}

\end{theorem}

\begin{proof}
  Recall that $\tilde{\Phi}^\bracm (\tau,z_1,z_2,t)
  =\tilde{\varphi}^{\bracm} (\tau,u,v,t)$, where $z_1=v-u$,
  $z_2=-(v+u)$.  Hence $\tilde{\varphi}^\bracm (M\tau\,,\,
u+a\tau\,,\, v+b\tau\,,\, t) = \tilde{\Phi}^\bracm (\tau \,,\,
z_1+(b-a)\tau\,,\, z_2-(b+a)\tau\,,\, t) =
\tilde{\Phi}^{\bracm}(\tau \,,\, z_1+j\tau \,,\, z_2+k\tau \,,\,
t)$, where $j=b-a$, $k=-(b+a)$.  Hence under the map
(\ref{eq:6.12}), (\ref{eq:6.13}) formulae from
Theorem~\ref{th:6.5} correspond to those of Theorem~\ref{th:6.1}.
\end{proof}

\begin{remark}
  \label{rem:6.6}
If $m=0$, then Theorem~\ref{th:6.5} holds for an arbitrary
positive integer $M$.  Indeed, in this case $gcd (M,m+1)=1$,
Lemma~\ref{lem:6.2} still holds with $I$ replaced by
$I'=I_{M,m+1}$ and $b$ replaced by $b'=\tfrac12 b$ (proof is the
same).  Also $\Theta_{j,1} - \Theta_{-j,1} =0$ ($j \in \ZZ/2\ZZ$),
hence $\varphi^{[0]}_\add =0$, and $\tilde{\varphi}^{[0]} =
\varphi^{[0]}$, $\tilde{\Phi}^{[0]}= \Phi^{[0]}$.  Therefore we
have for any $M \geq 1$:
\begin{eqnarray*}
 && \Phi^{[0]} \left(  -\frac{M}{\tau}\,,\, \frac{z_1}{\tau}\,,\,
    \frac{z_2}{\tau}\,,\, t-\frac{z_1z_2}{\tau M}\right)\\
   &&  = \frac{\tau}{M} \sum_{j,k\in \ZZ /M\ZZ} q^{\frac{jk}{M}}
       e^{\frac{2\pi i}{M} (kz_1+jz_2)} \Phi^{[0]}
         (M\tau\,,\, z_1+j\tau\,,\, z_2+k\tau\,,\, t)\, .
\end{eqnarray*}
\end{remark}

\section{Modular transformation formulae for modified 
  normalized characters of admissible
  $\hat{s\ell}_{2|1}$-modules}
\label{sec:7}
Recall (cf. Section~\ref{sec:4}) that in order to have $SL_2
(\ZZ)$-invariance, we need to take, along with characters and
supercharacters, also twisted characters and supercharacters.
For the twisted 
$\hat{s\ell}_{2|1}$
-modules we choose 
\begin{equation}
\label{eq:7.0}
\xi =  -\tfrac12 (\alpha_1 + \alpha_2).  
\end{equation}
Throughout this section we shall work, as in Section \ref{sec:5}, 
in the following
coordinates of the Cartan subalgebra $\hat{\fh}$ of $\hat{s\ell}_{2|1}$:  
\begin{equation}
 \label{eq:7.a}
h=2\pi i (-\tau \Lambda_0 - z_1\alpha_2-z_2\alpha_1 +t\delta)
:=(\tau,z_1,z_2,t).
\end{equation} 
In particular, we have:
\begin{equation}
  \label{eq:7.1}
  t_{-\xi} (\tau, z_1, z_2, t) = \left( \tau, z_1 + \tfrac{\tau}{2}, z_2 +
    \tfrac{\tau}{2}, t+\tfrac{z_1+z_2}{2}+\tfrac{\tau}{4} \right)\,.
\end{equation}

As in Section~\ref{sec:4}, throughout this section the
superscripts $(1/2)$ and $(0)$ will refer to characters and
supercharacters, while the subscripts $0$ and $1/2$ will refer to
non-twisted and twisted sectors respectively.

By (\ref{eq:4.9})--(\ref{eq:4.12}) 
we have the following formula for the normalized affine
denominators $(\epsilon,\epsilon' = 0 \hbox{\,\,or\,\,} \tfrac12)$:
\begin{equation}
  \label{eq:7.2}
  \hat{R}^{(\epsilon)}_{\epsilon'} (\tau,z_1,z_2,t) =(-1)^{2\epsilon(1 -2\epsilon')}
    ie^{2\pi i t} \frac{\eta (\tau)^3 \vartheta_{11} (\tau,z_1+z_2)} 
       {\vartheta_{1-2\epsilon',1-2\epsilon} (\tau,z_1) 
           \vartheta_{1-2\epsilon',1-2\epsilon}(\tau, z_2)} \,.
\end{equation}
By Theorem~\ref{th:4.2} we have the following modular
transformation formulae:
\begin{equation}  
\label{eq:7.3}
    \hat{R}^{(\epsilon)}_{\epsilon'} \left( -\tfrac{1}{\tau}\,,\,
      \tfrac{z_1}{\tau}\,,\, \tfrac{z_2}{\tau},\,t  \right) 
     = (-1)^{4\epsilon\epsilon'} i\tau e^{\frac{2\pi i z_1z_2}{\tau}}
         \hat{R}^{(\epsilon')}_\epsilon (\tau,z_1,z_2,t)\,;
\end{equation}  
\begin{equation}  
    \label{eq:7.4}
  \hat{R}^{(\epsilon)}_{\epsilon'} (\tau+1,z_1,z_2, t)  = 
e^{\pi  i \epsilon'}
 \hat{R}^{(|\epsilon-\epsilon'|)}_{\epsilon'} (\tau,z_1,z_2,t)\, .
\end{equation}

Fix a positive integer $M$ and a 
non-negative integer $m$, such  that $gcd (M,2m+2)=1$ if $m>0$.  
In connection with the study of the numerators of the normalized
characters of admissible $\hat{s\ell}_{2|1}$-modules introduce
the following functions ($\epsilon, \epsilon' =0 $ or $\tfrac12$,
$j,k \in \epsilon'+\ZZ $):

\begin{equation}
\label{eq:7.6}
  \Psi^{[M,m;\epsilon]}_{j,k;\epsilon'} (\tau,z_1,z_2,t) 
    = q^{\frac{(m+1)jk}{M}}
       e^{\frac{2\pi i (m+1)}{M} (kz_1+jz_2)}
\Phi^{[m]} (M\tau,z_1+j\tau+\epsilon, z_2+k\tau+\epsilon ,\frac{t}{M})\, ,
\end{equation}
and denote by  
$\tilde{\Psi}^{[M,m;\epsilon]}_{j,k;\epsilon'} (\tau,z_1,z_2,t)$ 
the function given by the same formula, except that $\Phi$ is replaced by
$\tilde{\Phi}$.
Since the functions (\ref{eq:6.15}) are independent of the choice
of $j,k \mod M\ZZ$, the same holds for the functions 
$$
e^{\frac{2\pi i (m+1)\epsilon}{M}(j+k)}\tilde{\Psi}^{[M,m;\epsilon]}_{j,k;\epsilon'} (\tau,z_1,z_2,t). 
$$
The following theorem is immediate by Theorem~\ref{th:6.5}(a) 
and Remark~\ref{rem:6.6}.

\begin{theorem} 
\label{th:7.1}Let $\epsilon, \epsilon'=0$ or  $1/2 $, and let 
$j,k \in \epsilon'+ \ZZ/M\ZZ$.  Then 
\begin{eqnarray*}
\tilde{\Psi}^{[M,m;\epsilon]}_{j,k;\epsilon'} \left( -\tfrac{1}{\tau}\,,\,  \tfrac{z_1}{\tau}\,,\, \tfrac{z_2}{\tau}\,,\, t \right)  = 
\frac{\tau}{M} e^{\tfrac{2\pi i (m+1)}{M} z_1z_2} 
   \sum_{a,b \in \epsilon + \ZZ/M\ZZ}    e^{-\frac{2\pi i
       (m+1)}{M} (ak+bj)}
   \tilde{\Psi}^{[M,m;\epsilon']}_{a,b;\epsilon}
   (\tau,z_1,z_2,t)\, ;\\
\lefteqn{\hspace{-6in}\tilde{\Psi}^{[M,m;\epsilon]}_{j,k;\epsilon'} \left(\tau+1,\,z_1,\,z_2,\,t \right)  = 
e^{2\pi i \frac{(m+1)jk}{M}}\tilde{\Psi}^{[M,m;\epsilon+\epsilon']}_{j,k;\epsilon'} \left(\tau,\,z_1,\,z_2,\,t \right). }
\end{eqnarray*}
\end{theorem}

Now we link the functions $\Psi^{[M,m;\epsilon]}_{j,k;\epsilon'}$
to the normalized characters of admissible
$\hat{s\ell}_{2|1}$-modules of level $K=\frac{m+1}{M}-1$.  Recall
(cf. Proposition~\ref{prop:3.13}) that we have admissible weights
of this level $K$ of two types $(j,k \in \ZZ_{\geq 0})$:
\begin{eqnarray*}
  \Lambda^{(1)}_{j,k} &=& \left((j+k+1) \frac{m+1}{M}-1  \right)
     \Lambda_0 -j \frac{m+1}{M} \Lambda_1 -k\frac{m+1}{M}\Lambda_2
        ,\,\, 0 \leq j,k \,,\, j+k \leq M-1\, ;\\[1ex]
   \Lambda^{(2)}_{j,k} &=& -\left((j+k-1) \frac{m+1}{M}+1  \right)      
      \Lambda_0 +j \frac{m+1}{M} \Lambda_1 +k\frac{m+1}{M}\Lambda_2
      , \,\,   1 \leq j,k\,,\, j+k \leq M \,.
\end{eqnarray*}
Also in coordinates (\ref{eq:7.a}) we have for the corresponding
simple subsets $S_1$ (resp. $S_2$) $=\{\tilde{\alpha}_i|\,\,i=0,1,2\}$:
\begin{eqnarray}
\label{eq:7.b}
\tilde{\alpha}_1(h) &=&-2\pi i (z_1+k_1\tau) ,\,\,\, 
\,\tilde{\alpha}_2(h)=-2\pi i (z_2+k_2\tau) ,     \,\\
(\hbox{resp.\,\,} \tilde{\alpha}_1(h)&=&-2\pi i (-z_1+k_1\tau) ,\,\, 
\,\tilde{\alpha}_2(h)= -2\pi i (-z_2+k_2\tau) )\, . \nonumber
\end{eqnarray}

As before, it is convenient to introduce the following notation
for normalized untwisted and twisted characters and
supercharacters:
\begin{equation}
\label{eq:7.c}
  \ch^+_\Lambda = \ch^{(1/2)}_{\Lambda;0}\,\,\, ,\, \, \,
      \ch^-_{\Lambda} =\ch^{(0)}_{\Lambda ;0}\, \,\,,\,\,\,
         \ch^{\tw,+}_{\Lambda} = \ch^{(1/2)}_{\Lambda ;1/2}\,\,\,,\,\,\,
            \ch^{\tw,-}_\Lambda = \ch^{(0)}_{\Lambda;1/2}\,.
\end{equation}

\begin{proposition}
  \label{prop:7.2}

  \begin{list}{}{}
  \item (a) If $\Lambda = \Lambda^{(1)}_{j,k}$, where $j,k \in
    \ZZ$,\, $0 \leq j,k, j+k \leq M-1$, then
    \begin{equation*}
      (\hat{R}^{(\epsilon)}_{\epsilon'}
      \ch^{(\epsilon)}_{\Lambda ; \epsilon'}) (\tau, z_1, z_2, t) 
         = \Psi^{[M,m;\epsilon]}_{j+\epsilon', k+\epsilon';\epsilon'}
        (\tau,\,z_1,\,z_2,\,t).       
\end{equation*}

\item (b)  If $\Lambda = \Lambda^{(2)}_{j,k}$, where $j,k \in
  \ZZ$,\, $1 \leq j,k, j+k \leq M$, then
  \begin{equation*}
    (\hat{R}^{(\epsilon)}_{\epsilon'}
      \ch^{(\epsilon)}_{\Lambda ; \epsilon'}) (\tau, z_1, z_2, t) 
    = \Psi^{[M,m;\epsilon]}_{M+\epsilon'-j, M+\epsilon'-k;\epsilon'}
     (\tau,\,z_1,\,z_2,\,t).        
  \end{equation*}

  \end{list}
\end{proposition}

\begin{proof}
We explain how to derive (a), the proof of (b) being similar.  

First, the (super)character of an admissible module
is obtained from the (super)character of the corresponding partially integrable
module $L(\Lambda^0)$, as described by formula  (\ref{eq:3.a}).
This amounts to replacing the simple roots $\alpha_i\in \Pi$ by the
$\tilde{\alpha}_i\in S_1,\, i\in I$, described by  (\ref{eq:7.b}), in the 
numerator of the (super)character of $L(\Lambda^0)$.
Thus, the supercharacter of an admissible $\hat{\sl}_{2|1}$-module 
$L(\Lambda^{(1)}_{j,k})$ 
is obtained from that of $L(m\Lambda_0)$ 
by the following substitution in the RHS of
(\ref{eq:5.3}): the $z_i$ are replaced according to (\ref{eq:7.b}),
and $t$ is 
replaced by 
$\Lambda^{(1)}_{j,k}(h)=\frac{t+(m+1)(z_1k+z_2j)}{M}$.

In order to deduce the formula for the normalized
supercharacter (\ref{eq:4.5}), we find for $\Lambda = \Lambda^{(1)}_{j,k}$:
$$
m_\Lambda= \frac{m+1}{M}jk.
$$
Then (a) for 
$\epsilon=\epsilon'=0$ 
follows. 

To deduce (a) for $\epsilon'=0,\,\epsilon=\frac{1}{2}$, note that the 
character is obtained from the supercharacter by replacing $z_i,\,i=1,2,$
by $z_i + \frac{1}{2}$. Finally, to deduce (a)
for $\epsilon'=\frac{1}{2}$, $\epsilon=0$ or
 $\frac{1}{2}$, we should replace $\Lambda$ by
 $\Lambda^{\tw}$, which, according to (\ref{eq:4.a}), (\ref{eq:4.b}), 
(\ref{eq:4.14a}) and (\ref{eq:4.c}), amounts to replacing 
$j$ and $k$ by  $j+\frac{1}{2}$ and $k+\frac{1}{2}$, and replacing
$\ch^\pm$ by $t_\xi(\ch^\pm)$. 
\end{proof}

In order to state the modular transformation formula for the
modified normalized admissible characters it is convenient to
change notation as follows:
\begin{equation*}
  \ch^{[M,m;\epsilon]}_{j+\epsilon',k+\epsilon';\epsilon'} :=
    \ch^{(\epsilon)}_{\Lambda^{(1)}_{j,k;\epsilon'}}\,,\,\,\,
     \ch^{[M,m;\epsilon]}_{M+\epsilon'-j,M+\epsilon'-k;\epsilon'}
       := \ch^{(\epsilon)}_{\Lambda^{(2)}_{j,k;\epsilon'}}\,.
\end{equation*}
Then
\begin{equation*}
  \{ \ch^{[M,m;\epsilon]}_{j+\epsilon',k+\epsilon';\epsilon'} |\,\,
    j,k \in \ZZ \,,\, 0 \leq j,k \leq M-1 ,\,\epsilon,\epsilon'=0, 1/2\}
\end{equation*}
is precisely the set of all admissible characters
(resp. supercharacters) if $\epsilon'=0$ and $\epsilon =1/2$
(resp. $\epsilon =0$), and it is the set of all twisted
admissible characters (resp. supercharacters) if $\epsilon'=1/2$
and $\epsilon =1/2$ (resp. $\epsilon =0$).  In view of these
observations, introduce the {\it modified} normalized characters
$(\epsilon =1/2\,,\, \epsilon'=0)$, supercharacters $(\epsilon
=0 \,,\, \epsilon'=0)$, twisted characters $(\epsilon =1/2 \,,\,
\epsilon'=1/2)$, and twisted supercharacters $(\epsilon =0\,,\,
\epsilon'=1/2)$, letting 
\begin{equation}
\label{eq:7.5}
  \tilde{\ch}^{[M,m;\epsilon]}_{j,k;\epsilon'} (\tau,z_1,z_2,t)    =\frac{\tilde{\Psi}^{[M,m;\epsilon]}_{j,k;\epsilon'} (\tau,z_1,z_2,t)}      {\hat{R}^{(\epsilon)}_{\epsilon'}(\tau,z_1,z_2,t)}, \,\,         j,k \in \epsilon'+\ZZ \,,\, 
0 \leq j,k <M \, .
\end{equation}
Then from (\ref{eq:7.3}), (\ref{eq:7.4}) and Theorem~\ref{th:7.1}
we obtain the following theorem.

 \begin{theorem}
  \label{th:7.3}
Let $M$ be a positive integer and let $m$ be a non-negative integer, 
such that $gcd (M,2m+2)=1$ if $m>0$.    
One has the following modular transformation
formulae ($\epsilon, \epsilon'=0$
or $\frac12; \,j,k \in
 \epsilon' + \ZZ \,,\, 0 \leq j,k <M)$:
 \begin{eqnarray*}
   &&  \tilde{\ch}^{[M,m;\epsilon]}_{j,k;\epsilon'}
    \left(-\frac{1}{\tau}\,,\, \frac{z_1}{\tau}\,,\,
       \frac{z_2}{\tau}\,,\, t-\frac{z_1z_2}{\tau} \right)\\[1ex]
   &&  =-(-1)^{4\epsilon\epsilon'} \frac{1}{M}
       \sum_{\substack{a,b \in \epsilon + \ZZ \\ 0 \leq a,b < M}}
       e^{-\frac{2\pi i (m+1)}{M} (ak+bj)}
      \tilde{\ch}^{[M,m;\epsilon']}_{a,b;\epsilon} (\tau,z_1,z_2,t)\,;\\[1ex]
      &&  \tilde{\ch}^{[M,m;\epsilon]}_{j,k;\epsilon'}
         (\tau+1,z_1,z_2,t) =e^{2\pi i \frac{(m+1)jk}{M}-\pi i \epsilon'}
    \tilde{\ch}^{[M,m;|\epsilon-\epsilon'|]}_{j,k;\epsilon'}
            (\tau,z_1,z_2,t)\, .
 \end{eqnarray*}
 \end{theorem}

\begin{remark}
\label{rem:7.4}
If $m=0$, we have admissible $\hat{sl}_{2|1}$-modules of  
{\it boundary level} 
$K=\frac{1}{M}-1$ 
\cite{KW4}. In this case, by Remark \ref{rem:6.6},
$\tilde{\Phi}^{[0]}=\Phi^{[0]}$, 
and, since $\ch^-_0=1$, we have, by
(\ref{eq:5.3}), 
(\ref{eq:5.4}):
\begin{equation*}
\Phi^{[0]}=\hat{R}_0^{(0)}.
\end{equation*}
(Note that this is the famous Ramanujan summation formula for the
bilateral basic hypergeometric function $_1\Psi_1$, cf. \cite{KW3}.)
Hence in this case (\ref{eq:7.5}) becomes ($j
,k \in \epsilon'+\ZZ , 0 \leq j,k <M $): 
\begin{equation*}
  \tilde{\ch}^{[M,0;\epsilon]}_{j,k;\epsilon'} (\tau,z_1,z_2,t)    =
q^{\frac{jk}{M}}e^{\frac{2\pi i}{M} (kz_1+jz_2)}
\frac{\hat{R}^{(\epsilon)}_{\epsilon'} (M\tau,z_1+j\tau +\epsilon,z_2+k\tau+\epsilon,t)}      
{\hat{R}^{(\epsilon)}_{\epsilon'}(\tau,z_1,z_2,t)}. 
\end{equation*}
These are the normalized characters and supercharacters of all admissible untwisted and twisted modules of boundary level $K=\frac{1}{M}-1$. In this case
one needs no modifications, and modular transformation formulae for
these characters and supercharacters is given by Theorem \ref{th:7.3}
for $m=0$ with tildes removed. 
\end{remark}

\section{Modular transformation formulae for modified normalized characters of 
admissible $\hat{A}_{1|1}$-modules}
\label{sec:8}

Consider the Lie superalgebra $g\ell_{2|2}$, endowed with the
structure of a Kac--Moody superalgebra as in
Example~\ref{ex:3.4}, and let $\hat{g\ell}_{2|2}$ be the
corresponding affine Lie superalgebra (see Section~\ref{sec:2}).  On
the Lie superalgebra $\sl_{2|2}$ the supertrace  form is
degenerate with kernel $\CC I_4$, but it induces a non-degenerate
invariant bilinear form, which we again denote by $\bl$, on the
Lie superalgebra $A_{1|1}=ps\ell_{2|2}(=\sl_{2|2}/\CC I_4)$.  The associated
affinization (see Section~\ref{sec:2})
\begin{equation*}
  \hat{A}_{1|1} = A_{1|1} [t,t^{-1}]\oplus \CC K \oplus \CC d
\end{equation*}
is a Lie superalgebra with a non-degenerate invariant bilinear
form $\bl$ induced from $\hat{g\ell}_{2|2}$.

Throughout this section, $\fg =A_{1|1}$ and $\hat{\fg} =
\hat{A}_{1|1}$.  The Lie superalgebras $\fg$ and $\hat{\fg}$
inherit from $\gl_{2|2}$ and $\hat{\gl}_{2|2}$ all the basic
features of a Kac--Moody superalgebra, discussed in
Sections~\ref{sec:1} and \ref{sec:2}, like the root space
decomposition, the triangular decomposition, the Weyl group,
etc.  The Cartan subalgebra $\fh$ of $\fg$ is, by definition, the
quotient of the space of diagonal matrices from $\sl_{2|2}$ by
$\CC I_4$. (Thus, $\dim \fh = 2$, $|I|=3$, the corank of the Cartan
matrix is $1$, and $\Pi$ is not a linearly independent set,
so that $\fg$ is, strictly speaking, not a
Kac--Moody superalgebra, but a simple variation of it.)  The Cartan subalgebra $\hat{\fh}$ of $\hat{\fg}$
is, as before, defined by (\ref{eq:2.7a}).

It is easy to see that an irreducible highest weight module
$L(\Lambda)$ over $\hat{\gl}_{2|2}$ is actually a
$\hat{\fg}$-module, provided that $\Lambda (I_4)=0$, and it
remains irreducible when restricted to $\hat{\fg}$.  The
condition $\Lambda (I_4)=0$ is equivalent to the following
condition on labels of $\Lambda$ (defined by (\ref{eq:3.5}),
(\ref{eq:3.6})):
\begin{equation}
  \label{eq:8.1}
  m_1=m_3 \, .
\end{equation}

Let $\Pi = \{ \alpha_1,\alpha_2,\alpha_3 \}$ be the set of
simple roots of $\fg$ (note that $\alpha_1 =\alpha_3$).  We denote 
$\alpha_{12}=\alpha_1+\alpha_2$, 
$\alpha_{23}=\alpha_2+\alpha_3$, $\alpha_{123}=\alpha_1+\alpha_2+\alpha_3$, 
etc, for short.
We have:
\begin{eqnarray*}
 \lefteqn{\hspace*{1.2in}\sdim \fg = -2 ,\,\ell=2,\,\Delta^+_{\bar{0}} =
    \{ \alpha_2,\, \theta=\alpha_{123}\}\, , 
  \Delta^+_{\bar{1}} =\{ \alpha_1,\alpha_3,\alpha_{12}, \alpha_{23}\}\,,}\\
  &&  (\alpha_1|\alpha_1) = (\alpha_3|\alpha_3) =(\alpha_1|\alpha_3) 
    =0\,,\, (\alpha_2|\alpha_2)=-2\,,\, (\alpha_1|\alpha_2)
    =(\alpha_2|\alpha_3)=1\,,\, (\theta |\theta)=2\, , \,\\
&& 2\rho_{\bar{0}} = \alpha_{12}+\alpha_{23}\,,\,\rho_{\bar{1}}=\theta\,,\,
  2\rho =-\alpha_{13} \,,\, h^\vee =0 \, .
\end{eqnarray*}
We choose $\xi \in \fh$ by letting 
\begin{equation}
\label{eq:8.a}
  (\xi|\alpha_1) = (\xi|\alpha_3)=\tfrac12\,,\, (\xi |\alpha_2)=0\,,\,
\end{equation}
so that
\begin{equation*}
  (\rho_{\bar{0}}|\xi ) = \tfrac12 \,,\,
  (\rho_{\bar{1}}|\xi)= 1\,,\, (\rho |\xi)= -\tfrac12.
\end{equation*}

We choose the following coordinates in $\hat{\fh}$ (cf. (\ref{eq:4.8})):
\begin{equation}
  \label{eq:8.2}
   h=2 \pi i (-\tau \Lambda_0-(z_1+z_2)\alpha_1-z_1\alpha_2+t\delta)
  := (\tau,z_1,z_2,t)\, .
\end{equation}
Note that for $z=-(z_1+z_2) \alpha_1-z_1\alpha_2$ we have: $(z|z)= 2z_1z_2$.

By (\ref{eq:4.9})--(\ref{eq:4.12}) we have the following
formulae for the normalized affine denominators $(\epsilon,
\epsilon'=0$ or $\tfrac12)$:
\begin{equation}
  \label{eq:8.3}
  \hat{R}^{(\epsilon)}_{\epsilon'} (\tau,z_1,z_2)
     = (-1)^{2\epsilon'}\eta (\tau)^4
   \frac{\vartheta_{11} (\tau,z_1-z_2) \vartheta_{11}(\tau,z_1+z_2)} {\vartheta_{1-2\epsilon',1-2\epsilon} (\tau,z_1)^2 \vartheta_{1-2\epsilon',1-2\epsilon} (\tau, z_2)^2\,} . 
\end{equation}
Note that these denominators are independent of $t$ since
$h^\vee=0 $.

By Theorem~\ref{th:4.2}, we have the following modular
transformation formulae:
\begin{eqnarray}
  \label{eq:8.4}
 && \hat{R}^{(\epsilon)}_{\epsilon'} \left( -\tfrac{1}{\tau} \,,\, 
    \tfrac{z_1}{\tau}\,,\, \tfrac{z_2}{\tau} \right) 
   =(-1)^{2(\epsilon-\epsilon')} i\tau e^{\frac{\pi i z_1z_2}{\tau}}
   R^{(\epsilon')}_\epsilon (\tau,z_1,z_2) \, ,  \\
  \label{eq:8.5}
&& \hat{R}^{(\epsilon)}_{\epsilon'}(\tau+1,z_1,z_2)
    = e^{\pi i (2\epsilon'-\frac12)} 
        \hat{R}^{(|\epsilon-\epsilon'|)}_{\epsilon'}(\tau,z_1,z_2)\,.
\end{eqnarray}
%

Next, we study modular transformation properties of the numerator
of the normalized supercharacter of the partially integrable
$\hat{\fg}$-module $L(m\Lambda_0)$, where $m$ is a non-zero
integer (see Example~\ref{ex:3.4}), using formula (\ref{eq:3.15a})
for $\ch^-_{L(m\Lambda_0)}$.  

We have: 
$W=\{1, r_{\alpha_2}, r_\theta, r_{\alpha_2}r_\theta\} $.    
We choose $T_\rho = \{ \alpha_1,\alpha_3\}$ in 
(\ref{eq:3.15a}) . 
Since $r_{\alpha_2} r_\theta$ fixes $e^{m\Lambda_0+\rho}/(1-e^{-\alpha_1})(1-e^{-\alpha_3})$, we deduce that $j_0 =2$, and the summation in 
(\ref{eq:3.15a}) is over the semidirect product of the group
$\{1, r_{\alpha_2}\} $ and $t_{L^\#}$. Furthermore, 
$L^\#= \ZZ \theta$ if $m>0$, and  $L^\#= \ZZ \alpha_2$ if $m<0$.      
Hence formula (\ref{eq:3.15a}) gives, using (\ref{eq:2.10}), the following
supercharacter formulae, where $m$ is a positive integer: 
\begin{equation}
\label{eq:8.b}  
e^{-m\Lambda_0-\rho}\hat{R}^- \ch^-_{L (m\Lambda_0)} =
\sum_{j \in \ZZ} \left( \frac{e^{jm\theta}q^{mj^2+j}}{(1-e^{-\alpha_1}q^{j})(1-e^{-\alpha_3}q^j)} - \frac{e^{-\alpha_2-jm\theta} q^{mj^2+j}}     
{(1-e^{-\alpha_1-\alpha_2}q^j)(1-e^{-\alpha_2-\alpha_3}q^j )}\right)\,,
\end{equation}
\begin{equation}
\label{eq:8.c}  
e^{m\Lambda_0-\rho}\hat{R}^- \ch^-_{L (-m\Lambda_0)} =
\sum_{j \in \ZZ}   \left( \frac{e^{jm\alpha_2}q^{mj^2+j}}{(1-e^{-\alpha_1}q^{j})(1-e^{-\alpha_3}q^j)} - \frac{e^{-\alpha_2-jm\alpha_2} q^{mj^2+j}}     
{(1-e^{-\alpha_1-\alpha_2}q^j)(1-e^{-\alpha_2-\alpha_3}q^j )}\right)\,.
\end{equation}
After passing to the normalized supercharacter, equation (\ref{eq:8.b}) 
in coordinates (\ref{eq:8.2}) looks as follows:
\begin{eqnarray}
  \label{eq:8.6}
  &&(\hat{R}^{(0)}_0\ch^-_{m\Lambda_0}) (\tau,z_1,z_2,t)\\
  &&  = e^{2\pi i mt} 
\sum_{j\in \ZZ} \left(\frac{e^{2\pi i jm (z_1+z_2)}e^{2\pi iz_1}q^{mj^2+j}}
{(1-e^{2\pi iz_1}q^j)^2} - \frac{e^{-2\pi i jm (z_1+z_2)} e^{-2\pi iz_2} 
q^{mj^2+j}} {(1-e^{-2\pi i z_2} q^j)^2}\right).\nonumber
\end{eqnarray}
We denote the RHS of (\ref{eq:8.6})
by $\Phi^{A_{1|1}\bracm} (\tau,z_1,z_2,t)$. 
 Recall that $m$ is a positive integer. 
\begin{lemma}
\label{lem:8.1}
$$
\Phi^{A_{1|1}\bracm}(\tau,z_1,z_2,t) =D_0
\Phi^{[m-1]} (\tau,z_1,z_2,t),
$$
where 
$$
 D_0=\tfrac{1}{2\pi i} \left( \tfrac{\partial}{\partial z_1}- \tfrac{\partial}{\partial z_2}\right), 
$$
and $\Phi^{[m-1]}$ is the RHS of (\ref{eq:5.3}) with $m$ replaced by
$m-1$.
\end{lemma}

\begin{proof}
  It is straightforward, using that 
$D_0e^{2\pi i jm  (z_1+z_2)}=0 $.
\end{proof}

Next, introduce the following differential operator:
$$ 
D_1=
\tfrac{1}{2\pi i} \left( \tfrac{\partial}{\partial z_1}- 
\tfrac{\partial}{\partial z_2}- 
\tfrac{z_1-z_2}{2\tau}\tfrac{\partial}{\partial t}  \right). 
$$
It is immediate to check that
\begin{equation}
\label{eq:8.7}
(D_1F)|_S=\tau D_1(F|_S),  
\end{equation}
where (cf. (\ref{eq:4.6}) and (\ref{eq:4.6a}))
$$
 (F|_S)(\tau ,z_1,z_2,t) =\tau^{-1}F\left(-\frac{1}{\tau},\frac{z_1}{\tau},
\frac{z_2}{\tau}, t-\frac{z_1z_2}{\tau}  \right).
$$
Consider the following non-meromorphic {\em modification}
of the {\em numerator}  $\Phi^{A_{1|1}\bracm}$:
\begin{equation*}  
\tilde{\Phi}^{A_{1|1}\bracm}(\tau,z_1,z_2,t) 
=D_1 \tilde{\Phi}^{[m-1]} 
(\tau,z_1,z_2,t)\,
\end{equation*}
and let (see (\ref{eq:7.6}))
\begin{equation*}  
\tilde{\Psi}^{A_{1|1}[M,m;\epsilon]}_{j,k;\epsilon'}(\tau,z_1,z_2,t) 
=D_1 \tilde{\Psi}^{[M,m-1;\epsilon]}_{j,k;\epsilon'}(\tau,z_1,z_2,t). 
\end{equation*}

\begin{theorem} 
\label{th:8.2}
(a) If $m$ is a positive integer, then 
\begin{equation*}
\tilde{\Phi}^{A_{1|1}[m]} 
(-\frac{1}{\tau}\,,\, \frac{z_1}{\tau}\,    ,\,\frac{z_2}{\tau}\,,\, t-\frac{z_1z_2}{\tau})      
= \tau^2 \tilde{\Phi}^{A_{1|1}[m]}(\tau,z_1,z_2,t).
\end{equation*}
(b) Let $M$ and $m$ be positive integers, such that $gcd(M,2m)=1$
if $m>1$.
Let $\epsilon, \epsilon'=0$ or  $1/2 $, 
and let $j,k \in \epsilon'+ \ZZ/M\ZZ$. Then 
\begin{equation*}
\tilde{\Psi}^{A_{1|1}[M,m;\epsilon]}_{j,k;\epsilon'} 
\left( -\tfrac{1}{\tau}\,,\,  \tfrac{z_1}{\tau}\,,\, \tfrac{z_2}{\tau}\,,\, 
t-\frac{z_1z_2}{\tau} \right)  
= \frac{\tau^2}{M} 
\sum_{a,b \in \epsilon + \ZZ/M\ZZ}    e^{-\frac{2\pi i m}{M} (ak+bj)}    
\tilde{\Psi}^{A_{1|1}[M,m;\epsilon']}_{a,b;\epsilon} (\tau,z_1,z_2,t)\, .
\end{equation*}
\end{theorem}
\begin{proof}
It follows immediately from Corollary 
\ref{cor:5.10} and Theorem \ref{th:7.1}, using (\ref{eq:8.7}).
\end{proof}

Next, we link the functions 
$\tilde{\Psi}^{A_{1|1}[M,m;\epsilon]}_{j,k;\epsilon'}(\tau,z_1,z_2,t)$ 
to the normalized characters of
admissible $\hat{A}_{1|1}$-modules of level $K=\tfrac{m}{M}$,
where $m$ and $M$ are positive coprime integers, for which the corresponding
partially integrable module is $L(m\Lambda_0)$.  The highest
weights of these modules are admissible with respect to the simple subsets
of type $S_j$, $j=1,2,3,4$, described in Example~\ref{ex:3.10}.
These highest weights are listed in Proposition~\ref{prop:3.14}
, and they should satisfy the additional condition
(\ref{eq:8.1}).  
Below we introduce a more convenient indexing of
them, where $k_i \in \ZZ_{\geq 0}$, $i=0,1,2,3$, $k_1=k_3$ 
(the last condition is 
equivalent to (\ref{eq:8.1})).  

Recall that
\begin{equation*}
  S_1 = \{ \tilde{\alpha}_i = k_i \delta + \alpha_i|\,i=0,1,2,3,\,
     \sum^3_{i=0} k_i = M-1 \,\}.
\end{equation*}
In coordinates (\ref{eq:8.2}) we have:
\begin{equation}
  \label{eq:8.8}
  \tilde{\alpha}_1 (h) =-2 \pi i (z_1+k_1\tau)\,,\, 
     (\tilde{\alpha}_1+\tilde{\alpha}_2 ) (h) = -2\pi i 
        (z_2+(k_1+k_2)\tau)\,.
\end{equation}
Let $j=k_1$, $k=k_1+k_2$.  Then the pairs $(j,k)$ that determine
the corresponding highest weight $\Lambda^{(1)}_{k_1,k_2}$
, run
over the following set of pairs of integers:
\begin{equation}
  \label{eq:8.9}
  k \geq j \geq 0\,,\, j+k \leq M-1\,.
\end{equation}
We denote the corresponding admissible highest weights by $\Lambda_{jk}$, and let $s=1$.

The simple subsets of the second type are
\begin{equation*}
  S_2 = \{\tilde{\alpha}_i=k_i\delta - \alpha_i|\,i=0,1,2,3,\,
    \sum^3_{i=0} k_i = M+1,\, k_i>0\,\}.
\end{equation*}
In coordinates (\ref{eq:8.2}) we have:
\begin{equation}
  \label{eq:8.10}
  \tilde{\alpha}_1(h)=-2\pi i(-z_1+k_1\tau)\,,\,
     (\tilde{\alpha}_1+ \tilde{\alpha}_2) (h) 
     = -2\pi i (-z_2+(k_1+k_2)\tau)\, .
\end{equation}
Let $j=M-k_1$, $k=M-k_1-k_2$.  Then the pairs $(j,k)$ 
that determine the corresponding highest weight $\Lambda^{(2)}_{k_1,k_2}$
run over
the following set of pairs of integers:
\begin{equation}
  \label{eq:8.11}
  M-1 \geq j > k \geq 1 \,,\, j+k \geq M \, .
\end{equation}
We denote the corresponding admissible highest weights by
$\Lambda_{jk}$, and let $s=2$.

The simple subsets of the third type are 
\begin{equation*}
  S_3 = \{ \tilde{\alpha}_0 = k_0 \delta + \alpha_0,
     \tilde{\alpha}_1=k_1\delta+\alpha_{12}, 
    \tilde{\alpha}_2=k_2\delta-\alpha_2,
    \tilde{\alpha}_3 = k_3 \delta +\alpha_{23}|\,
      \sum^3_{i=0} k_i = M-1,\, k_2>0\}.
\end{equation*}
In coordinates (\ref{eq:8.2}) we have:
\begin{equation}
  \label{eq:8.12}
  \tilde{\alpha}_1(h) = -2 \pi i (z_2+k_1\tau)\,,\,\,\,\,
  (\tilde{\alpha}_1 +\tilde{\alpha}_2) (h)
    = -2\pi i (z_1 + (k_1+k_2)\tau)\,.
\end{equation}
Let $j=k_1+k_2$, $k=k_1$.  Then the pairs $(j,k)$ 
that determine the corresponding highest weight $\Lambda^{(3)}_{k_1,k_2}$
run over the
following set of pairs of integers:
\begin{equation}
  \label{eq:8.13}
0 \leq k<j\,,\, j+k \leq M-1\, .
\end{equation}
We denote the corresponding admissible weights by $\Lambda_{jk}$, and let $s=3$.

The simple subsets of the fourth type are
\begin{equation*}
 S_4 =  \{\tilde{\alpha}_0 =k_0 \delta - \alpha_0\,,\, 
    \tilde{\alpha}_1=k_1\delta -\alpha_{12},\,
      \tilde{\alpha}_2 =k_2\delta +\alpha_2,\,
       \tilde{\alpha}_3 = k_3\delta-\alpha_{23}|\,
       \sum^3_{i=0} k_i = M+1,\, k_0,k_1>0 \} .
\end{equation*}
In coordinates (\ref{eq:8.2}) we have:
\begin{equation}
  \label{eq:8.14}
  \tilde{\alpha}_1 (h) =-2\pi i (-z_2+k_1\tau)\,,\,
  (\tilde{\alpha}_1 + \tilde{\alpha}_2)(h)
  =-2 \pi i (-z_1+(k_1+k_2)\tau)\,.
\end{equation}
Let $j=M-k_1-k_2$, $k=M-k_1$.  Then the
 pairs $(j,k)$ that determine the corresponding highest weight 
$\Lambda^{(4)}_{k_1,k_2}$ run over the following set of
pairs of integers:
\begin{equation}
  \label{eq:8.15}
  1 \leq j \leq k \leq M-1\,,\, j+k \geq M\,.
\end{equation}
We denote the corresponding admissible weights by $\Lambda_{jk}$, and let $s=4$.

Summing up, we observe the following.
\begin{remark}
\label{rem:8.3}
The sets of pairs $(j,k)$ indexing
the admissible highest weights 
$\Lambda_{jk}$,
fill up without overlapping the set of pairs of
integers in the square $\{ (j,k) \in \ZZ^2 |\,\,0 \leq j,k \leq M-1\}$.
\end{remark}

\begin{proposition}
\label{prop:8.4}
We have the following formula for the normalized supercharacters
of all admissible $\hat{A}_{1|1}$-modules $L(\Lambda_{jk})$:
\begin{equation*}
(\hat{R}^{(0)}_0 ch^-_{ \Lambda_{jk}})(\tau,z_1,z_2,t)=
\epsilon_s q^{\frac{mjk}{M}}e^{\frac{2\pi i m}{M}(kz_1+jz_2)}
\Phi^{A_{1|1}\bracm}\left(M\tau, z_1+j\tau, z_2+ k\tau\,,\, \frac{t}{M}\right),
\end{equation*}
where  $\epsilon_s=(-1)^{\frac{(s-1)(s-2)}{2}}$, and the function
$\Phi^{A_{1|1}[m]}$ is given by Lemma \ref{lem:8.1}.  
\end{proposition}
\begin{proof}
Instead of proving the proposition in the same way as Proposition
\ref{prop:7.2},
we shall use formula (\ref{eq:3.d}) for $\Lambda^0=m\Lambda_0$. 
For this we need to compute 
$\beta\in \fh^*$ and $y\in W$, such that $S_s=t_\beta y(S_{(M)})$
for each $s=1,2,3,4$. A straightforward computation gives:
\begin{eqnarray*}
  s=1 &:&  y=1 ,\, \beta=-(k_1+\frac12 k_2)\alpha_{13}-k_1
        \alpha_2\,; \\ [1ex]
    s=2 &:& y=r_{\alpha_{123}}r_{\alpha_2},\, \beta=(k_1+\frac12
          k_2)\alpha_{13}+k_1\alpha_2\, ;\\[1ex]
      s=3 &:&  y=r_{\alpha_2}, \,\beta=-(k_1+\frac12
            k_2)\alpha_{13}-(k_1+k_2)\alpha_2\, ;\\[1ex]
       s=4 &:&  y=r_{\alpha_{12}},\, \beta=(k_1+\frac12
             k_2)\alpha_{13}+(k_1+k_2)\alpha_2\, .
\end{eqnarray*}
Now, using (\ref{eq:3.b}), we see that the highest weight coincides with the
$\Lambda^{(s)}_{k_1,k_2},\,s=1,2,3,4$, up to adding a multiple of $\delta$,
which can be ignored since it does not change the normalized character.
Applying formula (\ref{eq:3.d}) and using the correspondence
$\Lambda^{(s)}_{k_1,k_2} \mapsto \Lambda_{jk}$ gives the result
after a straightforward computation in all four cases. 
\end{proof}

We use the same notation (\ref{eq:7.c}) as before for normalized untwisted and 
twisted characters and supercharacters, and use $\xi$ given by (\ref{eq:8.a})
for the twisted characters and supercharacters, given by (\ref{eq:4.b}).
Then, as in Section \ref{sec:7}, we get from Proposition \ref{prop:8.4},
after the change of notation, similar to (\ref{eq:7.c}), the following
unified formula for normalized characters and supercharacters of all
admissible twisted and untwisted $\hat{A}_{1|1}$-modules ($\epsilon, 
\epsilon'=0$ or $\frac{1}{2}$, $j,k\in \epsilon'+\ZZ$): 
\begin{equation}
\label{eq:8.16} 
(\hat{R}^{(\epsilon)}_{\epsilon'}{\ch}^{(\epsilon)}_{\Lambda_{jk};\epsilon'}) 
(\tau,z_1,z_2,t)=
\epsilon_s q^{\frac{mjk}{M}}e^{\frac{2\pi i m}{M}(kz_1+jz_2)}\Phi^{A_{1|1}\bracm}\left(M\tau, z_1+\epsilon +j\tau, z_2+\epsilon + k\tau\,,\, \frac{t}{M}\right).
\end{equation}
In view of this formula,
%
%
%
we introduce the modified 
(super)characters of untwisted and twisted
$\hat{A}_{1|1}$-modules by the formula $(\epsilon, \epsilon' = 0$
or $\tfrac12$; $j,k \in \epsilon'+\ZZ$, $0\leq j,k <M)$:
\begin{equation*}
\tilde{\ch}^{A_{1|1}[M,m;\epsilon]}_{j,k;\epsilon'}(\tau, z_1, z_2, t) =
   \frac{\tilde{\Psi}^{A_{1|1}[M,m;\epsilon]}_{j,k;\epsilon'}
      (\tau,z_1,z_2,t)} {\hat{R}^{(\epsilon)}_{\epsilon'} (\tau,z_1,z_2)}\,,
\end{equation*}
where the denominators are given by (\ref{eq:8.3}).  
Recalling Lemma \ref{lem:8.1}, we see that the modification amounts to
replacing the meromorphic function 
$\Phi^{[m-1]}$
by its non-meromorphic modification $\tilde{\Phi}^{[m-1]}$
(as in the case of $\sl_{2|1}$), and, furthermore, putting the operator
$D_0$ in front and replacing it by $D_1$ (and dropping the unessential sign
$\epsilon_s$). 

The following theorem follows immediately from (\ref{eq:8.4})
and Theorem~\ref{th:8.2}.

\begin{theorem}
  \label{th:8.4}

Let $M$ and $m$ be positive integers, such that $gcd (M,2m) =1$
if $m>1$.  Let $\epsilon,\epsilon' =0$ or $\tfrac12$ and let $j,k
\in \epsilon' + \ZZ /M\ZZ$.  Then:
  \begin{eqnarray*}
 &&\tilde{\ch}^{A_{1|1}[M,m;\epsilon]}_{j,k;\epsilon'} 
    \left( - \tfrac1\tau\,,\, \tfrac{z_1}{\tau}\,,\,
      \tfrac{z_2}{\tau} \,,\,  t-\tfrac{z_1z_2}{\tau}\right)\\ 
&&  = -(-1)^{2(\epsilon-\epsilon')}\tfrac{i\tau}{M} 
      \displaystyle{\sum_{a,b\in \epsilon + \ZZ/M\ZZ}} 
        e^{-\tfrac{2 \pi i m}{M} (ak+bj)}
          \tilde{\ch}^{A_{1|1}[M,m;\epsilon']}_{a,b;\epsilon}
            (\tau,z_1,z_2,t).
\end{eqnarray*}

\end{theorem}

\begin{remark}
  \label{rem:8.5}
In the case $m=1$ we have:
\begin{equation*}
  \tilde{\Phi}^{A_{1|1}[1]} (\tau, z_1,z_2,0)=
    \Phi^{A_{1|1}[1]} (\tau, z_1,z_2,0)=
     D_0 \Phi^{(0)} (\tau, z_1,z_2,0)=
       -iD_0 \frac{\eta (\tau)^3 \vartheta_{11}(\tau,z_1+z_2)}
          {\vartheta_{11}(\tau,z_1)\vartheta_{11}(\tau,z_2)}\,.
\end{equation*}
Hence, in view of Lemma \ref{lem:8.1} and formula (\ref{eq:8.3}),
the supercharacter formula (\ref{eq:8.6}) for $m=1$ becomes:
$$
 \hat{R}^{(0)}_0(\tau, z_1, z_2)\ch^-_{\Lambda_0} (\tau,z_1,z_2,0)
=D_0( \hat{R}^{(0)}_0(\tau, z_1, z_2)),
$$
where
$$
 \hat{R}^{(0)}_0(\tau, z_1, z_2)=\eta(\tau)^4
 \frac{\vartheta_{11} (\tau,z_1-z_2) \vartheta_{11}(\tau,z_1+z_2)} 
{\vartheta_{11} (\tau,z_1)^2 \vartheta_{11} (\tau, z_2)^2\,} . 
$$
\end{remark}

In the remainder of this section we shall study the partially
integrable $\hat{A}_{1|1}$-modules $L(-m\Lambda_0)$ where $m$ is
a positive integer, and the corresponding admissible
$\hat{A}_{1|1}$-modules of level $K=-\frac{m}{M}$, where $M$ is a
positive integer, coprime with $m$.

We introduce coordinates, different from (\ref{eq:8.2}):
\begin{equation}
  \label{eq:8.20}
  h=2\pi i (-\tau \Lambda_0 - (z_1-z_2) \alpha_1-z_1\alpha_2+t\delta)\,.
\end{equation}
In these coordinates the normalized supercharacter of $L
(-m\Lambda_0)$ looks exactly like the RHS of (\ref{eq:8.6})
(except for the sign change of $t$):
\begin{equation}
  \label{eq:8.21}
  (\hat{R}^{(0)}_0 \ch^-_{-m\Lambda_0} )(\tau, z_1,z_2,t)
     = e^{-2\pi i mt} \sum_{j \in \ZZ}
       \left( \frac{e^{2\pi ijm (z_1+z_2)} e^{2\pi i z_1} q^{mj^2+j}}
        {(1-e^{2\pi i z_1} q^j)^2} 
        - \frac{e^{-2\pi i m (z_1+z_2)} e^{-2\pi i z_2} q^{mj^2+j}}
         {( 1-e^{-2\pi i z_2} q^{j})^2 }\right).
\end{equation}
Note that the RHS of this equation is the function
$\Phi^{A_{1|1}[m]} (\tau,z_1,z_2,-t)$.  Note also that passing
from coordinates (\ref{eq:8.2}) to coordinates (\ref{eq:8.20}) 
amounts to the change of sign of $z_2$, hence the
normalized affine denominators (\ref{eq:8.3}) remain unchanged.

Next, we derive a formula for normalized supercharacters of level
$K=-\frac{m}{M}$, where $m$ and $M$ are positive coprime
integers, for which the corresponding partially integrable module
is $L (-m\Lambda_0)$.  As in the case of positive level $K$,
there are four types of such modules, with highest weights
$\Lambda^{(s)}_{k_1,k_2} \,\, (s=1,2,3,4) $ given by the same
formulae as in Proposition~\ref{prop:3.14} except that $m$ is
replaced by $-m$.  We consider the same reparametrization
$\Lambda_{jk}$ of these highest weights
(cf.
(\ref{eq:8.9}),(\ref{eq:8.11}),(\ref{eq:8.13}),(\ref{eq:8.15})),
so that, according to Remark~\ref{rem:8.3}, the set of pairs
$(j,k)$, indexing this admissible highest weights, is the set of
integer points in the square $0 \leq x,y \leq M-1$.

As in the case of positive level, we use formula (\ref{eq:3.d})
in order to compute the normalized  superhcaracters
$ch^-_{\Lambda^{(s)}_{k_1,k_2}}$ of level $-\frac{m}{M}$ in
coordinates (\ref{eq:8.20}).  We get:
\begin{eqnarray}  
\label{eq:8.22} 
\lefteqn{\hspace{-5.15in} 
(\hat{R}^{(0)}_0 \ch^-_{\Lambda^{(s)}_{k_1,k_2}}) 
        (\tau,z_1,z_2,t)}\\
  \,\,\,\,\, = \epsilon_s q^{-\frac{m}{M}jk} e^{\frac{2\pi i m}{M}(-kz_1+jz_2)}
       (D_0 \Phi^{[m-1]}) (M\tau, z_1 + j\tau,z_2-k\tau, -\frac{t}{M})      
         \hbox{\,\, if \,\,} s=1 \hbox{\,\, or \,\,} 3\,.\nonumber
\end{eqnarray}
%
%
%
By Lemma~\ref{lem:5.1}(b),(c) and by
formula (\ref{eq:3.d}), we have:
\begin{eqnarray*}
   (\hat{R}^{(0)}_0 \ch^-_{\Lambda^{(2)}_{k_1,k_2}}) (\tau,z_1,z_2,t)
      &= &q^{-\frac{m}{M} k_1 (k_1+k_2)} 
          e^{\frac{2\pi i m}{M}((k_1+k_2)z_1-k_1z_2)}\\ 
%
%
    &\times &D_0 \left(\Phi^{[m-1]} 
         (M\tau,z_1-k_1\tau\,,\,
       z_2+(k_1+k_2) \tau  \, , \, -\frac{t}{M})\, \right)\\
%
%
%
%
%
%
   &=& \epsilon_2 q^{-\frac{m}{M}(j-M)(k-M)} e^{\frac{2\pi im}{M}(-(k-M)z_1+(j-M)z_2)}
     \left( D_0 \Phi^{[m-1]}\right)\\
      &&\left(M\tau_,z_1+(j-M)\tau, z_2-(k-M)\tau,-\frac{t}{M}  \right)\, .
\end{eqnarray*}
A similar calculation shows that an analogous formula  holds
also for $\Lambda^{(4)}_{k_1,k_2}$ (for the corresponding values
of $j$ and $k$).
Combining this with (\ref{eq:8.22}), we obtain the
following unified formula for $s=1,2,3,4$:
\begin{eqnarray}  
\label{eq:8.25} 
&&(\hat{R}^{(0)}_0
   \ch^-_{\Lambda^{(s)}_{k_1,k_2}})    (\tau,z_1,z_2,t)\\ 
 & = &\epsilon_s q^{-\frac{m}{M} j'k'} e^{\frac{2\pi im}{M}(-k'z_1+j'z_2)}
      (D_0 \Phi^{[m-1]}) 
        \left(M\tau_,z_1+j'\tau, z_2-k'\tau,-\frac{t}{M}  \right)\, ,\nonumber
\end{eqnarray}
where $j'=j,\,k'=k$ in cases $s=1,3$, and 
$j'=j-M,\,k'=k-M$ in cases $s=2,4$.

Next, we compute the normalized supercharacters of all twisted
admissible $\hat{A}_{1|1}$-modules of negative level
$K=-\frac{m}{M}$. For this, consider the following automorphism
of the Cartan subalgebra $\hat{\fh}$ of
$\hat{\fg}=\hat{A}_{1|1}$:
\begin{equation*}
  w=t_{-\frac12\alpha_2} r_{\alpha_2}\,.
\end{equation*}
Choosing a lifting $\tilde{r}_{\alpha_2}$ of $r_{\alpha_2}$ in the
corresponding $SL_2 (\CC)$, we can lift $w$ to an isomorphism
$\tilde{w}=t_{-\frac12\alpha_2} \tilde{r}_{\alpha_2}\,:
\hat{\fg}^{\tw} \overset{\sim}{\to}  \hat{\fg}$,
cf. Section~\ref{sec:4}. As in Section \ref{sec:4}, via this isomorphism, the
$\hat{\fg}$-module $L(\Lambda)$ becomes a $\hat{\fg}^{\tw}$-module
$L^{\tw}(\Lambda)$, with the highest weight
$\Lambda^{\tw}=w(\Lambda)$, and the character and supercharacter
$\ch^\pm_{L^{\tw}(\Lambda)}=w (\ch^\pm_{L(\Lambda)})$.

We have in coordinates (\ref{eq:8.20}):
\begin{equation*}
  w^{-1}(\tau,z_1,z_2,t) = \left( \tau, -z_2-\frac{\tau}{2},-z_1-
      \frac{\tau}{2}, t-\frac{\tau}{4}-\frac{z_1 +z_2}{2}
\right),
\end{equation*}
hence the normalized twisted character of an admissible twisted
$\hat{A}_{1|1}$-module $L^{\tw} (\Lambda)$ is computed by the
following formula:
\begin{eqnarray*}
 \lefteqn{\hspace{-3.5in} (\hat{R}^{(0)}_\frac12 \ch^{-,\tw}_{\Lambda}) 
       (\tau,z_1,z_2,t)}\\
  = q^{\frac{m}{4M}} e^{\frac{\pi i m}{M}(z_1+z_2)}
    (\hat{R}^{(0)}_0 \ch^-_\Lambda) (\tau,-z_2 - \frac{\tau}{2},-z_1
    - \frac{\tau}{2}, t)\, .
\end{eqnarray*}

Using this formula, we derive from (\ref{eq:8.25}) the following
formula for the normalized twisted supercharacter (by making use
of the properties of $\Phi^{[m]}$, given by Lemma~\ref{lem:5.1}(b), (c), (d)):
\begin{eqnarray}
  \label{eq:8.26}
  (\hat{R}^{(0)}_\frac12 \ch^{-,\tw}_{\Lambda^{(s)}_{k_1,k_2}})
      (\tau,z_1,z_2,t) = \epsilon_s q^{-\frac{m}{M}(j'-\frac12)(k'+\frac12)}
       e^{\frac{2\pi i m}{M} (-(j'-\frac12)z_1+(k'+\frac12)z_2)}\\
   \times \left(D_0\Phi^{[m-1]}\right) \left(M\tau ,z_1 
         + \left(k'+\tfrac12\right)\tau,  
   z_2+ (j'-\tfrac12)\tau, -\tfrac{t}{M}\right). \nonumber
\end{eqnarray}

As in the case of positive level $K=\frac{m}{M}$ considered
above, we express the normalized characters of non-twisted and
twisted admissible $\hat{A}_{1|1}$-modules of negative level
$K=-\frac{m}{M}$ via their normalized supercharacters, given by
(\ref{eq:8.25}) and (\ref{eq:8.26}), and we use for them a similar
notation $\ch^{A_{1|1}[M,-m;\epsilon ]}_{j,k;\epsilon'}$.
Similarly, we define their non-holomorphic  modifications
($\epsilon,\epsilon' = 0$ or $\frac12; j,k \in \epsilon'+\ZZ$) :
 \begin{eqnarray}
\label{eq:8.27}
 (\hat{R}^{(\epsilon)}_0 \tilde{\ch}^{A_{1|1}[M,-m;\epsilon]}_{j,k;\epsilon'}) 
   (\tau, z_1,z_2, t) 
&=& D_1 \tilde{\Psi}^{[M,m-1;\epsilon]}_{j,-k;0}(\tau,z_1,z_2,-t)\,;\\[1ex]
(\hat{R}^{(\epsilon)}_{\frac12} 
   \tilde{\ch}^{A_{1|1}[M,-m;\epsilon]}_{j,k;\epsilon'})(\tau, z_1,z_2, t) 
&=& D_1 \tilde{\Psi}^{[M,m-1;\epsilon]}_{k+\frac12,-j+ \frac12 ; \frac12}
  (\tau, z_1,z_2,- t) \,. \nonumber
 \end{eqnarray}

Since the pairs $(j,k)$ fill up all the integral points in the
square $0 \leq j,k < M$, consider the sets ($\epsilon' =0$ or $\frac12$):
\begin{equation*}
  \Omega^{A_{1|1}}_{\epsilon'} = \left\{ (j,k) |\,\,\, j,k \in \epsilon'+\ZZ
     \quad 0 \leq j <M\, ;\, -\frac12 \leq k <M-\frac12 \right\}\,.
\end{equation*}
This set parametrizes the normalized modified characters and
supercharacters in the non-twisted case when $\epsilon' =0$ and
in the twisted case when $\epsilon'=\frac12$.  Then the
numerators of the modified non-twisted characters and
supercharacters are:
\begin{eqnarray*}
  D_1 \tilde{\Psi}^{[M,m-1;\epsilon]}_{j,-k;0} (\tau,z_1,z_2,t)\,, \quad
     (j,k) \in \Omega^{A_{1|1}}_0\,, \\
\noalign{\hbox{and that of the twisted ones are:}}\\
   D_1\tilde{\Psi}^{[M,m-1;\epsilon]}_{j,-k;\frac12}(\tau,z_1,z_2,-t)\,,
   \quad  (j,k) \in \Omega^{A_{1|1}}_\frac12 \,.
\end{eqnarray*}

In the same way as above we obtained in Theorem~\ref{th:8.4} a
modular transformation formula in the case of the positive level
$K=\frac{m}{M}$, we otain now a similar formula in the case of
the negative level $K=-\frac{m}{M}$.

\begin{theorem}
  \label{th:8.7}

Let $M$ and $m$ be positive integers, such that $gcd (M,2m) =1$
if $m >1$.  Let $\epsilon, \epsilon' = 0$ or $\frac12$ and let
$j,k \in \Omega^{A_{1|1}}_{\epsilon'}$.  Then
\begin{eqnarray*}
  \tilde{\ch}^{A_{1|1}[M,-m;\epsilon]}_{j,k;\epsilon'}
  \left( -\frac{1}{\tau}\,,\, \frac{z_1}{\tau}\,,\,    \frac{z_2}{\tau}
  \,,\, t+ \frac{z_1z_2}{\tau}\right)\\
 = (-1)^{2(\epsilon-\epsilon')} \frac{\tau}{iM}
    \sum_{a,b \in \epsilon+\ZZ/M \ZZ}e^{\frac{2\pi i m}{M}(ak+bj)}
      \tilde{\ch}^{A_{1|1} [M,-m ;\epsilon']}_{a,b;\epsilon}
      (\tau,z_1,z_2,t)\,.
\end{eqnarray*}

\end{theorem}

\begin{remark}  
\label{rem:8.8}

In coordinates (\ref{eq:8.2}) we have, for $z=-(z_1+z_2) \alpha_1-z_1\alpha_2
\, : \quad (z|z) =2z_1z_2$, while in coordinates (\ref{eq:8.20})
we have, for $z=-(z_1-z_2)\alpha_1-z_1\alpha_2\, : \quad (z|z)
=-2z_1z_2$.  Thus, the modular transformations of
Theorem~\ref{th:8.4} and~\ref{th:8.7} are consistent with the usual
action of $S=\left(
  \begin{smallmatrix}
    0 & -1\\ 1& 0
  \end{smallmatrix}\right) \in SL_z (\ZZ)$, given by formula~(\ref{eq:4.6}).
  
\end{remark}

\begin{remark}  
\label{rem:8.9}
Each  summand of the RHS of the transformation formulae for the modified 
characters in Theorems
\ref{th:8.4} and \ref{th:8.7} remains unchanged  after adding to $a$ or $b$ 
an integer
multiple of $M$, but this is not the case for the modified characters
in these summands.
\end{remark}

\section{Modular transformation formulae for modified characters of  
admissible $N=2$ modules}
\label{sec:9}

Let $\fg =\sl_{2|1}$ or $A_{1|1}$ ($=p\sl_{2|2}$), and let
$\fh$ be its Cartan subalgebra; recall that $\ell :=\dim \fh = 2$
in both cases.  Let $x=\frac12 \theta \in \fh^* =\fh$, where
$\theta$ is the highest root, which we assume to be even.  
With respect to $\ad x$ we have
the following eigenspace decomposition:
\begin{equation*}
  \fg = \fg_{-1} + \fg_{-\frac12} + \fg_0 + \fg_{\frac12}+\fg_1\,,
\end{equation*}
where $\fg_{\pm 1} =\CC e_{\pm\theta}$, $\fg_{\pm \frac12}$
are purely odd, and $\fg_0=\CC x+\fg^\#$, where $\fg^\#$
is the orthogonal complement to $x$ in $\fg_0$ with respect to
$\bl$.  The subalgebra $\fg^\#$ is spanned by the element
$J_0=\alpha_2-\alpha_1$ in the case of $\fg =\sl_{2|1}$.  The
subalgebra $\fg^\#$ is isomorphic to $\sl_2$ in the case $\fg
=A_{1|1}$, and $\fg^\#\cap \fh$ is spanned by the element $J_0=-\alpha_2$.

Recall that the quantum Hamiltonian reduction associates to a
$\hat{\fg}$-module $L(\Lambda)$, such that $K+h^\vee \neq 0$, a
module $H(\Lambda)$ over the corresponding superconformal
algebra, which is $N=2$ (resp. $N=4$) algebra if $\fg =\sl_{2|1}$
(resp. $\fg=A_{1|1}$), for which the following properties hold,
\cite{KRW}, \cite{KW5}, \cite{A1}:

\romanparenlist
                    \begin{enumerate}{}{}

\item 
the module $H(\Lambda)$  is either $0$ or an
  irreducible positive energy module over the superconformal
  algebra;

\item 
$H (\Lambda)=0$ iff $(\Lambda|\alpha_0) \in\ZZ_{\geq 0}$;

\item 
the irreducible module $H(\Lambda)$ is characterized
  by three numbers:
%
%

%
%
\begin{list}{}{}
\item{} $(\alpha)$ the central charge
  \begin{equation}
\label{eq:9.1}
    c_K = c(K) -6K+h^\vee -4,\,
 \hbox {where}\,\, c(K)= \frac{K \sdim \fg}{K+h^\vee},
\end{equation}
\item{}  $(\beta)$ the lowest energy (i.e. the minimal eigenvalue of $L_0$)
  \begin{equation}
\label{eq:9.2}
    h_\Lambda = h(\Lambda)-(x+d|\Lambda),\,
\hbox{where}\,\, h (\Lambda)=\frac{(\Lambda + 2 \hat{\rho}|\Lambda)}{2(K+h^\vee)}, 
\end{equation}
\item{}  $(\gamma)$ the spin 
 \begin{equation}
\label{eq:9.3}
s_\Lambda =\Lambda (J_0);
\end{equation}
                    \end{list}

 \item 
 the (super)character of the module $H(\Lambda)$ is given by the
 following formula:
 \begin{equation}
\label{eq:9.4}
  \ch^\pm_{H(\Lambda)}(\tau,\,z):=
\tr^\pm_{H(\Lambda)}q^{L_0-\frac{c_K}{24}}e^{2\pi i zJ_0 }
=(\hat{R}^\pm \ch^\pm_{\Lambda})
     (\tau, -\tau x + J_0 z, 0)
\ON (\tau, z)^{-1},
\end{equation}
where
$$
\ON
(\tau, z)= \eta(\tau)^\ell \prod^\infty_{n= 1}\frac{\prod_{\substack{\alpha \in \Delta_+\\      \alpha (x) =0}}(1-q^{n-1} e^{-2\pi i z \alpha (J_0)})(1-q^n e^{2\pi i z \alpha (J_0)})}
{\prod_{\substack{\alpha \in      \Delta_+\\ \alpha (x) =1/2}} (1 \pm q^{n-\tfrac12} e^{2\pi    i z \alpha (J_0)})} 
$$
is the $N$ superconformal algebra denominator and superdenominator.
  \end{enumerate}
It is easy to rewrite these denominators in terms of the four Jacobi theta functions, using (\ref{eq:4.7}). We have:
%
%
\begin{eqnarray}  
\label{eq:9.0} 
\lefteqn{\hspace{-4in}   q^{\frac{2d_{\bar{0}} + d_{\bar{1}} /2}{24}}
e^{2\pi iz \rho (J_0)}  i^{d_{\bo} -1}      \ON (\tau ,z) }\\[1ex]
 =\eta (\tau)^{\ell -d_{\bar{0}} +d_{\bar{1}}/2 +1 }        \frac{\prod_{\substack{\alpha \in \Delta_+\\ \alpha              (x)=0}}\vartheta_{11} (\tau ,z \alpha (J_0))}           {\prod_{\substack{\alpha \in \Delta_+ \\ \alpha (x)                 =1/2,\, \alpha (J_0) >0}} \vartheta_{0a}             (\tau , z \alpha (J_0))}\,,  \nonumber
\end{eqnarray}
where $a=0$ (resp. $1$) in the case of $+$ (resp. $-$).  
Here and
further we assume that $\alpha (J_0) \neq 0$ for all $\alpha \in
\Delta_{\bo}$, which is the case for $N=2$ or $4$ (but not for
$N=1$ and $3$).

Now we turn to the Ramond twisted sector. For each $\alpha \in \Delta_+$
choose an integer if $\alpha$ is even (resp. $\frac12 +$ integer if $\alpha$
is odd), denoted by $s_{\alpha}$, cf. (\ref{eq:4.3a}), such that $s_\theta =0$ 
and $s_\alpha + s_{\theta -\alpha}=0$ if 
$\alpha ,\, \theta-\alpha \in \Delta_{\bar{1},+}$.  
Recall \cite{KW6}, \cite{A1} that, 
given such a suitable choice of $s_\alpha$'s, 
the twisted Hamiltonian reduction associates to a
$\hat{\fg}^{\tw}$-module $L^{\tw}(\Lambda)$ of non-critical level (i. e. 
$\Lambda(K)+h^\vee \neq 0$), a positive energy
module $H^{\tw}(\Lambda)$ over the corresponding Ramond
twisted superconformal algebra, for which the properties (i) and
(ii) hold with $H$ replaced by $H^{\tw}$.
The irreducible module $H^{\tw} (\Lambda) $
is again characterized by three numbers:  the same central charge
$c_K$, given by (\ref{eq:9.1}), the lowest energy (i.e. the 
minimal eigenvalue of $L^{\tw}_0$)
\begin{eqnarray}
  \label{eq:9.5}
    h^{\tw}_\Lambda &=& h (\Lambda^{\tw})-(x+d|\Lambda^{\tw})+
\frac{1}{16}\sdim \fg_{\frac12} , \\
\noalign{\hbox{and the spin}}\nonumber \\
  \label{eq:9.6}
  s^{\tw}_{\Lambda} &=&\Lambda^{\tw}(J_0)+ \frac12 \sum_{\substack{\alpha \in\Delta\\\alpha (x) =1/2}}s_\alpha \alpha(J_0)  .
\end{eqnarray}
Furthermore, the (super)character of the module $H^{\tw}(\Lambda)$ is
given by the following formula:
\begin{equation}
  \label{eq:9.7}
  \ch^\pm_{H^{\tw}(\Lambda)} (\tau ,z):= \tr^\pm_{H^{\tw}(\Lambda)}
     q^{L^{\tw}_0-\frac{c_K}{24}} e^{2\pi izJ^{\tw}_0}
   = (\hat{R}^{tw,\pm} \ch^{\tw,\pm}_\Lambda)
    (\tau,-\tau x  +J_0 z\,,\, 0) 
\TON (\tau, z)^{-1},
\end{equation}
where
$$
\TON (\tau, z)=  \eta(\tau)^\ell \prod^\infty_{n=1} 
      \frac{              
\prod_{\substack{ \alpha \in \Delta_+\\  \alpha{(x)}=0}}               
(1-q^{n-1+s_\alpha}e^{-2\pi i z\alpha (J_0)})         
(1-q^{n-s_\alpha}e^{2\pi i z\alpha (J_0)})}         
{\prod_{\substack{\alpha \in\Delta_+\\\alpha (x) =1/2}} 
       (1\pm q^{n-\frac12 +s_\alpha} e^{2\pi iz\alpha (J_0)})}
$$
is the $N$ superconformal algebra twisted denominator and superdenominator.
It is easy to rewrite these denominators in terms of the four Jacobi theta functions, using (\ref{eq:4.7}). We have:
\begin{eqnarray}  
\label{eq:9.00} 
\lefteqn{\hspace{-4in}   
 i^{d_{\bo} -1} 
e^{\pi iz(\sum_{\alpha\in \Delta_+,\alpha(x)=0}(2s_\alpha-1)\alpha(J_0)+
\sum_{\alpha \in\Delta_{\bar{1},+}, s_\alpha>0 }
 \alpha(J_0))}q^{\frac{d_{\bar{0}}-1 + d_{\bar{1}} /2}{12}}
\,\TON (\tau ,z) }\\[1ex]
  =\eta (\tau)^{\ell -d_{\bar{0}} +d_{\bar{1}}/2 +1 }        
\frac{\prod_{\substack{\alpha \in \Delta_+\\ \alpha              
(x)=0}}\vartheta_{11} (\tau ,z \alpha (J_0))}           
{\prod_{\substack{\alpha \in \Delta_+ \\ \alpha (x)                 
=1/2,\, s_\alpha >0}} 
\vartheta_{1a}          (\tau , z \alpha (J_0))}\,,  \nonumber
 \end{eqnarray}
where $a=0$ (resp. $1$) in the case of $+$ (resp. $-$).  
The right hand sides of fomulas (\ref{eq:9.0}) (resp. (\ref{eq:9.00}))
are called the normalized untwisted (resp. twisted) denominators (for +) and
superdenominators (for -).

In order to write down modular transformation formulae for these
normalized denominators, we denote them by 
 $ \EON(\tau,z)$,
where $\epsilon,
\epsilon' =0$ or $\frac12$, and, as before, the superscript
refers to the normalized denominator (resp. superdenominator) if
$\epsilon =\frac12$ (resp. $\epsilon =0$), and the subscript
refers to the untwisted case, also called the Neveu--Schwarz
sector (resp. to the twisted case, also called the Ramond
sector), if,  not as before, $\epsilon'=\frac12$
(resp. $\epsilon'=0$.)  Then we have
\begin{equation}
  \label{eq:9.x}
  \EON (\tau ,z)  
= \eta (\tau)^{\ell +1-d_{\bar{0}} +d_{\bar{1}}/2} 
   \frac{\prod_{\substack{\alpha \in \Delta_{\bar{0},+}\\ 
       \alpha (x) =0}} \vartheta_{11} (\tau, z \alpha (J_0))}
   {\prod_{\substack{\alpha \in \Delta_{\bar{1},+}\\ s_\alpha >0}}
    \vartheta_{1-2 \epsilon',1-2\epsilon} (\tau, z\alpha (J_0)) }\,.
\end{equation}
Using (\ref{eq:9.x}) and Propositon \ref{prop:A7} we deduce the
following modular transformation formulae for these denominators:

\begin{proposition}
  \label{prop:9.2}
  \begin{alphaparenlist}{}{}
  \item (a) $\NEO \left( -\frac{1}{\tau} , \frac{z}{\tau} \right)
   = (-i)^{d_{\bar{0}} -1-(1-2\epsilon)(1-2\epsilon')
        d_{\bar{1}}/2} (-i\tau)^{\ell /2}$ 
    $e^{\frac{\pi i z^2}{2\tau} (h^\vee (J_0 |J_0)+
      \sum_{\alpha \in \Delta_{\bar{0},+}}\alpha
      (J_0)^2)}      \EON (\tau,z)\,. $  \vspace{2ex}

\item(b) $\NOE (\tau +1,z) =
    e^{\frac{\pi i}{12}(\dim \fg_{0}-\dim \fg_\frac12)}\NOE$,\,\,
    $\NRE  (\tau +1 ,z) =
     e^{\frac{\pi i}{12} (\dim \fg_{0} +\frac12 \dim\fg_{\frac12})}
        \NRH(\tau,z) $.

  \end{alphaparenlist}

\end{proposition}

For the rest of this section we shall
consider the case $\fg =\sl_{2|1}$.  Let $L(\Lambda)$
be an admissible $\hat{\fg}$-module  of level
$K=\frac{m+1}{M}-1$, where $M$ is a positive integer, $m$ is a
non-negative integer, and $gcd (M,2m+2)=1$ if $m>0$ (see
Section~\ref{sec:7}).  As a result of a quantum Hamiltonian
reduction of the $\hat{\fg}$-module $L(\Lambda)$, where $\Lambda=
m_0\Lambda_0 + m_1 \Lambda_1 + m_2 \Lambda_2$, we obtain a 
positive energy module
$H(\Lambda)$ over the Neveu--Schwarz type  $N=2$ superconformal
algebra.  Recall that $H(\Lambda)=0$ iff $m_0 \in \ZZ_{\geq 0}$, and that
$H(\Lambda)$ is irreducible otherwise.

If $M=1$, then the $\hat{\fg}$-module is partially integrable,
hence $m_0 \in \ZZ_{\geq 0}$, and therefore $H(\Lambda)=0$.  If $M>1$,
then it follows from Proposition~\ref{prop:3.13} that for an
admissible $\Lambda$ we have $m_0 \in \ZZ_{\geq 0}$ iff $k_0=0$.
Thus, in what follows we may assume that $M \geq 2$ and $k_0\neq
0$.  Then $H(\Lambda)$ is an irreducible module over the NS type
$N=2$ superconformal algebra.  The corresponding three
characteristic numbers are easy to compute:

\begin{eqnarray}
  \label{eq:9.8}
  c_K &=& -6K-3 =3 \left( 1-\frac{2m+2}{M}\right)\, ,\\
   \label{eq:9.9}
     h_\Lambda &=& \frac{Mm_1m_2}{m+1} - \frac{m_1+m_2}{2} \, , \\
 \label{eq:9.10}
   s_\Lambda &=& m_1-m_2\, .
\end{eqnarray}

As has been pointed out in Section~\ref{sec:4}, in order to get
modular invariant family of characters and supercharacters, we
need to introduce the Ramond twisted sector, for which we need a choice
of $\xi \in \fh^*$, satisfying (\ref{eq:4.3a}).  However, in order
to apply the twisted quantum Hamiltonian reduction we need a more
special choice of~$\xi$ (cf. \cite{KW6}).  In the case $\fg
=\sl_{2|1}$ we made the choice (\ref{eq:7.0}) of $\xi$ which gave nice
formulae for twisted characters and supercharacters, but
unfortunately, it is not compatible with the twisted quantum
Hamiltonian reduction.  Instead we make the choice
\begin{equation}
  \label{eq:9.a}
  \xi' = \frac12 (\alpha_1-\alpha_2)\, 
\end{equation}
and let $s_\alpha =-(\xi'|\alpha),\, \alpha\in \Delta_+$.
Then we have in coordinates (\ref{eq:7.a}):
\begin{equation}
  \label{eq:9.b}
  t_{-\xi'} (\tau,z_1,z_2,t) = (\tau, z_1 + \frac{\tau}{2},
       z_2 - \frac{\tau}{2}, t 
          + \frac{z_2-z_1}{2} - \frac{\tau}{4})\,.
\end{equation}
While the non-twisted normalized denominators and
super-denominators remain as given in formula~(\ref{eq:7.2}) for
$\epsilon' =0$, the twisted ones change, namely, we let
\begin{eqnarray*}
  \hat{R}^{(\epsilon)}_\frac12 (h) := \hat{R}^{(\epsilon)}_0
     (t_{-\xi'} (h))\, .
\end{eqnarray*}
The obtained denominators differ only by a sign from (\ref{eq:7.2}), 
and only when $\epsilon = \epsilon' = \frac12$, 
so we keep for them the same notation in the hope that no confusion may
arise.  The modular transformation formulae differ from
(\ref{eq:7.3}) and (\ref{eq:7.4}) 
also by a sign:
\begin{eqnarray}
  \label{eq:9.c}
  \hat{R}^{(\epsilon)}_{\epsilon'} \left( - \frac{1}{\tau} \,,\,
    \frac{z_1}{\tau}\,,\, \frac{z_2}{\tau}\,,\, t \right) = 
(-1)^{4\epsilon \epsilon'}\tau e^{\frac{2\pi
      i z_1z_2}{\tau}}  \hat{R}^{(\epsilon')}_{\epsilon}
       (\tau,z_1,z_2,t)\,;\\[1ex]
\label{eq:9.d}
 \hat{R}^{(\epsilon)}_{\epsilon'} (\tau +1,z_1,z_2,t) 
    = e^{-\pi  i \epsilon'}
          \hat{R}^{(|\epsilon- \epsilon'|)}_\epsilon (\tau,z_1,z_2,t)\, .
\end{eqnarray}
Using the above notation for the denominators, the same notation
as (\ref{eq:7.c}) for the characters, and the same proof as
that of Proposition~\ref{prop:7.2}, we obtain the following
analogue of it for the choice of $\xi'$ instead of $\xi$.

\begin{proposition}
  \label{prop:9.1}
  \begin{list}{}{}
  \item (a)  If $\Lambda = \Lambda^{(1)}_{j,k}$, where $j,k \in
    \ZZ$, $0 \leq j,k,j+k \leq M-1$, then 
\begin{equation*}
\left( \hat{R}^{(\epsilon)}_{\epsilon'} \ch^{(\epsilon)}_{\Lambda ;
   \epsilon'} \right) (\tau, z_1,z_2,t) = 
     \Psi^{[M,m:\epsilon]}_{j+\epsilon', k-\epsilon'; \epsilon'}
        (\tau,z_1,z_2,t)\, .
\end{equation*}

\item (b) If $\Lambda = \Lambda^{(2)}_{j,k}$, where $j,k \in \ZZ$,
  $1 \leq j,k,j+k \leq M$, then
\begin{equation*}
\left( \hat{R}^{(\epsilon)}_{\epsilon'} \ch^{(\epsilon)}_{\Lambda
    ;\epsilon'}\right) (\tau,z_1,z_2,t)
  =-\Psi^{[M,m;\epsilon]}_{M+\epsilon'-j,M-\epsilon' ; \epsilon'}
     (\tau, z_1,z_2,t)\,.
\end{equation*}
  \end{list}

\end{proposition}

For the twisted $\hat{\fg}$-module $L^{\tw}(\Lambda)$ the highest
weight is $\Lambda^{\tw}=t_{\xi'} (\Lambda)$, 
the quantum Hamiltonian reduction produces
a module $H^{\tw}(\Lambda)$ over the Ramond type $N=2$
superconformal algebra.  As in the non-twisted case,
$H^{\tw}(\Lambda) =0$ if $M=1$, and if $M>1$ and $m_0=0$.
Otherwise $H^{\tw}(\Lambda)$ is an irreducible positive energy
module with central
charge (\ref{eq:9.1}), the remaining two characteristic numbers
being
\begin{eqnarray}
  \label{eq:9.11}
 h^{\tw}_\Lambda &=&\frac{M m_1m_2}{m+1}-m_2
   - \frac{m+1}{4M} -\frac18 \,,  \\
\label{eq:9.12}
 s^{\tw}_\Lambda &=& m_1 - m_2 -\frac{m+1}{M}-\frac12 \, .
\end{eqnarray}

It follows from the description of admissible weights of level $K$
given in Proposition~\ref{prop:3.13} that the list of $\Lambda$,
for which $H (\Lambda)\neq 0$ (resp. $H^{\tw}(\Lambda)\neq 0$),
consists of two sets:
\begin{eqnarray}\label{eq:9.13}
  \A^{(1)} &=& \{ \Lambda^{(1)}_{k_1,k_2} |\, k_1,k_2 \in  \ZZ_{\geq 0}\,,\,
     k_1+k_2 \leq M-2 \} \,,\\
  \A^{(2)} &=& \{ \Lambda^{(2)}_{k_1,k_2} |\, k_1,k_2 \in  1+\ZZ_{\geq 0}\,,\,
     k_1+k_2 \leq M \}\,.\nonumber
\end{eqnarray}
Note that we have the following bijective map:
\begin{equation*}
  \nu :\A^{(1)}\to \A^{(2)}\,,\, 
    \nu(\Lambda^{(1)}_{k_1,k_2})
       = \Lambda^{(2)}_{k_2+1, k_1+1}.
\end{equation*}
It is immediate to see that
\begin{equation*}
  h_{\Lambda^{(1)}_{k_1,k_2}} = h_{\nu    (\Lambda^{(1)}_{k_1,k_2})},\,\,
     s_{\Lambda^{(1)}_{k_1,k_2}} = s_{\nu  (\Lambda^{(1)}_{k_1,k_2})}
\end{equation*}
for $\Lambda^{(1)}_{k_1,k_2} \in \A^{(1)}$, and the same holds
for $h^{\tw}$ and $s^{\tw}$.  Hence for the quantum Hamiltonian reduction it
suffices to consider only the highest weights
$\Lambda^{(1)}_{k_1,k_2} \in \A^{(1)}$.

In order to compute the characters and supercharacters 
of the corresponding $N=2$ modules $H(\Lambda )$ and
$H^{\tw}(\Lambda)$ we use formulae (\ref{eq:9.4}) and (\ref{eq:9.7}).

First, we have from (\ref{eq:9.x})
%
\begin{equation}
  \label{eq:9.14}
  \OR (\tau ,z)
  = \frac{\eta (\tau)^3 (-1)^{(1-2\epsilon)(1-2\epsilon')}} 
     {\vartheta_{1-2\epsilon', 1-2 \epsilon} (\tau ,z)}\, .
\end{equation}
%
Using (\ref{eq:9.14}) and Proposition \ref{prop:A7} (or Proposition 
\ref{eq:9.2}) we deduce the following 
modular transformation formulae of the $N=2$ normalized denominators.

\begin{lemma}
  \label{lem:9.1}
  \begin{list}{}{}
  \item (a)  $\OR \left(
      -\frac{1}{\tau} \, , \, \frac{z}{\tau} \right) =
    i^{4\epsilon \epsilon' -2\epsilon - 2\epsilon'}  \tau
    e^{-\frac{\pi i z^2}{\tau}}
      \OX (\tau ,z)$.

\item (b)  $\ORO (\tau +1,z) =\ORO (\tau ,z)$,\,\,
  $\ORH (\tau+1,z) =
  e^{\frac{\pi i}{4}}   
    \ORMH (\tau,z)$.
  \end{list}
\end{lemma}

Let  $M$ be an integer  $\geq 2$, and let $m \in \ZZ_{\geq 0}$ be such
that $gcd (M,2m+2)=1$ if $m>0$. Recall that we have irreducible positive energy
$N=2$ modules $H (\Lambda^{(1)}_{k_1,k_2})$
(resp. $H^{\tw}(\Lambda^{(1)}_{k_1,k_2})$) in the Neveu-Schwarz (resp. Ramond)
sector with central charge $c_K=3 \left( 1-\frac{2m
    +2}{M}\right)$, obtained by the
quantum Hamiltonian reduction
from the $\hat{s\ell}_{2|1}$-modules $L
(\Lambda^{(1)}_{k_1,k_2})$  
(resp. $L^{\tw}(\Lambda^{(1)}_{k_1,k_2})$)
of level $K=\frac{m+1}{M} -1$, where $\Lambda^{(1)}_{k_1,k_2}
\in \A^{(1)}$.

It will be convenient to introduce the following two reindexings of the
set of weights $\A^{(1)}$:
\begin{eqnarray*}
  \A_{NS} &=& \{ \Lambda_{jk} = \Lambda^{(1)}_{j-\tfrac12,
    k-\tfrac12} |\, j,k \in \tfrac12 + \ZZ_{\geq 0} \, , \, j+k \leq M -1 \}\, ,\\
  \A_{R} &=& \{ \Lambda_{jk} = \Lambda^{(1)}_{j-1,k} |\, j,k \in
  \ZZ_{\geq 0} \,,\, j>0 \,,\, j+k \leq M-1 \}.
\end{eqnarray*}
We let 
\begin{equation*}
  H_{NS} (\Lambda_{jk}) = H (\Lambda^{(1)}_{j-\tfrac12 ,
    k-\tfrac12})\, , \, j,k \in \A_{NS}  ; \,\,  H_R
  (\Lambda_{jk})=H (\Lambda^{(1)}_{j-1,k}),\, j,k \in \A_R \, .
\end{equation*}
It follows from (\ref{eq:9.9})--(\ref{eq:9.12}) that the lowest
energy and the spin of these $N=2$ modules with central charge
$c_K = 3 \left( 1-\frac{2m+2}{M} \right)$ are as follows:
\begin{equation}
  \label{eq:9.15}
  h^{NS}_{jk} = \frac{m+1}{M} jk - \frac{m+1}{4M}\,,\,\,\,
                  s^{NS}_{jk} = \frac{m+1}{M} (k-j)\,;
\end{equation}
\begin{equation}
  \label{eq:9.16}
  h^R_{jk} = 
\frac{m+1}{M} jk -\frac{m+1}{4M}-\frac18\,,\,\,\, 
        s^R_{jk} =\frac{m+1}{M} (k-j)+\frac12\, .
\end{equation}

Introduce the following notation for the normalized characters
and supercharacters of these $N=2$ modules:
\begin{eqnarray*}
   ch^{N=2[M,m;\epsilon]}_{j,k;\tfrac12} (\tau ,z) 
     &=& ch^\pm _{H_{NS} (\Lambda_{jk})} (\tau,z)\,,\,
        \Lambda_{jk} \in \A_{NS}\,;\\
   ch^{N=2[M,m;\epsilon]}_{j,k;0}(\tau ,z)
     &=& ch^\pm_{H_R (\Lambda_{jk})} (\tau ,z)\,,\,
         \Lambda_{jk} \in \A_R \,.
\end{eqnarray*}
Formulae (\ref{eq:9.4}) and (\ref{eq:9.7}) imply the following
expressions for these characters:
\begin{equation}
  \label{eq:9.17}
  (\OR
  ch^{N=2[M,m;\epsilon]}_{j,k;\epsilon'}) (\tau,z) =
  \Psi^{[M,m;\epsilon]}_{j,k; \epsilon'} (\tau,-z,z,0),
\end{equation}
where the functions $\Psi^{[M,m;\epsilon]}_{j,k;\epsilon'}
(\tau,z_1,z_2,t)$ are defined by (\ref{eq:7.6}) and $j,k \in
\epsilon' + \ZZ_{\geq 0}$, subject to restrictions $j+k \leq M-1 $, $j>0$.

Introduce the {\it modified} normalized $N=2$ characters and
supercharacters, letting 
\begin{equation*}
  ( \OR
    \tilde{ch}^{N=2[M,m;\epsilon]}_{j,k;\epsilon'}) (\tau,z)=
 \tilde {\Psi}^{[M,m;\epsilon]}_{j,k; \epsilon'} (\tau,-z,z,0),
\end{equation*}
where the modification $\tilde{\Psi}$ of $\Psi$ was introduced in
Section~\ref{sec:7} (after (\ref{eq:7.6})).  Theorem~\ref{th:7.1}
 along with Lemma \ref{lem:9.1} give the  following modular
 transformation properties of the modified normalized $N=2$
 characters and supercharacters.
 \begin{theorem}
   \label{th:9.2}
Let $M$ be an integer $\geq 2$ and let $m \in \ZZ_{\geq 0}$ be such that
$gcd (M,2m+2)=1$ if $m>0$.  Let $c_{M,m} = 3 \left(
  1-\frac{2m+2}{M}\right)$.  Let $\epsilon, \epsilon' =0$ or
$\tfrac12$, and let $\Omega^{(M)}_\epsilon = \{ (j,k) \in
(\epsilon +\ZZ_{\geq 0})^2|\, j+k \leq M-1$, $j>0 \}$.   Then we have the
following modular transformation formulae for
$\tilde{ch}^{N=2[M,m;\epsilon]}_{j,k;\epsilon'}$, $j,k \in \Omega^{(M)}_{\epsilon'}$:
\begin{eqnarray*}
  \tilde{ch}^{N=2[M,m;\epsilon]}_{j,k;\epsilon'} \left( - \frac{1}{\tau}
    \,,\, \frac{z}{\tau} \right) = e^{\frac{\pi i c_{M,m}}{6\tau}z^2}\,
\sum_{(a,b)\in \Omega^{(M)}_\epsilon} S^{[M,m,\epsilon,\epsilon']}_{(j,k),(a,b)}\,
    \tilde{ch}^{N=2[M,m;\epsilon']}_{a,b;\epsilon} (\tau,z) \,, \\
   \noalign{\hbox{where}}\\
   S^{[M,m,\epsilon,\epsilon']}_{(j,k),(a,b)} = (-i)^{(1-2\epsilon) (1-2\epsilon')}\frac{2}{M}
    e^{\frac{\pi i (m+1)}{M} (j-k)(a-b)}\sin \frac{m+1}{M}(j+k)(a+b)\pi\,;
\end{eqnarray*}
\begin{equation*}
  \tilde{ch}^{N=2[M,m;\epsilon]}_{j,k;\epsilon'} (\tau +1,z)=
e^{\frac{2\pi i(m+1)}{M}jk -\frac{\pi i\epsilon'}{2}}  
\tilde{ch}^{N=2[M,m;\epsilon+\epsilon']}_{j,k;\epsilon'}(\tau,z)\,.
\end{equation*}

 \end{theorem}

 \begin{remark}
   \label{rem:9.3}
   \begin{list}{}{}
   \item (a)  Letting
    \, $\tilde{ch}^{N=2[M,m;\epsilon]}_{j,k;\epsilon'} (\tau,z,t) =
     e^{2\pi it c_{M,m}}\,
     \tilde{ch}^{N=2[M,m;\epsilon]}_{j,k;\epsilon'} (\tau,z)$, we
     can rewrite the transformation formula in
     Theorem~\ref{th:9.2} in a more suggestive form:
     \begin{equation*}
       \tilde{ch}^{N=2[M,m;\epsilon]}_{j,k;\epsilon'}\left( -
         \frac{1}{\tau}\,,\, \frac{z}{\tau}\,,\,
         t-\frac{z^2}{6\tau} \right)=\sum_{(a,b) \in
         \Omega^{(M)}_\epsilon} S^{[M,m,\epsilon,\epsilon']}_{(j,k)(a,b)}\,
          \tilde{ch}^{N=2[M,m;\epsilon']}_{a,b;\epsilon} (\tau,z,t)\,.
     \end{equation*}

\item(b)  If $m=0$, we have the well-known $N=2$ unitary discrete series
modules with
  central charge $3 \left( 1-\frac{2}{M} \right)$ and their
  well-known modular transformation properties (with tilde in 
Theorem \ref{th:9.2}
  removed),  see e.g. \cite{KW3}.
   \end{list}
 \end{remark}

\section{Modular transformation formulae for modified characters
  of admissible $N=4$ modules}
\label{sec:10}

In this section we study quantum Hamiltonian reduction of
admissible $\hat{A}_{1|1}$-modules of negative level
$K=-\frac{m}{M}$, where $m$ and $M$ are coprime positive
integers.  As we have seen in Section~\ref{sec:8}, $S$-invariance
of modified characters holds for arbitrary non-zero level
(Theorems~\ref{th:8.4} and \ref{th:8.7}), however the
$T=\left( 
\begin{smallmatrix}
1 & 1 \\ 0 & 1
\end{smallmatrix}
\right)
$-invariance fails
due to the fact that the operator~$D_1$ is not translation $(\tau
\to \tau +1)$ invariant).  Moreover, for $K>0$ translation
invariance still fails after the reduction, but, as we shall see, 
for $K<0$, after
the quantum Hamiltonian reduction the translation invariance gets
restored.

Consider an $\hat{A}_{1|1}$-module $L(\Lambda)$ of negative level
$K=-\frac{m}{M}$ with the highest weight
\begin{equation*}
  \Lambda = (K-2m_1 + m_2) \Lambda_0+m_1 (\Lambda_1 +
  \Lambda_3)+m_2 \Lambda_2 \, .
\end{equation*}
It follows from the remarks in the beginning of
Section~\ref{sec:9} that the module $H(\Lambda)$ over the $N=4$
superconformal algebra of Neveu--Schwarz type, obtained from
$L(\Lambda)$ by the quantum Hamiltonian reduction, is either $0$
(which happens off $K-2m_1+m_2 \in \ZZ_{\geq 0}$), or is irreducible.  In
the latter case the characteristic numbers of  $H(\Lambda)$ are
obtained from formulae (\ref{eq:9.1})--(\ref{eq:9.3}):
\begin{eqnarray}
  \label{eq:10.1}
     c_K &=& 6 \left( \frac{m}{M}-1\right) \, , \\[1ex]
   \label{eq:10.2}
     h_\Lambda &=& - \frac{Mm_1(m_1-m_2-1)}{m} -m_1+\frac{m_2}{2}\, , \\[1ex]
 \label{eq:10.3}
      s_\Lambda &=& m_2\, .
\end{eqnarray}
The same holds in the Ramond type case, also the central charge $c_K$
is the same, and the remaining two characteristic numbers are
obtained from (\ref{eq:9.5}) and (\ref{eq:9.6}):
\begin{eqnarray}
  \label{eq:10.4}
     h^{\tw}_\Lambda &=& - \frac{Mm_1 (m_1-m_2-1)}{m} 
         -m_1 - \frac{M-m}{4M} \, , \\[1ex]
 \label{eq:10.5}
    s^{\tw}_\Lambda &=& -m_2 - \frac{M-m}{M}\, .
\end{eqnarray}

In order to compute the characters and supercharacters of the
corresponding $N=4$ modules $H(\Lambda)$ and $H^{\tw} (\Lambda)$,
we use formulas (\ref{eq:9.4}) and (\ref{eq:9.7}) in the case $\fg =A_{1|1}$.

First, we  have from (\ref{eq:9.x}):
\begin{equation}
  \label{eq:10.6}
\fee (\tau, z) = \eta (\tau)^3 \frac{\vartheta_{11} (\tau,
  2z)}{\vartheta_{1-2\epsilon',1-2\epsilon} (\tau ,z)^2}\,.
\end{equation}
From (\ref{eq:10.6}) and Proposition~\ref{prop:A7}  (or from
Proposition~\ref{prop:9.2}) we deduce the following
transformation formulae for $N=4$ normalized denominators, where
we use the notation
\begin{equation*}
\fee (\tau,z,t) = e^{-2\pi i t} \fee (\tau, z)\, .
\end{equation*}

\begin{lemma}
  \label{lem:10.2}
  \begin{list}{}{}
  \item (a) $\fee \left( -\frac{1}{\tau}\,,\, \frac{z}{\tau}\,,\,
      t+\frac{z^2}{\tau}\right) =
    -(-1)^{(1-2\epsilon)(1-2\epsilon')}\tau \fee (\tau ,z, t)$.

\item (b)  $\fee (\tau +1,z,t) =e^{\pi i \epsilon'} \feee
  (\tau,z,t)$.

  \end{list}
\end{lemma}

Next, formulae (\ref{eq:9.4}) and  (\ref{eq:9.7}) for characters
and supercharacters of $N=4$ non-twisted and twisted modules are
respectively:
\begin{eqnarray}
  \label{eq:10.7}
 (\fpm \ch^\pm_{H(\Lambda)})(\tau,z,t)
   &=& (\hat{R}^{A_{1|1}\pm}\ch^\pm_\Lambda)
    \left(\tau,z+\frac{\tau}{2}\,,\,z-\frac{\tau}{2},t\right)\,;\\
\label{eq:10.8}
(\ftwpm\ch^\pm_{H^{\tw}(\Lambda)}) (\tau,z,t)
    &=& \left( \hat{R}^{A_{1|1}\tw,\pm} \ch^{\tw,\pm}_\Lambda \right)
         \left(\tau,z+\frac{\tau}{2}, z-\frac{\tau}{2},t\right)\,.
 \end{eqnarray}

As before, we change now notation by letting $\epsilon =\frac12$
(resp. $=0$) in the case of $+$ (resp. $-$), and $\epsilon'=\frac12$
 (resp. $=0$) in the non-twisted (resp. twisted) case.  For
 example,  $\ftwp = \fRh_0$.

Applying formulae (\ref{eq:10.7}) and (\ref{eq:10.8}) to $\Lambda
= \Lambda^{(s)}_{k_1,k_2}$ $(s=1,2,3,4)$ and using the
modification (\ref{eq:8.27}) of the numerators of the admissible
characters of $\hat{A}_{1|1}$, we obtain the following {\em
  partially modified} $N=4$ characters:
\begin{eqnarray*}
  \fReh \ch^{(\epsilon)'}_{H (\Lambda^{(s)}_{k_1,k_2})}(\tau,z,t)
     = D_1 \tilde{\Psi}^{[M,m-1;\epsilon]}_{j,-k;0} 
         \left(\tau, z+\frac{\tau}{2},z-\frac{\tau}{2},-t   \right)\,;\\[1ex]
  \left( \fRez \ch^{(\epsilon)'}_{H^{\tw}(\Lambda^{(s)}_{k_1,k_2})}\right)
     (\tau,z,t)
    = D_1 \tilde{\Psi}^{[M,m-1;\epsilon]}_{k+\frac12, -j+\frac12 ;\frac12}
      (\tau,z+\frac{\tau}{2}, z-\frac{\tau}{2}, -t)\, .
\end{eqnarray*}

The {\em modified} $N=4$ characters are obtained from the
partially modified ones by the shifts $j \mapsto j+\frac12\,,\,
k\mapsto k+\frac12$, and the corresponding shifts $z \mapsto z\pm
\frac{\tau}{2}$.  Namely, we let
\begin{eqnarray}
  \label{eq:10.9}
  \left( \fReh \tilde{\ch}^{(\epsilon)}_{H \left(
        \Lambda^{(s)}_{k_1,k_2}\right)}\right)(\tau,z,t)
  &=&  G^{[M,m;\epsilon]}_{j+\frac{1}{2},-k-\frac{1}{2};\frac12} (\tau,z,t)\,;\\[1ex]
\label{eq:10.10}
\left( \fRez
  \tilde{\ch}^{(\epsilon)}_{H^{\tw}{\left(\Lambda^{(s)}_{k_1,k_2}
      \right)}}\right) (\tau,z,t)
  &=& G^{[M,m;\epsilon]}_{k+1,-j;0} (\tau,z,t)\,, \\
\noalign{\hbox{where}}\nonumber \\
\label{eq:10.11}
  G^{[M,m;\epsilon]}_{j,k;\epsilon'} (\tau,z,t)
  &=& D_1 \tilde{\Psi}^{[M,m-1;\epsilon]}_{j,k;\epsilon'} (\tau,z,z,-t)\,.
\end{eqnarray}

\begin{remark}
  \label{rem:10.3}
  \begin{list}{}{}
  \item (a) One can replace $D_1$ by $D_0$ in (\ref{eq:10.11}).

\item (b)  We have:
  \begin{eqnarray*}
    \left( D_1 \Psi^{[M,m;\epsilon]}_{j,-k;\epsilon'} \right)
    \left( \tau, z+ \frac{\tau}{2}, z-\frac{\tau}{2},-t \right)
    = q^{\frac{m+1}{4M}} \left( D_1
     \Psi^{[M,m;\epsilon]}_{j+\frac12, -k-\frac12;\frac12- \epsilon'} 
       \right) (\tau,z,z,-t)\,,
\end{eqnarray*}
  \end{list}
\noindent{and the same identity holds for $\tilde{\Psi}$.  
Hence passing from
the partial modification to the modification of the $N=4$
characters amounts to removing the factor $q^{\frac{m+1}{4M}}$.}

\end{remark}

Since the action of $SL_2 (\ZZ)$ commutes with $D_1$, we deduce from
Theorem~\ref{th:7.1} modular transformation formulae
for the numerators of the modified $N=4$ characters:
\begin{eqnarray}
  \label{eq:10.12}
  G^{[M,m;\epsilon]}_{j,k;\epsilon'}\left(-\frac{1}{\tau},
    \frac{z}{\tau}, t+\frac{z^2}{\tau}   \right)
  &=& \frac{\tau^2}{M} \sum_{a,b \in \epsilon +\ZZ/M\ZZ}
      e^{-\frac{2\pi i m}{M}(ak+bj)}G^{[M,m;\epsilon']}_{a,b;\epsilon}
        (\tau,z,t)\,,\\[1ex]
  \label{eq:10.13}
   G^{[M,m;\epsilon]}_{j,k;\epsilon'} (\tau +1,z,t)
     & =& e^{\frac{2\pi i mjk}{M}} 
     G^{[M,m;|\epsilon - \epsilon'|]}_{j,k;\epsilon'}  (\tau,z,t)\,.          
\end{eqnarray}

We rewrite formula (\ref{eq:10.12}), using that
$G^{[M,m;\epsilon]}_{k,j;\epsilon'}$ changes sign if we permute
$k$ and $j$.  Permuting $a$ and $b$ in the RHS of
(\ref{eq:10.12}) and adding the obtained equation to
(\ref{eq:10.12}) we obtain, after some simple manipulations:
\begin{eqnarray}
  \label{eq:10.14}
\lefteqn{\hspace{-4.65in}  G^{[M,m;\epsilon]}_{j,-k;\epsilon'} 
  \left(-\frac{1}{\tau}, \frac{z}{\tau},
    t+ \frac{z^2}{\tau}\right)}\\
 = \frac{i\tau^2}{M} \sum_{a,b \in \epsilon +\ZZ/M\ZZ} 
    e^{-\frac{\pi i m}{M} (a-b)(j-k)} \sin \frac{\pi m}{M}(a+b)(j+k)
      G^{[m,M;\epsilon']}_{a,-b;\epsilon} (\tau,z,t)\,.\nonumber
\end{eqnarray}

We have used above the parametrization of the admissible weights
$\Lambda^{(s)}_{k_1,k_2}$ by the pairs $(j,k)$, given by
(\ref{eq:8.9}), (\ref{eq:8.11}), (\ref{eq:8.13}) and
(\ref{eq:8.15}).  In order to write down a unified formula for
modified $N=4$ characters, it is convenient to introduce new parameters:
\begin{eqnarray}
  \label{eq:10.15}
    \tilde{j} &=& j+\frac12\,,\, \tilde{k}=k+\frac12 
    \hbox{\,\,  in the non-twisted case,}\\
    \tilde{j} &=& k+1\,,\, \tilde{k}=j  \hbox{\,\, in the twisted case.}
      \nonumber
\end{eqnarray}
Then (\ref{eq:10.9}) and (\ref{eq:10.10}) can be written in a
unified way:
\begin{equation}
  \label{eq:10.16}
  \left(\fee \tilde{\ch}^{N=4[M,m;\epsilon]}_{\tilde{j},\tilde{k};\epsilon'}  \right) (\tau,z,t) = 
G^{[M,m;\epsilon]}_{\tilde{j},    -\tilde{k};\epsilon'}
    (\tau,z,t)\,.
\end{equation}

We describe below the precise parametrization of these modified
$N=4$ characters.  First, the following lemma is obtained by a
direct calculation from (\ref{eq:10.2})--(\ref{eq:10.5}).

\begin{lemma}
  \label{lem:10.4}

Let $\Lambda = \Lambda^{(s)}_{k_1,k_2}$ be an admissible weight
of level $K=-\frac{m}{M}$ for $\hat{A}_{1|1}$.  Assume that $H
(\Lambda) \neq 0$ 
, so that the $N=4$ superconformal algebra
modules $H(\Lambda)$ and $H^{\tw}(\Lambda)$ are irreducible.  Then
their lowest energies and spins are described by the following
formulae in the parameters $(\tilde{j},\tilde{k})$ given by
(\ref{eq:10.15}):
\begin{eqnarray*}
  s=1 : h_\Lambda &=& K\tilde{j}\tilde{k}+\tilde{j} -\frac{K+2}{4}\,,\,
             s_\Lambda = K (\tilde{k}-\tilde{j})\,,\\
        h^{\tw}_\Lambda &=& K\tilde{j}\tilde{k}+\tilde{k} -\frac{K+1}{4}\,,\,\,
             s^{\tw}_\Lambda =K (\tilde{k}-\tilde{j})-1\,;\,\\
   s=2 : h_\Lambda &=& K(M-\tilde{j})(M-\tilde{k})+M-\tilde{j}
               -\frac{K+2}{4}\,,\,
              s_{\Lambda} =K (\tilde{k}-\tilde{j})-2\,,\,\\
         h^{\tw}_\Lambda  &=& K(M-\tilde{j}) (M-\tilde{k}) +M-\tilde{k}
                 -\frac{k+1}{4}\,,\,\,
             s^{\tw}_\Lambda = K (\tilde{k}-\tilde{j})+1 \, ; \\
  s=3 :  h_\Lambda &=& K\tilde{j}\tilde{k}+\tilde{k} -\frac{K+2}{4}\,,\,
             s_\Lambda = K (\tilde{k}-\tilde{j})-2\,,\,\\
         h^{\tw}_\Lambda &=& K\tilde{j}\tilde{k}+\tilde{j} 
              -\frac{K+1}{4}  \,,\, s^{\tw}_\Lambda 
                  = K (\tilde{k}-\tilde{j}) +1\, ;\, \\ 
  s=4 : h_\Lambda &=& K (M-\tilde{j})(M-\tilde{k})+M-\tilde{k}
               -\frac{K+2}{4}\,,\, s_\Lambda = K (\tilde{k}-\tilde{j})\,,\\
        h^{\tw}_\Lambda &=& K (M-\tilde{j}) (M-\tilde{k})+M-\tilde{j}
               -\frac{K+1}{4}\,,\, 
             s^{\tw}_\Lambda = K(\tilde{k}- \tilde{j})-1\,.
\end{eqnarray*}

\end{lemma}

Since the positive energy irreducible modules over the $N=4$ superconformal
algebras are determined by their characteristic numbers, we obtain
the following corollary.

\begin{corollary}
  \label{cor:10.5}
  We have the following isomorphisms of modules over the $N=4$
  superconformal algebra of Neveu--Schwarz type:
  \begin{equation*}
  H (\Lambda^{(1)}_{k_1,k_2})\simeq H (\Lambda^{(4)}_{k_1 +1 ,k_2})\,,\,
    H (\Lambda^{(3)}_{k_1,k_2})\simeq H(\Lambda^{(2)}_{k_1+1,k_2})\,,
  \end{equation*}
and similar isomorphims for the $N=4$ superconformal algebra of
Ramond type, replacing $H$ by $H^{\tw}$\,.
\end{corollary}




Now we turn to modular transformation formulae for the
modified $N=4$ characters, defined by (\ref{eq:10.16}).  Using
Lemma~\ref{lem:10.2}(a) and (\ref{eq:10.14}), we obtain:
\begin{eqnarray}
   \label{eq:10.17}
 \lefteqn{\hspace{-5.25in} \tilde{\ch}^{N=4 [M,m;\epsilon]}_{\tilde{j},\tilde{k};\epsilon'}
    \left(-\frac{1}{\tau}, \frac{z}{\tau}\,, t+\frac{z^2}{\tau} \right)}\\
    = -(-1)^{(1-2\epsilon)(1-2\epsilon')}\frac{i\tau}{M}
      \sum_{\tilde{a},\tilde{b}\in \epsilon +\ZZ/M\ZZ}
        e^{-\frac{\pi i m}{M} (\tilde{a}-\tilde{b})(\tilde{j}-\tilde{k})}
          \sin \frac{\pi m}{M} (\tilde{a}+\tilde{b}) (\tilde{j}+\tilde{k})
        \,\,   \tilde{\ch}^{N=4[M,m;\epsilon']}_{\tilde{a},\tilde{b};\epsilon}
             (\tau,z,t)\,.\nonumber
\end{eqnarray}
%
%
In order to obtain the final modular transformation formulae we need the
following remarks.

\begin{remark}   
\label{rem:10.6}
  \begin{list}{}{}
  \item (a)
    $\tilde{\ch}^{N=4[M,m;0]}_{\tilde{j},\tilde{k};\epsilon'}$
    and $e^{\frac{\pi im (\tilde{j}-\tilde{k})}{M}}
\tilde{\ch}^{N=4[M,m;\tfrac12]}_{\tilde{j},\tilde{k};\epsilon'} $  remain
    unchanged if we add to $\tilde{j}$ and to $\tilde{k}$
    some integer multiples of $M$.

 \item (b)  $\tilde{\ch}^{N=4 [M,m;\epsilon]}_{\tilde{j},\tilde{k};\epsilon'} 
 = - \tilde{\ch}^{N=4  [M,m;\epsilon]}_{-\tilde{k},-\tilde{j};\epsilon'}$.

 \item (c) $\tilde{\ch}^{N=4 [M,m;\epsilon]}_{\tilde{j},\tilde{k};\epsilon'}
  =0$ if $\tilde{j}+\tilde{k}\in M\ZZ$.

\item (d)  The $(\tilde{a},\tilde{b})$ coefficient and the
   $(M-\tilde{b},M-\tilde{a})$ coefficient in (\ref{eq:10.17}) are equal.

\item (e) $\tilde{\ch}^{N=4 [M,m;\epsilon]}_{0,\tilde{k};0}=-
e^{2\pi im\epsilon} 
\tilde{\ch}^{N=4 [M,m;\epsilon]}_{M-\tilde{k},0;0}$ hence 
$\tilde{\ch}^{N=4 [M,m;\epsilon]}_{0,\tilde{k};0}$  
for   $0<\tilde{k} <M$ can be replaced in (\ref{eq:10.17}) by
  $\tilde{\ch}^{N=4 [M,m;\epsilon]}_{\tilde{j},0;0}$, $0<\tilde{j}<M$.
   \end{list}
Claims (a) and (b) follow, by (\ref{eq:10.11}), from the corresponding
properties of the functions  $\tilde{\Psi}^{[M,m-1;\epsilon]}_{j,k;\epsilon'} 
(\tau,z,z,t)$. For (c), letting 
$\tilde{k}=-\tilde{j}$ 
in (b), we obtain that 
$\tilde{\ch}^{N=4[M,m;\epsilon]}_{\tilde{j},-\tilde{j};\epsilon'}=0$. 
Hence, if $\tilde{k}+\tilde{j}=nM$ with $n\in \ZZ$, we have, by (a):
$\tilde{\ch}^{N=4[M,m;\epsilon]}_{\tilde{j},\tilde{k};\epsilon'}=
\tilde{\ch}^{N=4[M,m;\epsilon]}_{\tilde{j},nM-\tilde{j};\epsilon'}= 
\tilde{\ch}^{N=4[M,m;\epsilon]}_{\tilde{j},-\tilde{j};\epsilon'}=0$.
The proof of (d) is straightforward. Claim (e) is obtained by letting 
$\tilde{j}=0$ in (b) and using (a).  
\end{remark}

Letting, as before, $\epsilon=\frac12$ in the Neveu--Schwarz case and $\epsilon =0$ in the Ramond case, introduce the following subsets in the $\tilde{j},\tilde{k}$-plane:
\begin{equation*}   
\Omega^{N=4(M)}_\epsilon = \{ (\tilde{j},\tilde{k}) \in   
\epsilon + \ZZ\, |\,\,\, 0<\tilde{j},\tilde{j}+\tilde{k}<M\,,\, 0\leq 
\tilde{k}<M\} \,.
\end{equation*}
It follows from Remark \ref{rem:10.6} and Corollary~\ref{cor:10.5} that these 
subsets parametrize the non-zero modified characters of irreducible modules 
over the Neveu--Schwarz type and the Ramond type $N=4$ superconformal 
algebras, obtained 
by the quantum Hamiltonian reduction from all the admissible 
$\hat{A}_{1|1}$-modules $L(\Lambda^{(s)}_{k_1,k_2})$ and 
$L^{\tw}(\Lambda^{(s)}_{k_1,k_2})$. 
 As a result, we can rewrite (\ref{eq:10.17}) as in the following
theorem. The $T$-transformation formula in this theorem
follows from Lemma~\ref{lem:10.2}(b) and (\ref{eq:10.13}). 

\begin{theorem}
  \label{th:10.7}
Let $M$ and $m$ be positive coprime integers, such that $M\geq 2$
and $\gcd(M,2m)=1$ if $m>1$, and let
$\epsilon,\epsilon'=0$ or $\frac12$. Let
$\tilde{\ch}^{N=4[M,m;\epsilon]}_{\tilde{j},\tilde{k};\epsilon'}
(\tau,z,t)$, $(\tilde{j},\tilde{k} )\in
\Omega^{N=4(M)}_{\epsilon'}$, be the modified characters and
supercharacters  of modules over the $N=4$ Neveu--Schwarz
(resp. Ramond) type superconformal algebras if $\epsilon=\frac12$
(resp. $\epsilon =0$), obtained by the quantum Hamiltonian
reduction from level $K=-\frac{m}{M}$ admissible $\hat{A}_{1|1}$-
modules (resp. twisted modules).  Then
\begin{eqnarray*}
\lefteqn{\hspace{-4.5in}  
\tilde{\ch}^{N=4[M,m;\epsilon]}_{\tilde{j},\tilde{k};\epsilon'}
   \left(  -\frac{1}{\tau}, \frac{z}{\tau}, t+\frac{z^2}{\tau}\right)
  =-(-1)^{(1-2\epsilon)(1-2\epsilon')}\frac{2 i\tau}{M}} \\[1ex]
     \times\sum_{(\tilde{a},\tilde{b}) \in\Omega^{N=4(M)}_\epsilon}
       e^{-\frac{\pi  im}{M}(\tilde{a}-\tilde{b})(\tilde{j}-\tilde{k})}
       \sin \frac{\pi m}{M} (\tilde{a}+\tilde{b})(\tilde{j}+\tilde{k})
       \,\, \tilde{\ch}^{N=4 [M,m;\epsilon']}_{\tilde{a},\tilde{b};\epsilon}
          (\tau,z,t) \,; \\[1ex]
 \tilde{\ch}^{N=4[M,m;\epsilon]}_{\tilde{j},\tilde{k};\epsilon'}
     (\tau+1,z,t) 
       =e^{-\frac{2\pi im \tilde{j}\tilde{k}}{M}- \pi i \epsilon'}
 \tilde{\ch}^{N=4 [M,m;|\epsilon-\epsilon'|]}_{\tilde{j},\tilde{k};\epsilon'}
      (\tau,z,t)\,.
\end{eqnarray*}

\end{theorem}

\begin{remark} 
\label{rem:10.7} 
Note that $H(-m\Lambda_0)$, where $m\in \ZZ_{\geq 1}$, 
is a non-zero $N=4$ module
with $M=1$, whose modified character is zero. 
For example, the numerator of $\ch^-_{-\Lambda_0}$ is equal to 
$D_1( $normalization factor$\times 
\Phi^{[0]})$, which gives $0$ when we apply the quantum Hamiltonian reduction.
(It is easy to show that such a situation may occur only if $s=1$ or $3$ and 
$\tilde{j}+\tilde{k} =M$.)  
 \end{remark}   

\begin{remark} 
\label{rem:10.8} 
If $m=1$, we
obtain a family of positive energy $N=4$ modules
with central charge $6(\frac{1}{M}-1)$, where $M$ is a positive integer.
It is easy to deduce from our calculations in Section \ref{sec:10}
the following formulae for their characters:
\begin{eqnarray*}
  \left(\fee \tilde{\ch}^{N=4[M,1;\epsilon]}_{\tilde{j},\tilde{k};\epsilon'} 
 \right) (\tau,z,t) =
-i(-1)^{2\epsilon}e^{-\frac{2\pi it}{M}}q^{-\frac{\tilde{j}\tilde{k}}{M}}e^{\frac{2\pi i}{M}(\tilde{j}-\tilde{k})z}\\
\times D_0\left(
\frac{\eta(M\tau)^3\vartheta_{11}(M\tau,z_1+z_2+(\tilde{j}-\tilde{k})\tau)}
{\vartheta_{11}(M\tau,z_1+\tilde{j}\tau +\epsilon)\vartheta_{11}(M\tau,z_2-\tilde{k}\tau +\epsilon)} \right)
|_{z_1=z_2=z}.
\end{eqnarray*}
In order to obtain the corresponding modified characters (hence a modular
invariant family), one has to add to the RHS the expression:
\begin{equation*}
i(-1)^{2\epsilon}\frac{\tilde{j}+\tilde{k}}{M}e^{-\frac{2\pi it}{M}}q^{-\frac{\tilde{j}\tilde{k}}{M}}e^{\frac{2\pi i}{M}(\tilde{j}-\tilde{k})z}
\frac{\eta(M\tau)^3\vartheta_{11}(M\tau,2z+(\tilde{j}-\tilde{k})\tau)}
{\vartheta_{11}(M\tau,z+\tilde{j}\tau +\epsilon)
\vartheta_{11}(M\tau,z-\tilde{k}\tau +\epsilon)}.
\end{equation*}
\end{remark} 

\appendix
\section{Appendix.~~A brief review of theta functions.}

\numberwithin{equation}{section}

In this appendix we review some basic facts about theta functions
(rather Jacobi forms), following the exposition in \cite{K2},
Chapter~13.

Let $L$ be a positive definite integral lattice of rank $\ell$
with a positive definite symmetric bilinear form $\bl$.  Let $\fh
= \CC \otimes_\ZZ L$ be the complexification of $L$ with the
bilinear form $\bl$, extended from $L$ by bilinearity.  Let
$\hat{\fh} = \fh \oplus \CC K \oplus \CC d$ be an $\ell
+2$-dimensional vector space over $\CC$ with the (non-degenerate)
symmetric bilinear form $\bl$, extended from $\fh$ by letting
$\fh \perp (\CC K + \CC d)$, $(K|K)=0$, $(d|d)=0$, $(K|d)=1$. We
shall identify $\hat{\fh}$ with $\hat{\fh}^*$, using this
bilinear form, so that any 
$h\in \hat{\fh}$ 
defines a linear function
$l_h$ on $\hat{\fh}$ via $l_h (h_1)=(h|h_1)$ .  

Let $X = \{ h \in \hat{\fh}|\Re (K|h)>0 \}$.
Define the following action of the additive group of the vector
space $\fh$ on $\hat{\fh}^*$ (cf. (\ref{eq:2.10})):
\begin{equation*}
  t_\alpha (h) = h+(K|h)\alpha - ((\alpha |h)+ \frac{(\alpha |
    \alpha)}{2}(K|h))K \, , \,\, \alpha \in \fh .
\end{equation*}
This action leaves the bilinear form $\bl$ on $\hat{\fh}$
invariant and fixes $K$, hence leaves the domain $X$ invariant.

A theta function (rather Jacobi form) of degree $k \in \ZZ_{\geq 0}$ is a
holomorphic function $F$ in the domain $X$, satisfying the
following four properties:

\begin{list}{}{}
\item (i)  $F (t_\alpha (h)) = F(h)$;

\item (ii)  $F (h+2\pi i \alpha) = F (h)$ for $\alpha \in L$;

\item (iii) $F (h+aK) = e^{ka} F (v)$ for all $a \in \CC$;

\item (iv)  $DF = 0$, where $D$ is the Laplace operator on $\hat{\fh}$ ,
  associated to $\bl$.

\end{list}

Denote by $Th_k$, $k \in \ZZ_{\geq 0}$, the vector space over $\CC$ of
all theta functions of degree $k$.

Let $P_k = \{ \lambda \in \hat{\fh} | (\lambda |K)=k$ and
$\bar{\lambda} \in L^*\}$, where $\bar{\lambda}$ stands for the
projection of $\lambda$ on $\fh$ and $L^* \subset \fh$ is the dual
lattice of the lattice $L$.
Given $\lambda \in P_k$, where $k$ is a positive integer, let
\begin{equation*}
  \Theta_\lambda = e^{-\frac{(\lambda|\lambda)}{2k}K}
  \sum_{\alpha \in L} e^{t_\alpha (\lambda)}\, .
\end{equation*}
This series converges to a holomorphic function in the domain
$X$, which is an example of a theta function of degree $k>0$
(properties (i)--(iii) are obvious, and property (iv) holds
since $D e^\lambda = (\lambda | \lambda) e^\lambda$).  Note that
\begin{equation*}
  \Theta_{ \lambda + k\alpha +aK } = \Theta_{\lambda} \hbox{ \,\, for
  \,\,} \alpha \in L, \,\, a \in \CC \, .
\end{equation*}

\begin{proposition}
  \label{prop:A1}
The set $\{ \Theta_\lambda |\lambda \in P_k \mod (kL+\CC K)\}$ is
a $\CC$-basis of $Th_k$ if $k>0$, and $Th_0=\CC$.
\end{proposition}

\begin{proof}
  See the proof of Proposition~13.3 and Lemma~13.2 in \cite{K2}.
\end{proof}

Introduce coordinates $(\tau,z,t)$ on $\hat{\fh}$ by
(\ref{eq:4.8}), so that $X = \{ (\tau,z,t)|\,\, \Im \tau >0 \}$ and
$q: = e^{2\pi i \tau} = e^{-K}$.  In these coordinates we have
the usual formula for a Jacobi form $\Theta_\lambda$, $\lambda \in P_k$,
of degree $k>0$:
\begin{equation}
  \label{eq:A1}
  \Theta_\lambda (\tau ,z,t)= e^{2\pi i kt} \sum_{\gamma \in
    L+\frac{\bar{\lambda}}{k}}q^{k\frac{(\gamma | \gamma)}{2}}
       e^{2\pi i k (\gamma |z)}\, .
\end{equation}

\begin{proposition}
  \label{prop:A2}
One has the following elliptic transformation formula of a Jacobi
form $\Theta_\lambda$ of degree $k$ for $\alpha \in L^*$:
\begin{equation*}
  \Theta_\lambda (\tau,z+\alpha \tau ,t) =q^{-\frac{k}{2} (\alpha|\alpha)}
    e^{-2\pi i k (\alpha |z)} \Theta_{\lambda +k\alpha} (\tau,z,t)\, .
\end{equation*}

\end{proposition}

\begin{proof} It is straightforward.
 
\end{proof}

Recall the action of the group $SL_2 (\RR)$ in the domain $X$
%
%
%
and the action of the corresponding metaplectic group on the
space of meromorphic functions on $X$, given in Section 4.
%
%
%

\begin{proposition}
  \label{prop:A3}
One has the following modular transformation formulae of a Jacobi
form $\Theta_\lambda$ of degree $k>0$:
\begin{list}{}{}
\item (a)  $\Theta_\lambda \left( -\frac{1}{\tau}\,,\,\frac{z}{\tau}\,,\,
    t -\frac{(z|z)}{2\tau} \right) = (-i\tau)^\frac{\ell}{2}|L^*/kL|^{-\frac12}
 \displaystyle{\sum_{\mu \in P_k\!\!\! \mod \!(kL +\CC K)}}
e^{-\frac{2\pi i}{k}(\bar{\lambda} |\bar{\mu})} \Theta_\mu (\tau,z,t)$.

\item (b) $\Theta_\lambda (\tau +1,z,t) = e^{\frac{\pi
      i(\lambda|\lambda)}{k}}  \Theta_\lambda (\tau,z,t)$, 

provided that $k(\alpha|\alpha)\in 2\ZZ$ for 
$\alpha \in L$ (in particular, 
provided that the lattice $L$ is even).

 \item (c)  The space $Th_k$ is invariant with respect to the
   (right) action of the group $SL_2 (\ZZ)$, provided that
   $k(\alpha|\alpha) \in 2 \ZZ$ for all $\alpha \in L$.

\end{list}

\end{proposition}

\begin{proof}  The proof of (a) is based on the formula
(see (\ref{eq:4.6}) and (\ref{eq:4.6a}) for notation):
\begin{equation}
\label{eq:A7}
(DF)|_A=(c\tau+d)^2 D(F|_A),  
\hbox{\,\, where\,\, } A=\left(
\begin{array}[]{cc}
  a & b \\c & d 
\end{array}
\right)\, \in SL_2(\RR).
\end{equation}
The rest is straightforward.  See \cite{K2}, Theorem~13.5.
\end{proof}

\begin{remark}
  \label{rem:A4}
Note that $P_k = \{ kd + \bar{\lambda} +aK |\,\, \bar{\lambda} \in
L^*$, $a \in \CC \}$.  Hence we may use a slightly different
notation:
\begin{equation*}
  \Theta_\lambda = \Theta_{\bar{\lambda},k}\, , 
\end{equation*}
%
%
%
so that the basis $\{\Theta_\lambda|\, \lambda \in P_k \!\!\! \mod kL+\CC K \}$
of $Th_k$ is identified with the basis
 $ \{ \Theta_{\bar{\lambda},k}|\,\bar{\lambda} \in L^*/kL \}$.
%

\end{remark}

\begin{example}
  \label{ex:A5}

Let $L=\ZZ$ with the bilinear form $(a|b) = 2ab$, so that $L^*
=\tfrac12 \ZZ$ .  Then for a positive integer $k$ we have the
following basis of $Th_k$ \,\,$(\tau,z,t\in \CC\, , \, \Im \tau >0)$:
\begin{equation}
  \label{eq:A2}
  \Theta_{j,k}(\tau ,z ,t) =e^{2\pi i kt} \sum_{n \in \ZZ
    +\frac{j}{2k}}q^{kn^2} e^{2\pi i knz}\, , \quad j \in \ZZ/2k\ZZ\,.
\end{equation}
The  elliptic transformation formula is as follows $(n \in \ZZ)$:
\begin{equation}
  \label{eq:A3}
  \Theta_{j,k} (\tau,z+n \tau,t) = q^{-\frac{kn^2}{4}}e^{-\pi i
    knz} \Theta_{j+kn,k}(\tau,z,t) \, .
\end{equation}
The modular transformation formulae are:
\begin{equation}
  \label{eq:A4}
  \Theta_{j,k} \left( -\frac{1}{\tau}\,,\, \frac{z}{\tau}\,,\,    t-\frac{z^2}{2\tau} \right)
     = \left( \frac{-i\tau}{2k}\right)^{\tfrac12} \sum_{j'\in \ZZ   /2k\ZZ} 
         e^{-\frac{\pi i jj'}{k}} \Theta_{j',k} (\tau ,z,t)\, , 
\end{equation}
\begin{equation}
  \label{eq:A5}
  \Theta_{j,k }(\tau +1,z,t) =e^{\frac{\pi i j^2}{2k}}
    \Theta_{{j,k}}(\tau,z,t).
\end{equation}
\end{example}

Especially important are the celebrated four Jacobi theta
functions of degree two (we put $t=0$ here) \cite{M}:
\begin{eqnarray*}
  \vartheta_{00} = \Theta_{2,2} + \Theta_{0,2} \, , \, 
     \vartheta_{01} =-\Theta_{2,2} + \Theta_{0,2}\, , \,
  \vartheta_{1 0} = \Theta_{1,2} + \Theta_{-1,2} \, , \, 
      \vartheta_{11} = i \Theta_{1,2} -i\Theta_{-1,2}\, .
\end{eqnarray*}
Due to the Jacobi triple product identity, the following infinite
products can be expressed in terms of the four Jacobi theta functions
and the $\eta$-function $\eta(\tau) = q^{\frac{1}{24}}
\prod^\infty_{n=1} (1-q^n)$: 
\begin{eqnarray}
  \label{eq:4.7}
  \prod^\infty_{n=1}  (1+e^{2\pi i z} q^{n-\frac12}) 
     (1+e^{-2\pi i z} q^{n-\frac12})= q^{\frac{1}{24}} 
        \frac{\vartheta_{00} (\tau ,z)}{\eta (\tau)}\nonumber,\\[1ex]
 \prod^\infty_{n=1}   (1-e^{2\pi i z} q^{n-\frac12})
     (1-e^{-2\pi i z} q^{n-\frac12})= q^{\frac{1}{24}} 
        \frac{\vartheta_{01} (\tau ,z)}{\eta (\tau)}\nonumber, \\[1ex]
\prod^\infty_{n=1}  (1+e^{-2\pi i z} q^n)
     (1+e^{2\pi i z} q^{n-1})= q^{-\frac{1}{12}} 
        e^{\pi i z}\frac{\vartheta_{10} (\tau ,z)}{\eta (\tau)},\\[1ex]
 \prod^\infty_{n=1} (1-e^{-2\pi i z} q^n)
     (1-e^{2\pi i z} q^{n-1})= iq^{-\frac{1}{12}} 
        e^{\pi i z}\frac{\vartheta_{11} (\tau ,z)}{\eta (\tau)}\, .
\nonumber
\end{eqnarray}


\begin{proposition}\cite{M}
  \label{prop:A6}  For $a,b =0$ or $1$ and $n \in \ZZ$ one has:
\begin{equation*}
  \vartheta_{ab} (\tau,z+n\tau) = (-1)^{bn} q^{-\frac{n^2}{2}}
  e^{-2\pi i nz} \vartheta_{ab} (\tau ,z)\, .
\end{equation*}

\end{proposition}

\begin{proof}
  It follows from (\ref{eq:A3}).
\end{proof}

\begin{proposition} (\cite{M}, p.~36)
  \label{prop:A7} For $a,b =0$ or $1$ one has:
%
  \begin{eqnarray*}
    \vartheta_{ab} \left( - \frac{1}{\tau} , \frac{z}{\tau} \right)
        &=& (-i)^{ab} (-i\tau)^{\frac12} e^{\frac{\pi i
            z^2}{\tau}} \vartheta_{ba} (\tau ,z )\, ;\\
    \vartheta_{0a} (\tau +1,z) &=& \vartheta_{0b} (\tau,z)\, , \, {\,
      \hbox{where\,\,} } a \neq b \, ;\\
     \vartheta_{1a} (\tau +1,z) &=& e^{\frac{\pi i}{4}}
     \vartheta_{1a}(\tau ,z)\, .
  \end{eqnarray*}
\end{proposition}

\begin{proof}
  It follows from (\ref{eq:A4}), (\ref{eq:A5}).

\end{proof}




\end{document}